\documentclass[reqno,a4paper]{amsart}

\usepackage{amsthm}
\usepackage{amsmath}
\usepackage[all]{xy} \SelectTips{eu}{}
\usepackage{latexsym,amsfonts,amssymb}
\usepackage{hyperref}
\usepackage{mathptmx}
\usepackage{nicefrac}
\usepackage{array}
\usepackage{xcolor}
\usepackage{mathrsfs}
\usepackage{rotating}

\renewcommand{\theequation}{$\smash{\sharp}\mspace{0.5mu}$\arabic{equation}}

\newcommand{\numberseries}{\mdseries}   

\newlength{\thmtopspace}                
\newlength{\thmbotspace}                
\newlength{\thmheadspace}               
\newlength{\thmindent}                  

\newtheoremstyle{bfupright head,slanted body}
                {\thmtopspace}{\thmbotspace}
                {\slshape}{\thmindent}{\bfseries}{.}{\thmheadspace}
                {{\numberseries \thmnumber{{\bf #2} }}\thmnote{#3}}

\newtheoremstyle{bfupright head,upright body}
                {\thmtopspace}{\thmbotspace}
                {\upshape}{\thmindent}{\bfseries}{.}{\thmheadspace}
                {{\numberseries \thmnumber{{\bf #2} }}\thmnote{#3}}

\newtheoremstyle{bfit head,upright body}
                {\thmtopspace}{\thmbotspace}
                {\upshape}{\thmindent}{\upshape}{.}{\thmheadspace}
                {{\numberseries\thmnumber{{\bf #2} }}
                {\bfseries\itshape\thmnote{\negthickspace#3}}}

\newtheoremstyle{it head,upright body}
                {\thmtopspace}{\thmbotspace}
                {\upshape}{\thmindent}{\upshape}{.}{\thmheadspace}
                {{\numberseries\thmnumber{{\bf #2} }}
                {\itshape\thmnote{\negthickspace#3}}}

\newtheoremstyle{fixed bf head,slanted body}
                {\thmtopspace}{\thmbotspace}{\slshape}
                {\thmindent}{\bfseries}{.}{\thmheadspace}
                {{\numberseries \thmnumber{{\bf #2} }}\thmname{#1}\thmnote{ (#3)}}

\newtheoremstyle{fixed bf head,upright body}
                {\thmtopspace}{\thmbotspace}{\upshape}
                {\thmindent}{\bfseries}{.}{\thmheadspace}
                {{\numberseries \thmnumber{{\bf #2} }}\thmname{#1}\thmnote{ (#3)}}

\newtheoremstyle{independent paragraph}
                {\thmtopspace}{\thmbotspace}
                {\upshape}{\thmindent}{\upshape}{}{0pt}
                {\thmnote{#3 }}

\newtheoremstyle{subparagraph}
                {\thmbotspace}{\thmbotspace}
                {\upshape}{\thmindent}{\upshape}{}{0pt}
                {\thmnote{#3 }}

\newtheoremstyle{notes}
                {\thmtopspace}{\thmbotspace}
                {\ttfamily}{\thmindent}{\ttfamily\small }{}{0pt}
                {\thmnote{#3 }}
                
\newtheoremstyle{numbered paragraph}
                {\thmtopspace}{\thmbotspace}{\upshape}
                {\thmindent}{\upshape}{}{\thmheadspace}
                {{\numberseries \thmnumber{\bf #2.}}}

\theoremstyle{bfupright head,slanted body}
\newtheorem{res}{}[section]             \newtheorem*{res*}{}

\theoremstyle{bfit head,upright body}
                 \newtheorem*{com*}{}

\theoremstyle{bfupright head,upright body}
\newtheorem{bfhpg}[res]{}               \newtheorem*{bfhpg*}{}

\theoremstyle{it head,upright body}
               \newtheorem*{ithpg*}{}

\theoremstyle{fixed bf head,slanted body}
\newtheorem{thm}[res]{Theorem}          \newtheorem*{thm*}{Theorem}
\newtheorem{prp}[res]{Proposition}      \newtheorem*{prp*}{Proposition}
\newtheorem{cor}[res]{Corollary}        \newtheorem*{cor*}{Corollary}
\newtheorem{lem}[res]{Lemma}            \newtheorem*{lem*}{Lemma}

\theoremstyle{fixed bf head,upright body}
\newtheorem{dfn}[res]{Definition}       \newtheorem*{dfn*}{Definition}
      \newtheorem*{obs*}{Observation}
\newtheorem{rmk}[res]{Remark}           \newtheorem*{rmk*}{Remark}
\newtheorem{exa}[res]{Example}          \newtheorem*{exa*}{Example}
         \newtheorem*{exe*}{Exercise}
\newtheorem{stp}[res]{Setup}            \newtheorem{stp*}{Setup}
         \newtheorem{ntn*}{Notation}
\newtheorem{con}[res]{Construction}     \newtheorem{con*}{Construction}

\theoremstyle{numbered paragraph}
\newtheorem{ipg}[res]{}

\theoremstyle{subparagraph}

\theoremstyle{notes}

\newlength{\thmlistleft}        
\newlength{\thmlistright}       
\newlength{\thmlistpartopsep}   
\newlength{\thmlisttopsep}      
\newlength{\thmlistparsep}      
\newlength{\thmlistitemsep}     

\newcounter{eqc} 
\newenvironment{eqc}{\begin{list}{\upshape (\textit{\roman{eqc}})}%
    {\usecounter{eqc}%
      \setlength{\leftmargin}{\thmlistleft}%
      \setlength{\labelwidth}{\thmlistleft}%
      \setlength{\rightmargin}{\thmlistright}%
      \setlength{\partopsep}{\thmlistpartopsep}%
      \setlength{\topsep}{\thmlisttopsep}%
      \setlength{\parsep}{\thmlistparsep}%
      \setlength{\itemsep}{\thmlistitemsep}}}%
  {\end{list}}%

\newcommand{\eqclbl}[1]{{\upshape(\textit{#1})}}

\newcounter{prt}
\newenvironment{prt}{\begin{list}{\upshape (\alph{prt})}%
    {\usecounter{prt}%
      \setlength{\leftmargin}{\thmlistleft}%
      \setlength{\labelwidth}{\thmlistleft}%
      \setlength{\rightmargin}{\thmlistright}%
      \setlength{\partopsep}{\thmlistpartopsep}%
      \setlength{\topsep}{\thmlisttopsep}%
      \setlength{\parsep}{\thmlistparsep}%
      \setlength{\itemsep}{\thmlistitemsep}}}%
  {\end{list}}%

\newcommand{\prtlbl}[1]{{\upshape(#1)}}

\newcounter{rqm}
\newenvironment{rqm}{\begin{list}{\upshape (\arabic{rqm})}%
    {\usecounter{rqm}%
      \setlength{\leftmargin}{\thmlistleft}%
      \setlength{\labelwidth}{\thmlistleft}%
      \setlength{\rightmargin}{\thmlistright}%
      \setlength{\partopsep}{\thmlistpartopsep}%
      \setlength{\topsep}{\thmlisttopsep}%
      \setlength{\parsep}{\thmlistparsep}%
      \setlength{\itemsep}{\thmlistitemsep}}}%
  {\end{list}}%

\newcommand{\rqmlbl}[1]{{\upshape(#1)}}

\newcounter{rqmm}
\newenvironment{rqmm}{\begin{list}{\upshape (\arabic{rqmm}$^*$)}%
    {\usecounter{rqmm}%
      \setlength{\leftmargin}{\thmlistleft}%
      \setlength{\labelwidth}{\thmlistleft}%
      \setlength{\rightmargin}{\thmlistright}%
      \setlength{\partopsep}{\thmlistpartopsep}%
      \setlength{\topsep}{\thmlisttopsep}%
      \setlength{\parsep}{\thmlistparsep}%
      \setlength{\itemsep}{\thmlistitemsep}}}%
  {\end{list}}%

  {\end{list}}%

\newenvironment{itemlist}{\nopagebreak \begin{list}{{\small $\bullet$}}%
    {\setlength{\leftmargin}{\thmlistleft}%
      \setlength{\labelwidth}{\thmlistleft}%
      \setlength{\rightmargin}{\thmlistright}%
      \setlength{\partopsep}{\thmlistpartopsep}%
      \setlength{\topsep}{\thmlisttopsep}%
      \setlength{\parsep}{\thmlistparsep}%
      \setlength{\itemsep}{\thmlistitemsep}}}%
  {\end{list}}%

  \newcommand{\proofoftag}[2][:]{(#2)#1}

  \newcommand{\proofofimp}[3][:]{\mbox{\eqclbl{#2}$\!\implies\!$\eqclbl{#3}#1}}  

\newcommand{\pgref}[1]{\ref{#1}}

\renewcommand{\eqref}[1]{(\pgref{eq:#1})}
\newcommand{\eqmref}[1]{(\pgref{eq:#1}$^*$)}

\newcommand{\corref}[2][Corollary ]{#1\pgref{cor:#2}}
\newcommand{\dfnref}[2][Definition~]{#1\pgref{dfn:#2}}
\newcommand{\exaref}[2][Example ]{#1\pgref{exa:#2}}
\newcommand{\lemref}[2][Lemma ]{#1\pgref{lem:#2}}

\newcommand{\prpref}[2][Proposition ]{#1\pgref{prp:#2}}

\newcommand{\rmkref}[2][Remark ]{#1\pgref{rmk:#2}}
\newcommand{\thmref}[2][Theorem ]{#1\pgref{thm:#2}}
\newcommand{\stpref}[2][Setup~]{#1\pgref{stp:#2}}
\newcommand{\secref}[2][Section ]{#1\pgref{sec:#2}}
\newcommand{\appref}[2][Appendix~]{#1\ref{app:#2}}

\newcommand{\conref}[2][Construction~]{#1\ref{con:#2}}

\makeatletter
\def\@nobreak@#1{\mathchoice%
  {\nobreakdef@\displaystyle\f@size{#1}}%
  {\nobreakdef@\nobreakstyle\tf@size{\firstchoice@false #1}}%
  {\nobreakdef@\nobreakstyle\sf@size{\firstchoice@false #1}}%
  {\nobreakdef@\nobreakstyle\ssf@size{\firstchoice@false #1}}%
  \check@mathfonts}%
\def\nobreakdef@#1#2#3{\hbox{{%
                    \everymath{#1}%
                    \let\f@size#2\selectfont%
                    #3}}}%
\makeatother

\DeclareFontFamily{T1}{cmex}{}
\DeclareFontShape{T1}{cmex}{m}{n}{<-> s * [0.89] cmex10}{}
\DeclareSymbolFont{cmlargesymbols}{T1}{cmex}{m}{n}
\DeclareMathSymbol{\mycoprod}{\mathop}{cmlargesymbols}{"60} 
\DeclareMathSymbol{\myprod}{\mathop}{cmlargesymbols}{"51} \let\prod\myprod

\DeclareSymbolFont{usualmathcal}{OMS}{cmsy}{m}{n}
\DeclareSymbolFontAlphabet{\mathcal}{usualmathcal}

\DeclareSymbolFont{letters}{OML}{txmi}{m}{it}

\DeclareMathSymbol{\alpha}{\mathord}{letters}{"0B}
\DeclareMathSymbol{\beta}{\mathord}{letters}{"0C}
\DeclareMathSymbol{\gamma}{\mathord}{letters}{"0D}
\DeclareMathSymbol{\sigma}{\mathord}{letters}{"0E}
\DeclareMathSymbol{\epsilon}{\mathord}{letters}{"0F}
\DeclareMathSymbol{\zeta}{\mathord}{letters}{"10}
\DeclareMathSymbol{\eta}{\mathord}{letters}{"11}
\DeclareMathSymbol{\theta}{\mathord}{letters}{"12}
\DeclareMathSymbol{\iota}{\mathord}{letters}{"13}
\DeclareMathSymbol{\kappa}{\mathord}{letters}{"14}
\DeclareMathSymbol{\lambda}{\mathord}{letters}{"15}
\DeclareMathSymbol{\mu}{\mathord}{letters}{"16}
\DeclareMathSymbol{\nu}{\mathord}{letters}{"17}
\DeclareMathSymbol{\xi}{\mathord}{letters}{"18}
\DeclareMathSymbol{\pi}{\mathord}{letters}{"19}
\DeclareMathSymbol{\rho}{\mathord}{letters}{"1A}
\DeclareMathSymbol{\sigma}{\mathord}{letters}{"1B}
\DeclareMathSymbol{\tau}{\mathord}{letters}{"1C}
\DeclareMathSymbol{\upsilon}{\mathord}{letters}{"1D}
\DeclareMathSymbol{\phi}{\mathord}{letters}{"1E}
\DeclareMathSymbol{\chi}{\mathord}{letters}{"1F}
\DeclareMathSymbol{\psi}{\mathord}{letters}{"20}
\DeclareMathSymbol{\omega}{\mathord}{letters}{"21}
\DeclareMathSymbol{\varepsilon}{\mathord}{letters}{"22}
\DeclareMathSymbol{\vartheta}{\mathord}{letters}{"23}
\DeclareMathSymbol{\varpi}{\mathord}{letters}{"24}
\DeclareMathSymbol{\varrho}{\mathord}{letters}{"25}
\DeclareMathSymbol{\varsigma}{\mathord}{letters}{"26}
\DeclareMathSymbol{\varphi}{\mathord}{letters}{"27}

\DeclareMathSymbol{\Gamma}{\mathord}{letters}{"00}
\DeclareMathSymbol{\Delta}{\mathord}{letters}{"01}
\DeclareMathSymbol{\Theta}{\mathord}{letters}{"02}
\DeclareMathSymbol{\Lambda}{\mathord}{letters}{"03}
\DeclareMathSymbol{\Xi}{\mathord}{letters}{"04}
\DeclareMathSymbol{\Pi}{\mathord}{letters}{"05}
\DeclareMathSymbol{\Sigma}{\mathord}{letters}{"06}
\DeclareMathSymbol{\Upsilon}{\mathord}{letters}{"07}
\DeclareMathSymbol{\Phi}{\mathord}{letters}{"08}
\DeclareMathSymbol{\Psi}{\mathord}{letters}{"09}
\DeclareMathSymbol{\Omega}{\mathord}{letters}{"0A}

\DeclareMathSymbol{\upGamma}{\mathalpha}{operators}{"00}
\DeclareMathSymbol{\upDelta}{\mathalpha}{operators}{"01}
\DeclareMathSymbol{\upTheta}{\mathalpha}{operators}{"02}
\DeclareMathSymbol{\upLambda}{\mathalpha}{operators}{"03}
\DeclareMathSymbol{\upXi}{\mathalpha}{operators}{"04}
\DeclareMathSymbol{\upPi}{\mathalpha}{operators}{"05}
\DeclareMathSymbol{\upSigma}{\mathalpha}{operators}{"06}
\DeclareMathSymbol{\upUpsilon}{\mathalpha}{operators}{"07}
\DeclareMathSymbol{\upPhi}{\mathalpha}{operators}{"08}
\DeclareMathSymbol{\upPsi}{\mathalpha}{operators}{"09}
\DeclareMathSymbol{\upOmega}{\mathalpha}{operators}{"0A}

\DeclareMathAlphabet\PazoBB{U}{fplmbb}{m}{n}%

\newcommand{\cpx}[1]{#1_{\scriptscriptstyle{\bullet}}}

\newcommand{\Hom}[3]{\operatorname{Hom}_{#1}(#2,#3)}
\newcommand{\Ext}[4]{\operatorname{Ext}_{#1}^{#2}(#3,#4)}
\newcommand{\Tor}[4]{\operatorname{Tor}^{#1}_{#2}(#3,#4)}
\renewcommand{\Im}[1]{\operatorname{Im}\mspace{1mu}#1}
\newcommand{\Ker}[1]{\operatorname{Ker}\mspace{1mu}#1}
\newcommand{\Coker}[1]{\operatorname{Cok}\mspace{1mu}#1}

\newcommand{\lMod}[1]{{}_{#1}\mspace{-1mu}\operatorname{Mod}}
\newcommand{\rMod}[1]{\operatorname{Mod}_{#1}}
\newcommand{\bMod}[2]{{}_{#1}\mspace{-1mu}\operatorname{Mod}_{#2}}

\newcommand{\GPrj}{\operatorname{GPrj}}
\newcommand{\lGPrj}[1]{{}_{#1}\mspace{-1mu}\operatorname{GPrj}}
\newcommand{\GInj}{\operatorname{GInj}}
\newcommand{\lGInj}[1]{{}_{#1}\mspace{-1mu}\operatorname{GInj}}

\newcommand{\lPrj}[1]{{}_{#1}\mspace{-1mu}\operatorname{Prj}}

\newcommand{\lInj}[1]{{}_{#1}\mspace{-1mu}\operatorname{Inj}}

\newcommand{\Fin}{\mathcal{L}}
\newcommand{\lFin}[1]{{}_{#1}\mspace{-1mu}\mathcal{L}}

\newcommand{\pd}[2]{\operatorname{pd}_{#1}#2}
\newcommand{\id}[2]{\operatorname{id}_{#1}#2}

\newcommand{\Serre}{\mathbb{S}}

\newcommand{\QSD}[2]{\mathcal{D}_{#1}(#2)}

\newcommand{\stalkco}[1]{S\mspace{-2mu}\langle{#1}\rangle}
\newcommand{\stalkcn}[1]{S\mspace{-2mu}\{#1\}}

\newcommand{\dom}[1]{\mathrm{dom}\mspace{2mu}#1}
\newcommand{\cod}[1]{\mathrm{cod}\mspace{2mu}#1}
\newcommand{\alg}{A}

\newcommand{\tr}{\tau}
\newcommand{\str}{\sigma}
\newcommand{\mH}[1]{\mathscr{H}_{#1}}

\newcommand{\vtx}{\text{\small $\bullet$}}

\newcommand{\Cq}[1]{C_{\mspace{-1mu}\smash{#1}}}
\newcommand{\Sq}[1]{S_{\mspace{-4mu}\smash{#1}}}
\newcommand{\Kq}[1]{K_{\smash{#1}}}
\newcommand{\Fq}[1]{F_{\mspace{-5mu}\smash{#1}}}
\newcommand{\Eq}[1]{E_{\smash{#1}}}
\newcommand{\Gq}[1]{G_{\mspace{-1mu}\smash{#1}}}

\newcommand{\setco}[1]{J_{\mspace{-2mu}#1}}
\newcommand{\setcn}[1]{I_{\mspace{-1mu}#1}}
\newcommand{\mapco}[1]{\psi_{\mspace{-2mu}#1}}
\newcommand{\mapcn}[1]{\varphi_{\mspace{-1mu}#1}}
\newcommand{\Mapco}[2]{\Psi_{\mspace{-2mu}#1}^{#2}}
\newcommand{\Mapcn}[2]{\Phi_{\mspace{-2mu}#1}^{\mspace{2mu}#2}}

\newcommand{\ORt}[3][Q]{#2 \otimes_{#1} #3}

\newcommand{\hH}[2][1]{\mathbb{H}_{#1}^{\textnormal{\tiny[}#2\textnormal{\tiny]}}}
\newcommand{\cH}[2][1]{\mathbb{H}^{#1}_{\textnormal{\tiny[}#2\textnormal{\tiny]}}}

\newcommand{\dou}[1]{#1^\mathrm{dou}}
\newcommand{\rep}[1]{#1^\mathrm{rep}}
\newcommand{\mesh}[1]{Q_\mathrm{mesh}(#1)}

\newcommand{\isomap}[3]{\upXi^{#1}_{#2,\mspace{2mu}#3}}
\newcommand{\isoelt}[3]{\xi^{\mspace{2mu}#1}_{#2,\mspace{1mu}#3}}
\newcommand{\isoeltd}[3]{\Check{\xi}^{\mspace{2mu}#1}_{#2,\mspace{1mu}#3}}

\newcommand{\mat}[3][]{T^{#1}_{\mspace{-3mu}#2,#3}}
\newcommand{\kfct}[2]{k_{#1,#2}}

\begin{document}


\title{The $Q$-shaped derived category of a ring}

\author{Henrik Holm \ }

\address{Department of Mathematical Sciences, University of Copenhagen, Universitetsparken~5, 2100 Copenhagen {\O}, Denmark} 
\email{holm@math.ku.dk}

\urladdr{http://www.math.ku.dk/\~{}holm/}

\author{\ Peter J{\o}rgensen}

\address{Department of Mathematics, Aarhus University, Ny Munkegade 118, Building 1530, 8000 Aarhus C, Denmark}
\email{peter.jorgensen@math.au.dk}

\urladdr{https://sites.google.com/view/peterjorgensen}





\begin{abstract}
  For any ring $\alg$ and a small, preadditive, Hom-finite, and locally bounded~category $Q$ that has a Serre functor and satisfies the (strong) retraction property, we show that the category of additive functors $Q \to \lMod{\alg}$ has a projective and an injective model structure. These model structures have the same trivial objects and weak equivalences, which in most cases can be naturally characterized in terms of certain (co)homology functors introduced in this paper. The associated homotopy category, which is triangulated, is called the $Q$-shaped derived category of $\alg$. The usual derived category of $\alg$ is one example;~more general examples arise by taking $Q$ to be the mesh category of a suitably nice stable translation quiver. This paper builds upon, and generalizes, works of Enochs,~Estrada,~and~Garc\'{\i}a-Rozas \cite{MR2404296} and Dell'Ambrogio, Stevenson, and \v{S}\v{t}ov\'{\i}\v{c}ek \cite{MR3719530}.
\end{abstract}

\maketitle


\section{Introduction}
\label{sec:Introduction}

Let $\mathcal{Q}$ and $\mathcal{M}$ be categories, where $\mathcal{Q}$ is small. It is well-known that in some cases---for example, if $\mathcal{Q}$ is a direct, an inverse, or a Reedy category---a model structure on $\mathcal{M}$, in the sense of Quillen \cite{Qui67}, induces a model structure on the category $\operatorname{Fun}(\mathcal{Q},\mathcal{M})$~of~functors from $\mathcal{Q}$ to $\mathcal{M}$. The category $\mathcal{M} = \lMod{\alg}$ of left modules over a ring $\alg$ does not, in general, have any non-trivial model structures (unless $\alg$ is special, e.g.~Gorenstein). Nevertheless, we show in this paper that  if $\mathcal{Q}=Q$ is a suitably nice preadditive category, then the category $\lMod{Q,\alg}$ of additive functors $Q \to \lMod{\alg}$ does always have two interesting model structures, the so-called \emph{projective} and  \emph{injective} model structures. These model structures have the same weak equivalences and hence the same homotopy category,
\begin{equation*}
  \QSD{Q}{\alg} \,:=\, \operatorname{Ho}(\lMod{Q,\alg}) \,=\, \{\text{weak equivalences}\}^{-1}(\lMod{Q,\alg})\;,
\end{equation*}
which we call \emph{$Q$-shaped derived category} of $\alg$. This terminology is inspired by the si\-tua\-tion where $Q$ is the mesh category of the repetitive quiver of \smash{$\vec{\mathbb{A}}_2 \,=\, \vtx \to \vtx$}\,. Indeed, in this case, $\lMod{Q,\alg}$ is equivalent to the category $\operatorname{Ch}\mspace{1mu}\alg$ of chain complexes of left $\alg$-modules and the $Q$-shaped derived category  is the ordinary derived category of $\alg$ (see \exaref{derived-cat}). However, the theory developed in this paper applies to many other types of categories as well; for example, let $Q$ be the path category of the quiver
  \begin{equation*}
  \xymatrix{
    \underset{1}{\vtx} \ar@<0.6ex>[r]^-{a^{}_1} &
    \underset{2}{\vtx} \ar@<0.6ex>[l]^-{a^*_1} \ar@<0.6ex>[r]^-{a^{}_2} &
    \underset{3}{\vtx} \ar@<0.6ex>[l]^-{a^*_2}
  }
  \end{equation*}
  modulo the (mesh) relations $a^*_1a^{}_1=0$, $a^{}_1a^*_1+a^*_2a^{}_2=0$, and $a^{}_2a^*_2=0$. Also in this case there is a $Q$-shaped derived category of any ring $\alg$ (and \prpref[Propositions~]{quiso-2} and \prpref[]{2H} yield explicit descriptions of the weak equivalences). More examples can be found in \secref{Example}.
  
The precise statements about the model structures we construct on $\lMod{Q,\alg}$ are~\mbox{contained} in the next result, which is a special case of \thmref{model-structures} and \prpref{we} with $\Bbbk=\mathbb{Z}$.

\begin{res*}[Theorem~A]
  Let $Q$ be a small preadditive category which is Hom-finite, locally bounded, has a Serre functor and has the Retraction Property (\stpref{Bbbk}). For any ring $\alg$ there is a class $\mathscr{E}$ of \emph{\emph{exact}} objects in $\lMod{Q,\alg}$ (\dfnref{E}) and two model structures as follows:
  \begin{itemlist}
  \item The \emph{\emph{projective model structure}} on $\lMod{Q,\alg}$, where ${}^\perp\mathscr{E}$ is the class of cofibrant objects, $\mathscr{E}$ is the class of trivial objects, and every object is fibrant.
  
  \item The \emph{\emph{injective model structure}} on $\lMod{Q,\alg}$, where $\mathscr{E}^\perp$ is the class of fibrant objects, $\mathscr{E}$ is the class of trivial objects, and every object is cofibrant.
  \end{itemlist}
These two model categories have the same weak equivalences.
\end{res*}

To prove Theorem~A we apply Hovey's theory \cite{Hovey02} for abelian model categories, which in this case boils down to demonstrating that we have complete hereditary cotorsion pairs $({}^\perp\mathscr{E},\mathscr{E})$ and $(\mathscr{E},\mathscr{E}^\perp)$ in $\lMod{Q,\alg}$ such that ${}^\perp\mathscr{E} \cap \mathscr{E} = \lPrj{Q,\alg}$ and $\mathscr{E} \cap \mathscr{E}^\perp = \lInj{Q,\alg}$; here $\lPrj{Q,\alg}$ and $\lInj{Q,\alg}$ denote the classes of projective and injective objects in $\lMod{Q,\alg}$. These arguments take up \secref[Sections~]{Existence} and \secref[]{Completeness} and \appref{Kaplansky}. In fact, for the cotorsion pair $(\mathscr{E},\mathscr{E}^\perp)$ we show in \thmref{completeness-2} an even stronger result:

\begin{res*}[Theorem~B]
  Let $Q$ be as in Theorem~A. The cotorsion pair $(\mathscr{E},\mathscr{E}^\perp)$ in $\lMod{Q,\alg}$ is \emph{\emph{perfect}} meaning that every object in $\lMod{Q,\alg}$ has an $\mathscr{E}$-cover and an $\mathscr{E}^\perp$-envelope.
\end{res*}

\noindent
Let us again consider the special case where $Q$ is the mesh category of the repetitive quiver of \smash{$\vec{\mathbb{A}}_2$} and $\lMod{Q,\alg}$ is the category of chain complexes of left $\alg$-modules. In this case, completeness of the cotorsion pairs $({}^\perp\mathscr{E},\mathscr{E})$ and $(\mathscr{E},\mathscr{E}^\perp)$ means that every chain complex has an epic DG-projective and a monic DG-injective resolution. This is of course well-known and goes back to Spaltenstein \cite{NSp88}. Theorem~B asserts in this case that every chain complex has an exact cover and a DG-injective envelope; this is also well-known and can be found in \cite{MR1644453} by Enochs, Jenda, and Xu.

Once the complete hereditary cotorsion pairs $({}^\perp\mathscr{E},\mathscr{E})$ and $(\mathscr{E},\mathscr{E}^\perp)$ in $\lMod{Q,\alg}$ have been established, the general theory of abelian model categories provides us with rich~information about the homotopy category of $\lMod{Q,\alg}$. For example, an application of main results from \cite{MR3608936} by Gillespie yields the following, which is contained in \thmref{Frobenius}.

\begin{res*}[Theorem~C]
  Let $Q$ be as in Theorem~A. There are equivalences of categories,\begin{equation*}
  \frac{{}^\perp\mathscr{E}}{\lPrj{Q,\alg}} \ \simeq \ 
  \QSD{Q}{\alg} \ \simeq \
  \frac{\mathscr{E}^\perp}{\lInj{Q,\alg}}\;.
\end{equation*}  
Here the leftmost, respectively, rightmost, category is the stable category of the Frobenius category ${}^\perp\mathscr{E}$, respectively, $\mathscr{E}^\perp$. In particular, $\QSD{Q}{\alg}$ is triangulated.
\end{res*}

The goal of \secref{cohomology} is to obtain (co)homological characterizations of the trivial~objects and the weak equivalences in the projective/injective model structure on $\lMod{Q,\alg}$. To that end, we consider a category $Q$ as in Theorem~A, only the 
Retraction Property (condition \eqref{retraction} in \stpref{Bbbk}) must now be replaced with the Strong Retraction Property (condition \eqmref{strong-retraction} in \dfnref{srp}). This slightly stronger assumption on $Q$ does not exclude any examples of interest to us. The power of the Strong Retraction Property is that it allows one to define the \emph{pseudo-radical} ideal $\mathfrak{r}$ of $Q$ and certain \emph{(co)homology} functors,
\begin{equation*}
  \cH[i]{q},\, \hH[i]{q} \colon \lMod{Q,\alg} \longrightarrow \lMod{\alg}\;,
\end{equation*}
for every object $q \in Q$ and $i \geqslant 0$. The next result is a special case of \thmref[Theorems~]{E-characterization-hereditary} and \thmref[]{quiso} (it also explains the terminology ``exact'' for the objects in $\mathscr{E}$, introduced in Theorem~A).

\begin{res*}[Theorem~D]
  Let $Q$ be a small preadditive category which is Hom-finite, locally bounded, has a Serre functor and has the Strong Retraction Property. Assume that the pseudo-radical $\mathfrak{r}$ is nilpotent, that is, $\mathfrak{r}^N=0$ for some $N \in \mathbb{N}$. Finally, let $\alg$ be any ring.
  
For every object $X$ in $\lMod{Q,\alg}$ the following conditions are equivalent:
\begin{eqc}
\item $X \in \mathscr{E}$.
\item $\cH[i]{q}(X)=0$ for every $q \in Q$ and $i>0$.
\item $\hH[i]{q}(X)=0$ for every $q \in Q$ and $i>0$.
\end{eqc}

For every morphism $\varphi$ in $\lMod{Q,\alg}$ the following conditions are equivalent:
  \begin{eqc}
  \item $\varphi$ is a weak equivalence.
  \item $\cH[i]{q}(\varphi)$ is an isomorphism for every $q \in Q$ and $i>0$.
  \item $\hH[i]{q}(\varphi)$ is an isomorphism for every $q \in Q$ and $i>0$.
  \end{eqc}    
\end{res*}

By definition, \smash{$\cH[i]{q}$}, respectively, \smash{$\hH[i]{q}$}, is the $i^\mathrm{th}$ right, respectively, left, derived of a~cer\-tain functor $\Kq{q}$, respectively, $\Cq{q}$, which is treated in \prpref[Propositions~]{adjoint-stalk} and \prpref[]{K-C-formulae}. Together with the class $\mathscr{E}$, the functors $\Cq{q}$ and $\Kq{q}$ provide useful characterizations of the projective and injective objects in $\lMod{Q,\alg}$. The following is a special case of \thmref{Prj-Inj}.

\begin{res*}[Theorem~E]
  Let $Q$ be as in Theorem~D and $\alg$ be any ring. For every $X \in \lMod{Q,\alg}$ one has:
\begin{prt}
\item $X \in \lPrj{Q,\alg}$ if and only if $X \in \mathscr{E}$ and the $\alg$-module $\Cq{q}(X)$ is projective for all $q \in Q$.
\item $X \in \lInj{Q,\alg}$ if and only if $X \in \mathscr{E}$ and the $\alg$-module $\Kq{q}(X)$ is injective for all $q \in Q$.
\end{prt}
\end{res*}

\noindent
In the special case where $Q$ is the mesh category of the repetitive quiver of \smash{$\vec{\mathbb{A}}_2$}, and hence $\lMod{Q,\alg} \simeq \operatorname{Ch}\mspace{1mu}\alg$, Theorem~E asserts that a chain complex $(X,\partial)$ is projective, respectively, injective, if and only if $X$ is exact and each cokernel $\Coker{\partial_q}$, respectively, kernel $\Ker{\partial_q}$, is a projective, respectively, an injective, $\alg$-module. These characterizations of the projective and injective objects in the category $\operatorname{Ch}\mspace{1mu}\alg$ are of course well-known.

In \secref{Example} we study some concrete examples of preadditive categories $Q$ that satisfy the assumptions in Theorem~D (and thus all the theorems in this Introduction apply to $Q$). More precisely, we consider the case where $Q$ is the (integral) mesh category $\mesh{\upGamma}$ of a stable translation quiver $\upGamma$. We prove in \thmref[Theorems~]{A-dou} and \thmref[]{A-rep} that if $\upGamma$ is the \emph{double quiver} or the \emph{repetitive quiver} of \smash{$\vec{\mathbb{A}}_n$}, then its mesh category does, in fact, fulfill the requirements in Theorem~D. We expect the same to be true if \smash{$\vec{\mathbb{A}}_n$} is replaced by, for example, \smash{$\vec{\mathbb{D}}_n$}.

For the mesh category $Q = \mesh{\upGamma}$ of a \emph{normal} stable translation quiver $\upGamma$ we show in \prpref{2H} that the homology functor $\hH{*}$ from \secref{cohomology} agrees with the \emph{mesh homology} $\mH{*}$. Combined with \thmref{E-characterization-hereditary} this gives a hands-on description of the trivial (= exact) objects in $\lMod{Q,\alg}$. As shown in \thmref[Theorems~]{A-dou-normal}~and \thmref[]{A-rep-normal} the double quiver and the repetitive quiver of \smash{$\vec{\mathbb{A}}_n$} are, in fact, normal, and we expect the same to be true if \smash{$\vec{\mathbb{A}}_n$} is replaced by e.g.~\smash{$\vec{\mathbb{D}}_n$}.

\vspace*{1ex}

We end this Introduction by explaining how our work is related to the existing literature. First of all, the entire theory developed in this paper is relative to a commutative base ring $\Bbbk$. This means that $Q$ is actually a $\Bbbk$-preadditive category, $\alg$ is a $\Bbbk$-algebra, and $\lMod{Q,\alg}$ is the category of $\Bbbk$-linear functors $Q \to \lMod{\alg}$, but in this Introduction we have focused on the special case $\Bbbk=\mathbb{Z}$. In this generality, the  assumptions needed on the $\Bbbk$-algebra $\alg$ in Theorems~A, B, and C are that $\Bbbk$ is Gorenstein and $\alg$ has finite projective/injective~dimension over~$\Bbbk$; in Theorems~D and E the ring $\Bbbk$ must be noetherian and hereditary but $\alg$ can be any $\Bbbk$-algebra. All these assumptions are satisfied if we take $\Bbbk$ to be $\mathbb{Z}$.

We emphasize that the conditions in \stpref{Bbbk} (which are the assumptions on $Q$ in Theorems~A, B, and C) come from the paper \cite{MR3719530} by Dell'Ambrogio, Stevenson, and \v{S}\v{t}ov\'{\i}\v{c}ek. If $\Bbbk$ is arbitrary and $\alg$ is {\bf Gorenstein}, then \cite[Thm.~4.6]{MR3719530} shows\footnote{More precisely, the assertion in \cite[Thm.~4.6]{MR3719530} is that the $\Bbbk$-preadditive category $\alg \otimes Q$, i.e.~the \emph{extension} of $Q$ by $\alg$, is \emph{Gorenstein}, which by \cite[Dfn.~2.1]{MR3719530} means that the category $\lMod{\alg \otimes Q}$ of $\Bbbk$-linear functors $\alg \otimes Q \to \lMod{\Bbbk}$ is \emph{locally Gorenstein}. However, it is easy to see that $\lMod{\alg \otimes Q}$ is equivalent to the category $\lMod{Q,\alg}$ of $\Bbbk$-linear functors $Q \to \lMod{\alg}$.} that for a $\Bbbk$-preadditive category $Q$ that satisfies \stpref{Bbbk}, the abelian category $\lMod{Q,\alg}$ is \emph{locally Gorenstein} in the sense of Enochs, Estrada, and Garc\'{\i}a-Rozas \cite[Dfn.~2.18]{MR2404296} (see also \dfnref{Gorenstein-cat}). Being a locally Gorenstein category with enough projectives, $\lMod{Q,\alg}$ has by \cite[Thm.~2.32]{MR2404296} both a projective and an injective model structure; thus Theorem~A is known in this case. 

In the case where $\Bbbk$ is a {\bf field} and $\alg$ is any $\Bbbk$-algebra, some results in this paper follow from our previous work \cite{MR4013804}. Indeed, if $\Bbbk$ is a field one can apply \cite[Thm.~3.2]{MR4013804}, combined with Theorem~E, to obtain the previously mentioned hereditary cotorsion pairs $({}^\perp\mathscr{E},\mathscr{E})$ and $(\mathscr{E},\mathscr{E}^\perp)$ in $\lMod{Q,\alg}$ such that ${}^\perp\mathscr{E} \cap \mathscr{E} = \lPrj{Q,\alg}$ and $\mathscr{E} \cap \mathscr{E}^\perp = \lInj{Q,\alg}$. It follows from \cite[Thm.~3.3(i)]{MR4013804} that the cotorsion pair $({}^\perp\mathscr{E},\mathscr{E})$ is complete, and hence one gets the projective model structure on $\lMod{Q,\alg}$ (cf.~the proof of \thmref{model-structures}). However, completeness of the other cotorsion pair $(\mathscr{E},\mathscr{E}^\perp)$, and hence the existence of the injective model structure on $\lMod{Q,\alg}$, was only proved under special circumstances in \cite[Thm.~3.3(ii)]{MR4013804}.

The reader should notice that condition \eqref{Serre} in \stpref{Bbbk}, i.e.~existence of a Serre functor, and \dfnref{normal}, i.e.~normality of a stable translation quiver, depend on the base ring $\Bbbk$. If
$\Bbbk$ is a {\bf field} and $\upGamma$ is the Auslander--Reiten quiver of a suitable $\Bbbk$-linear category $\mathcal{C}$, then the Auslander--Reiten theory of $\mathcal{C}$ can be used easily to prove that $\upGamma$ is normal and that the mesh category $Q = \mesh{\upGamma}$ has a Serre functor. However, this approach is not available if $\Bbbk$ is a general commutative ring (and we are mainly interested in the case $\Bbbk=\mathbb{Z}$). This is why we have provided proofs of \thmref[Theorems~]{A-dou} and \thmref[]{A-dou-normal}, which yield normality of the double quiver of \smash{$\vec{\mathbb{A}}_n$} and the existence of a Serre functor on its mesh category relative to any commutative base ring $\Bbbk$. The proofs are hands-on, but technical, so in order not to interrupt the flow of the paper, they have been relegated to \appref{proof}. \thmref[Theorems~]{A-rep} and \thmref[]{A-rep-normal}---which yield normality of the \textsl{repetitive} quiver of \smash{$\vec{\mathbb{A}}_n$} and the existence of a Serre functor on its mesh category relative to any commutative base ring $\Bbbk$---have similar (but easier) proofs, which are left to the reader.

As already pointed out, Theorems~A--E in this Introduction hold for any ring $\alg$ (corresponding to the case $\Bbbk=\mathbb{Z}$). On the one hand, Theorem~A generalizes the results from \cite{MR3719530} and \cite{MR2404296} mentioned above; on the other hand, these results play an important role in our arguments and even in the very definition of the class $\mathscr{E}$ of \emph{exact} objects (see \dfnref{E}). Nevertheless, existence and completeness of the cotorsion pairs \smash{$({}^\perp\mathscr{E},\mathscr{E})$} and \smash{$(\mathscr{E},\mathscr{E}^\perp)$} in $\lMod{Q,\alg}$ are far from automatic when $\alg$ is a general ring, and the proofs require a mix of known and new techniques. The (co)homological theory developed in \secref{cohomology} and the systematic treatment of examples found in \secref{Example} are new; however, the ideas go back to \cite{MR4013804} (which only treats the easier special case where $\alg$ is an algebra over a field).

\enlargethispage{2ex}

\section{Preliminaries and Notation}
\label{sec:Preliminaries}

Let $\mathcal{A}$ be an abelian category and $\mathcal{C}$ be a class of objects in $\mathcal{A}$. If $\mathcal{A}$ has enough projectives (resp., enough injectives), then $\mathcal{C}$ is called \emph{resolving} (resp., \emph{coresolving}) if it contains all projective (resp., all injective) objects in $\mathcal{A}$ and is closed under extensions and kernels of epimorphisms (resp., extensions and cokernels of monomorphisms). We set
\begin{align*}
  {}^\perp\mathcal{C} &\,=\, \{ X \in \mathcal{A} \ |\ \Ext{\mathcal{A}}{1}{X}{C}=0 \textnormal{ for all } C \in \mathcal{C} \} \quad \textnormal{and}
  \\
  \mathcal{C}^\perp &\,=\, \{ X \in \mathcal{A} \ |\ \Ext{\mathcal{A}}{1}{C}{X}=0 \textnormal{ for all } C \in \mathcal{C} \}\;.
\end{align*}
A \emph{cotorsion pair} in $\mathcal{A}$ consists of a pair $(\mathcal{C},\mathcal{D})$ of classes of objects in $\mathcal{A}$ such that $\mathcal{C}^\perp = \mathcal{D}$ and  $\mathcal{C}={}^\perp\mathcal{D}$. A cotorsion pair $(\mathcal{C},\mathcal{D})$ is \emph{hereditary} if $\Ext{\mathcal{A}}{i}{C}{D}=0$ for all $C \in \mathcal{C}$, $D \in \mathcal{D}$, and $i>0$. If $\mathcal{A}$ has enough projectives (resp., enough injectives), then a cotorsion pair $(\mathcal{C},\mathcal{D})$ is hereditary if and only if $\mathcal{C}$ is resolving (resp., $\mathcal{D}$ is coresolving); see (the~proof of) \cite[Thm.~1.2.10]{GR99} or \cite[Lem.~2.2.10]{GobelTrlifaj}. A cotorsion pair $(\mathcal{C},\mathcal{D})$ is \emph{complete} if it satisfies the following two
conditions:
\begin{eqc}
\item For every $X \in \mathcal{A}$ there is an exact sequence \mbox{$0 \to D \to C \to X \to 0$} with $C \in \mathcal{C}$, $D \in \mathcal{D}$.

\item For every $X \in \mathcal{A}$ there is an exact sequence \mbox{$0 \to X \to D \to C \to 0$} with $C \in \mathcal{C}$, $D \in \mathcal{D}$.
\end{eqc}
By Salce's lemma one has \eqclbl{i}\,$\Rightarrow$\,\eqclbl{ii} if the category $\mathcal{A}$ has enough injectives, and similarly \eqclbl{ii}\,$\Rightarrow$\,\eqclbl{i} holds if $\mathcal{A}$ has enough projectives; see \cite[(proof of) Lem.~2.2.6]{GobelTrlifaj}.

We write $\pd{\mathcal{A}}{X}$ and $\id{\mathcal{A}}{X}$ for the projective and injective dimensions of an object $X$ in $\mathcal{A}$. The finitistic projective and finitistic injective dimensions of $\mathcal{A}$ are defined as usual:
\begin{align*}
  \operatorname{FPD}(\mathcal{A}) &\,=\,
  \sup\{\pd{\mathcal{A}}{X} \,|\ \text{$X \in \mathcal{A}$ with $\pd{\mathcal{A}}{X}<\infty$} \} \quad \textnormal{and}
  \\
  \operatorname{FID}(\mathcal{A}) &\,=\,
  \sup\{\id{\mathcal{A}}{X} \,|\ \text{$X \in \mathcal{A}$ with $\id{\mathcal{A}}{X}<\infty$} \}\;.
\end{align*}  

The next definition is due to Enochs, Estrada, and Garc\'{\i}a-Rozas \cite{MR2404296}.  

\begin{dfn}[{\cite[Dfn.~2.18]{MR2404296}}]
  \label{dfn:Gorenstein-cat}
  A Grothendieck category $\mathcal{A}$ is said to be \emph{locally Gorenstein}\footnote{\,The authors of \cite{MR2404296} simply use the term \emph{Gorenstein} for such a category, but we have adopted the term \emph{locally Gorenstein} from Dell'Ambrogio, Stevenson, and \v{S}\v{t}ov\'{\i}\v{c}ek \cite[Dfn.~2.1]{MR3719530}. See \ref{DASS} for further details.} if it satisfies the following conditions:
\begin{rqm}
\item For every object $X$ in $\mathcal{A}$ one has $\pd{\mathcal{A}}{X}<\infty$ if and only if $\id{\mathcal{A}}{X}<\infty$.
\item $\operatorname{FPD}(\mathcal{A})$ and $\operatorname{FID}(\mathcal{A})$ are both finite.  
\item $\mathcal{A}$ has a generator of finite projective dimension.
\end{rqm}  
In this situation one sets $\Fin(\mathcal{A}) = \{X \in \mathcal{A} \,|\, \pd{\mathcal{A}}{X}<\infty\} = \{X \in \mathcal{A} \,|\, \id{\mathcal{A}}{X}<\infty\}$.
\end{dfn}

Notice that in the definition above, the category $\mathcal{A}$ is not assumed to have enough projec\-tives. Thus, the projective dimension of an object $X$ in $\mathcal{A}$ is defined in terms of vanishing of $\Ext{\mathcal{A}}{*}{X}{-}$ and not in terms of a projective resolution of $X$. However, the Grothendieck categories of interest in this paper do have enough projectives.

The following is a recap of some main results by Enochs, Estrada, and Garc\'{\i}a-Rozas~\cite{MR2404296}. The definition of Gorenstein projective and Gorenstein injective objects in abelian categories can be found in Garc{\'{\i}}a Rozas \cite[Dfn.~1.2.8]{GR99}. Gorenstein projective and Gorenstein injective modules over a ring were introduced and studied by Enochs and Jenda \cite{EEnOJn95b}.

\begin{thm}[{\cite[Thms.~2.25, 2.26, and 2.28]{MR2404296}}]
  \label{thm:G-cotorsion-pairs}
  Let $\mathcal{A}$ be a locally Gorenstein category with enough pro\-jectives. There exist two complete and hereditary cotorsions pairs,
  \begin{equation*}
    (\GPrj(\mathcal{A}),\Fin(\mathcal{A}))
    \quad \text{ and } \quad
    (\Fin(\mathcal{A}),\GInj(\mathcal{A}))\;,
  \end{equation*}
  where $\GPrj(\mathcal{A})$ and $\GInj(\mathcal{A})$ are the classes of Gorenstein projective and Gorenstein injective objects in $\mathcal{A}$. Moreover, there is an equality $\operatorname{FPD}(\mathcal{A}) = \operatorname{FID}(\mathcal{A})$.\footnote{\,The number $\operatorname{FPD}(\mathcal{A}) = \operatorname{FID}(\mathcal{A})$ also coincides with the global Gorenstein projective dimension $\operatorname{glGpd}(\mathcal{A})$ and the global Gorenstein injective dimension $\operatorname{glGid}(\mathcal{A})$; however, this is not important to us.}
\end{thm}

\begin{ipg}
 \label{DASS}
 In the situation of \thmref{G-cotorsion-pairs}, Dell'Ambrogio, Stevenson, and \v{S}\v{t}ov\'{\i}\v{c}ek \cite[Dfn.~2.5]{MR3719530} refer to $\mathcal{A}$ as a \emph{locally $n$-Gorenstein} category, where $n=\operatorname{FPD}(\mathcal{A}) = \operatorname{FID}(\mathcal{A})$. They reserve the term \emph{$n$-Gorenstein} for the following situation:
 
 Let $\Bbbk$ be a nontrivial commutative ring and $Q$ be a small \emph{$\Bbbk$-preadditive} category (also just called a \emph{$\Bbbk$-category}), that is, $Q$ is a category enriched over the symmetric monoidal category $\lMod{\Bbbk}$ of $\Bbbk$-modules. Consider the Grothendieck category $\mathcal{A}=\rMod{Q}$ of right $Q$-modules, that is, the category of $\Bbbk$-linear functors $Q^\mathrm{op} \to \lMod{\Bbbk}$. In the terminology of \cite[Dfn.~2.1]{MR3719530} the $\Bbbk$-preadditive $Q$ is called \emph{\textnormal{(}$n$-\textnormal{)}Gorenstein} if the associated Grothendieck category $\rMod{Q}$ is locally \textnormal{(}$n$-\textnormal{)}Gorenstein in the sense of \dfnref{Gorenstein-cat}.
\end{ipg}

Recall that a commutative ring $\Bbbk$ is said to be \emph{$n$-Gorenstein} if it is noetherian with self-injective dimension $n$. One says that $\Bbbk$ is \emph{Gorenstein} if it $n$-Gorenstein for some $n \in \mathbb{N}_0$.

A main result in \cite{MR3719530} by Dell'Ambrogio, Stevenson, and \v{S}\v{t}ov\'{\i}\v{c}ek is the following.

\begin{thm}[\mbox{\cite[Thm.~1.6\,/\,4.6]{MR3719530}}]
  \label{thm:Q-Gorenstein}
  Assume that a small $\Bbbk$-preadditive category $Q$ satisfies the conditions in \stpref{Bbbk} below and that $\Bbbk$ is $n$-Gorenstein. In this case, $Q$ is $n$-Go\-ren\-stein, i.e.~the Grothendieck category $\rMod{Q}$ is locally $n$-Gorenstein (as in \dfnref{Gorenstein-cat}).
\end{thm}

Actually the result mentioned above is the special case of \cite[Thm.~4.6]{MR3719530} where $R=\Bbbk$. See \secref{Introduction} for further details.

\begin{stp}
\label{stp:Bbbk}
  Throughout this paper, $\Bbbk$ denotes a nontrivial commutative ring and $Q$ a small $\Bbbk$-preadditive category which may or may not satisfy the following conditions coming from \cite[Dfns.~4.1~and~4.5 and Rmk.~4.7]{MR3719530}.
\begin{rqm}
\item \label{eq:Homfin} \emph{Hom-finiteness}: Each hom $\Bbbk$-module $Q(p,q)$ is finitely generated and projective. 

\item \label{eq:locbd} \emph{Local Boundedness}: For each $q \in Q$ there are only finitely many objects in $Q$ map\-ping nontrivially into or out of $q$, that is, the following sets are finite:
\begin{equation*}
  \operatorname{N}_-(q)=\{p \in Q \,|\,Q(p,q) \neq 0\}
  \quad \text{ and } \quad
  \operatorname{N}_+(q)=\{r \in Q \,|\,Q(q,r) \neq 0\}\;.
\end{equation*}

\item \label{eq:Serre} \emph{Existence of a Serre Functor \textnormal{(}relative to $\Bbbk$\textnormal{)}}: There exists a $\Bbbk$-linear autoequivalence $\Serre \colon Q \to Q$ and a natural isomorphism $Q(p,q) \cong \Hom{\Bbbk}{Q(q,\Serre(p))}{\Bbbk}$.

\item \label{eq:retraction} \emph{Retraction Property}: For each $q \in Q$ the unit map $\Bbbk \to Q(q,q)$ given by $x \mapsto x\cdot\mathrm{id}_q$ has a $\Bbbk$-module retraction; whence there is a $\Bbbk$-module decomposition:
\begin{equation*}
  Q(q,q) = (\Bbbk\cdot\mathrm{id}_q) \oplus \mathfrak{r}_q\;.
\end{equation*}  
\end{rqm}
\end{stp}

Notice that the defining properties of a Serre functor depend on the base ring $\Bbbk$. Also note that in part \eqref{retraction} of \stpref{Bbbk}, the complement, $\mathfrak{r}_q$, of $\Bbbk\cdot\mathrm{id}_q$ in $Q(q,q)$ is not unique; see \rmkref{not-unique} and \exaref{PJ-counter} for further details.

\begin{rmk}
\label{rmk:opposite}
The conditions in \stpref{Bbbk} are self-dual, i.e.~if $Q$ satisfies these conditions, then so does its opposite category $Q^\mathrm{op}$ (e.g.~if $\Serre$ is a Serre functor for $Q$, then $\Serre^{-1}$ is a Serre functor for $Q^\mathrm{op}$). Thus, if $Q$ satisfies the conditions in \stpref{Bbbk} and $\Bbbk$ is a Gorenstein ring, then \thmref{Q-Gorenstein} yields that both
$\lMod{Q}$ and $\rMod{Q}$ are locally Gorenstein categories. Here $\lMod{Q}=\rMod{Q^\mathrm{op}}$ is the category of left $Q$-modules, that is, the category of $\Bbbk$-linear functors $Q \to \lMod{\Bbbk}$. In this paper, we shall favour the category $\lMod{Q}$ over $\rMod{Q}$.
\end{rmk}

The final result we will need from \cite{MR3719530} is the following.

\begin{thm}[\mbox{\cite[Cor.~4.8]{MR3719530}}]
  \label{thm:G-description}
  Assume that $Q$ satisfies the conditions in \stpref{Bbbk} and that 
  $\Bbbk$ is Gorenstein. An object $X \in \lMod{Q}$ is Gorenstein projective (resp.~Gorenstein injective) if and only if the $\Bbbk$-module $X(q)$ is Gorenstein projective
(resp.~Gorenstein injective) for every $q \in Q$.
\end{thm}

In the special case where $Q$ is the mesh category of the repetitive quiver of \smash{$\vec{\mathbb{A}}_2$}, and~hence $\lMod{Q}$ is the category $\operatorname{Ch}\mspace{1mu}\Bbbk$ of chain complexes of $\Bbbk$-modules (see \exaref{derived-cat}), the result above is due to Enochs and Garc\'{\i}a-Rozas \cite[Thms.~2.7 and 4.5]{EEnJGR98}.

\section{The Category $\lMod{Q,\alg}$}

\label{sec:Mod}

In the rest of this paper, $\alg$ is any $\Bbbk$-algebra and $\lMod{\alg}$ is the category of left $\alg$-modules.

\begin{dfn}
  \label{dfn:Q-A-Mod}
  We introduce the following $\Bbbk$-preadditive categories:\footnote{\,The notation introduced here is in slight conflict with some of the notation introduced in \secref{Preliminaries}. For~example, according to \thmref{G-cotorsion-pairs} we could use the notation $\operatorname{GPrj}(\lMod{Q,\alg})$ for the category of Gorenstein projective objects in $\lMod{Q,\alg}$, however, the shorthand notation $\lGPrj{Q,\alg}$ is much more convenient.}
  \begin{align*}
    \lMod{Q,\alg} &\,=\, \text{the category of $\Bbbk$-linear functors $Q \to \lMod{\alg}$\;,} 
    \\
    {}_{Q,\alg}\mspace{-1mu}\operatorname{(G)Prj} &\,=\, \text{the category of (Gorenstein) projective objects in $\lMod{Q,\alg}$\;,} 
    \\
    {}_{Q,\alg}\mspace{-1mu}\operatorname{(G)Inj} &\,=\, \text{the category of (Gorenstein) injective objects in $\lMod{Q,\alg}$\;.} 
  \end{align*}
  The hom set ($\Bbbk$-module) in the category $\lMod{Q,\alg}$ is written $\operatorname{Hom}_{Q,\alg}$ and its right derived functors are denoted by $\operatorname{Ext}_{Q,\alg}^i$. For the projective and injective dimensions of  $X \in \lMod{Q,\alg}$ we write $\pd{Q,\alg}{X}$ and $\id{Q,\alg}{X}$.
  
  In the case where $\alg=\Bbbk$ we drop the subscript $\alg$ in the definitions above, that is, we write e.g.~$\lMod{Q}$, $\lGPrj{Q}$, $\operatorname{Ext}_{Q}^i$, and $\pd{Q}{X}$ instead of $\lMod{Q,\Bbbk}$, $\lGPrj{Q,\Bbbk}$, $\operatorname{Ext}_{Q,\Bbbk}^i$, and $\pd{Q,\Bbbk}{X}$.
  
  Beware that \smash{$\operatorname{Ext}_{Q}^0 = \operatorname{Hom}_{Q}$} is the hom set in the category $\lMod{Q}$ and \emph{not} the hom set in $Q$; the latter is written $Q(-,-)$ as indicated in \stpref{Bbbk}.
\end{dfn} 

\begin{dfn}
  \label{dfn:forgetful}
  We use the (same) symbol $(-)^\natural$ for the forgetful functors,
  \begin{equation*}
    (-)^\natural \colon \lMod{\alg} \longrightarrow \lMod{\Bbbk}
    \qquad \text{and} \qquad
    (-)^\natural \colon \lMod{Q,\alg} \longrightarrow \lMod{Q}\;.
  \end{equation*} 
\end{dfn}

\begin{rmk}
  \label{rmk:V}
    Let $\mathcal{V}=\lMod{\Bbbk}$ be the bicomplete closed symmetric monoidal category of $\Bbbk$-modules. In the language of enriched category theory, $Q$ and $\lMod{\alg}$ are both $\mathcal{V}$-categories and $\lMod{Q,\alg}$ is the $\mathcal{V}$-category of $\mathcal{V}$-functors $Q \to \lMod{\alg}$. Thus, by Kelly \cite[eq.~(2.10)]{Kelly} the hom set ($\Bbbk$-module) in $\lMod{Q,\alg}$ can be expressed by the following end in $\lMod{\Bbbk}$:
\begin{equation}
  \label{eq:Hom}
  \Hom{Q,\alg}{X}{Y} \,= \, \int_{q \in Q} \Hom{\alg}{X(q)}{Y(q)}\;.
\end{equation}
\end{rmk}

\begin{prp}
  \label{prp:adjunctions-1}
  For every $M \in \lMod{\alg}$ and $N \in \rMod{\alg}$ there are the following adjunctions, where the left adjoints are displayed in the top of the diagrams:
  \begin{equation*}
  \xymatrix@C=5.5pc{
    \lMod{Q}
    \ar@<0.7ex>[r]^-{- \,\otimes_\Bbbk\, M}
    &
    \lMod{Q,\alg}
    \ar@<0.7ex>[l]^-{\Hom{\alg}{M}{\,-}}
  }
  \qquad \text{and} \qquad
  \xymatrix@C=5.5pc{
    \lMod{Q,\alg}
    \ar@<0.7ex>[r]^-{N \,\otimes_\alg\, -}
    &
    \lMod{Q}
    \ar@<0.7ex>[l]^-{\Hom{\Bbbk}{N}{\,-}}
  }.  
  \end{equation*} 
  (These functors are defined ``objectwise'', for example, for $U \in \lMod{Q}$ the functor $U\otimes_\Bbbk M$ is given by $(U\otimes_\Bbbk M)(q) = U(q)\otimes_\Bbbk M$ for $q \in Q$.)
\end{prp}

\begin{proof}
  For $U \in \lMod{Q}$ and $X \in \lMod{Q,\alg}$ there are by \eqref{Hom} isomorphisms:
  \begin{align*}
    \Hom{Q,\alg}{U \otimes_\Bbbk M}{X}
    &\,\cong\,
    \int_{q \in Q} \Hom{\alg}{U(q) \otimes_\Bbbk M}{X(q)}
    \\
    &\,\cong\,
    \int_{q \in Q} \Hom{\Bbbk}{U(q)}{\Hom{\alg}{M}{X(q)}}
    \\
    &\,\cong\,
    \Hom{Q}{U}{\Hom{\alg}{M}{X}}\;.
  \end{align*}
  This establishes the first adjunction; the other adjuntion is proved similarly.
\end{proof}

\begin{cor}
  \label{cor:adjoint-forgetful}
There is an adjoint triple $(- \otimes_\Bbbk \alg,\, (-)^\natural,\, \Hom{\Bbbk}{\alg}{-})$ as follows:
  \begin{equation*}
  \xymatrix@C=4pc{
    \lMod{Q,\alg}
    \ar[r]^-{(-)^\natural}
    &
    \lMod{Q}\;.
    \ar@/_1.8pc/[l]_-{- \,\otimes_\Bbbk\, \alg}
    \ar@/^1.8pc/[l]^-{\Hom{\Bbbk}{\alg}{\,-}}     
  }
  \end{equation*}
\end{cor}

\begin{proof}
  Apply \prpref{adjunctions-1} with $M={}_{\alg}\alg$ and $N=\alg_{\alg}$ and notice that $\Hom{\alg}{M}{-}$ and $N \otimes_\alg -$ are both naturally isomorphic to the forgetful functor $(-)^\natural \colon \lMod{Q,\alg} \to \lMod{Q}$.
\end{proof}

\begin{rmk}
  \label{rmk:V2}
  As in \rmkref{V} set $\mathcal{V}=\lMod{\Bbbk}$. The $\mathcal{V}$-category
  $\lMod{\alg}$ is both \emph{cotensored} and \emph{tensored}, and the cotensor and tensor products are given by
\begin{equation*}
  V \pitchfork M \,=\, \Hom{\Bbbk}{V}{M} \,\in\, \lMod{\alg}  
  \qquad \text{and} \qquad
  V \odot M \,=\, V \otimes_\Bbbk M \,\in\, \lMod{\alg}
\end{equation*}
for $V \in \lMod{\Bbbk}$ and $M \in \lMod{\alg}$. Indeed, the required/defining isomorphisms, see Kelly \cite[eq.~(3.42) and (3.44)]{Kelly}, in this case take the form
\begin{eqnarray*}
  \Hom{\alg}{N}{\Hom{\Bbbk}{V}{M}} 
  \hspace*{-1.4ex} &\cong& \hspace*{-1.4ex} 
  \Hom{\Bbbk}{V}{\Hom{\alg}{N}{M}} \qquad \text{and} \\
  \Hom{\alg}{V \otimes_\Bbbk M}{N}
  \hspace*{-1.4ex} &\cong& \hspace*{-1.4ex} 
  \Hom{\Bbbk}{V}{\Hom{\alg}{M}{N}}\;,
\end{eqnarray*}  
which are the well-known \emph{swap} and \emph{adjointness} isomorphims from \cite[(A.2.8) and (A.2.9)]{lnm}. 

Now, consider objects $U \in \lMod{Q}$, $W \in \rMod{Q}$, and $X \in \lMod{Q,\alg}$, that is, $\mathcal{V}$-functors $U \colon Q \to \lMod{\Bbbk}$, $W \colon Q^\mathrm{op} \to \lMod{\Bbbk}$, and $X \colon Q \to \lMod{\alg}$. As in \cite[\S3.1]{Kelly} we write:
\begin{itemlist}
\item[] $\{U,X\} \in \lMod{\alg}$ for the \emph{limit} of $X$ \emph{indexed} (or \emph{weighted}) by $U$, and  
\item[] $W \star X \in \lMod{\alg}$ for the \emph{colimit} of $X$ \emph{indexed} (or \emph{weighted}) by $W$.
\end{itemlist}
It follows from \cite[eq.~(3.69) and (3.70)]{Kelly} that the indexed (or weighted) limit and colimit can be computed by the following end and coend in $\lMod{\alg}$.
\begin{align}
  \label{eq:il}
  \{U,X\} &\,=\, \int_{q \in Q} \Hom{\Bbbk}{U(q)}{X(q)} \,=\, \Hom{Q}{U}{X}
  \\
  \label{eq:ic}
  W \star X &\,=\, \int^{q \in Q} W(q) \otimes_\Bbbk X(q) \,=\, \ORt{W}{X}
\end{align}
The last equality in \eqref{il} follows from \eqref{Hom}. The last equality in \eqref{ic} can be taken as a~definition of the symbol ``$\ORt{}{}$'', however, it is precisely the tensor product of $\Bbbk$-linear functors studied by Oberst and R\"{o}hrl \cite[p.~93]{OberstRohrl}, where the same symbol, ``$\ORt{}{}$'', is used.
\end{rmk}

\begin{lem}
  \label{lem:associativity}
  For $W \in \rMod{Q}$, $U \in \lMod{Q}$, and $M \in \lMod{\alg}$ there is an isomorphism,
\begin{equation*}
   \label{eq:associativity}
   \ORt{W}{(U \otimes_\Bbbk M)}
   \,\cong\, 
   (\ORt{W}{U}) \otimes_\Bbbk M\;.
\end{equation*}
\end{lem}

\begin{proof}
  The asserted isomorphism in $\lMod{\alg}$ follows from the formula \eqref{ic} and the fact that the functor $- \otimes_\Bbbk M$ preserves colimits (in particular, coends).
\end{proof}

\begin{prp}
  \label{prp:adjunctions-2}
  For every $U \in \lMod{Q}$ and $W \in \rMod{Q}$ there are the following adjunctions, where the left adjoints are displayed in the top of the diagrams:
  \begin{equation*}
  \xymatrix@C=5.5pc{
    \lMod{\alg}
    \ar@<0.7ex>[r]^-{U\, \otimes_\Bbbk\, -}
    &
    \lMod{Q,\alg}
    \ar@<0.7ex>[l]^-{\Hom{Q}{U}{\,-}}
  }
  \qquad \text{and} \qquad
  \xymatrix@C=5.5pc{
    \lMod{Q,\alg}
    \ar@<0.7ex>[r]^-{\ORt{W\,}{\,-}}
    &
    \lMod{\alg}
    \ar@<0.7ex>[l]^-{\Hom{\Bbbk}{W}{\,-}}
  }.  
  \end{equation*}    
\end{prp}

\begin{proof}
  For $M \in \lMod{\alg}$ and $X \in \lMod{Q,\alg}$ there are by \eqref{Hom} and \eqref{il} isomorphisms:
  \begin{align*}
     \Hom{Q,\alg}{U \otimes_\Bbbk M}{X}
     &\,\cong\,
     \int_{q \in Q} \Hom{\alg}{U(q) \otimes_\Bbbk M}{X(q)}
     \\
     &\,\cong\,
     \int_{q \in Q} \Hom{\alg}{M}{\Hom{\Bbbk}{U(q)}{X(q)}}
     \\
     &\,\cong\,
     \operatorname{Hom}_{\alg}\!\Big(M,\int_{q \in Q}\Hom{\Bbbk}{U(q)}{X(q)}\Big)
     \\
     &\,\cong\,
     \Hom{\alg}{M}{\Hom{Q}{U}{X}}\;,
  \end{align*}
  where the third isomorphism holds as the functor $\Hom{\alg}{M}{-}$ preserves limits (in particular, ends). This proves the first adjunction, and the other is shown similarly.
\end{proof}

\begin{cor}
  \label{cor:adjoint-evaluation}
For every $q \in Q$ there is an adjoint triple $(\Fq{q},\Eq{q},\Gq{q})$ as follows:
  \begin{equation*}
  \xymatrix@C=4pc{
    \lMod{Q,\alg}
    \ar[r]^-{\Eq{q}}
    &
    \lMod{\alg}
    \ar@/_1.8pc/[l]_-{\Fq{q}}
    \ar@/^1.8pc/[l]^-{\Gq{q}}     
  }
  \qquad \text{given by} \qquad
  {\setlength\arraycolsep{1.5pt}
   \renewcommand{\arraystretch}{1.2}
  \begin{array}{rcl}
  \Fq{q}(M) &=& Q(q,-) \otimes_\Bbbk M \\
  \Eq{q}(X) &=& X\mspace{1.5mu}(q) \\ 
  \Gq{q}(M) &=& \Hom{\Bbbk}{Q(-,q)}{M}\;.
  \end{array}
  }
  \end{equation*}
  Moreover, the following assertions hold.
  \begin{prt}
  \item The functor $\Fq{q}$ is exact if the $\Bbbk$-module $Q(q,r)$ is flat for every $r \in Q$. 
  \item The functor $\Gq{q}$ is exact if the $\Bbbk$-module $Q(p,q)$ is projective for every $p \in Q$. 
  \end{prt}
\end{cor}

\begin{proof}
  Apply \prpref{adjunctions-2} with $U=Q(q,-)$ and $W=Q(-,q)$. In this case both functors $\Hom{Q}{U}{?} = \Hom{Q}{Q(q,-)}{?}$ and $\ORt{W\,}{{?}} = \ORt{Q(-,q)}{{?}}$ are naturally isomorphic to the evaluation functor $\Eq{q}(?)$ by the Yoneda isomorphisms \cite[eq.~(3.10)]{Kelly} and \eqref{il}, \eqref{ic}. By the way, the isomorphism \mbox{$\ORt{Q(-,q)\,}{{?}} \cong \Eq{q}(?)$} is also established in \cite[\S1, p.~93]{OberstRohrl}. The assertions in \prtlbl{a} and \prtlbl{b} follow directly from the formulae for $\Fq{q}$ and $\Gq{q}$.
\end{proof}

\begin{lem}
  \label{lem:Fq-Gq-preserve}
  Each functor $\Fq{q}$ preserves projective objects and finitely presentable objects. Each functor $\Gq{q}$ preserves injective objects.
\end{lem}

\begin{proof}
  For every object $M$ in $\lMod{A}$ there is a natural isomorphism $\Hom{Q,\alg}{\Fq{q}(M)}{-} \cong \Hom{\alg}{M}{\Eq{q}(-)}$ by \corref{adjoint-evaluation}. The first assertion now follows as the functor $\Eq{q}$ is exact and preserves direct limits. The second assertion follows from the natural isomorphism $\Hom{Q,\alg}{-}{\Gq{q}(M)} \cong \Hom{\alg}{\Eq{q}(-)}{M}$ and exactness of the functor $\Eq{q}$.
\end{proof}

\begin{prp}
  \label{prp:lfp}
  The category $\lMod{Q,\alg}$ is Grothendieck and locally finitely presentable.
  \begin{prt}
  \item The objects $\Fq{q}(\alg) = Q(q,-) \otimes_{\Bbbk} \alg$, where $q \in Q$, are projective and finitely present\-able and they generate $\lMod{Q,\alg}$.
  \item The objects $\Gq{q}(I) = \Hom{\Bbbk}{Q(-,q)}{I}$, where $q \in Q$ and $I$ is any (fixed) faithfully injective left $\alg$-module, are injective and they cogenerate $\lMod{Q,\alg}$. 
  \end{prt}
\end{prp}

\begin{proof}
  \proofoftag{a} By \lemref{Fq-Gq-preserve} each object $\Fq{q}(\alg)$ is projective and finitely presentable. To see that these objects generate $\lMod{Q,\alg}$, let $\tau$ be any morphism in this category. By \corref{adjoint-evaluation} we have
  $\Hom{Q,\alg}{\Fq{q}(\alg)}{\tau} \cong \Hom{\alg}{\alg}{\tau(q)}$. Since $\alg$ is a faithfully projective left~$\alg$-module it follows that if $\Hom{Q,\alg}{\Fq{q}(\alg)}{\tau}=0$ holds for every $q \in Q$, then $\tau=0$.
  
  \proofoftag{b} Dual to the proof of part \prtlbl{a}.
  
    It is well-known that $\lMod{Q,\alg}$ is a Grothendieck category. By \prtlbl{a} the category $\lMod{Q,\alg}$ even has a \textsl{projective} generator.  
  Part \prtlbl{a} shows that $\lMod{Q,\alg}$ is generated by a set of finitely presentable 
objects, so it is a locally finitely presentable Grothendieck category in the sense of Breitsprecher \cite[Dfn.~(1.1)]{SBr70}. Hence it is also a locally finitely presentable category in the usual sense of Crawley-Boevey \cite[\S1]{WCB94} or Ad{\'a}mek~and Rosick{\'y} \cite[\mbox{Dfn.~1.17 with $\lambda=\aleph_0$}]{AdamekRosicky} (cf.~\appref{Kaplansky}). This follows from \cite[Satz~(1.5)]{SBr70} and is also pointed out in \cite[(2.4)]{WCB94}.
\end{proof}

\section{Existence and Heredity of the Cotorsion Pairs $({}^\perp\mathscr{E},\mathscr{E})$ and $(\mathscr{E},\mathscr{E}^\perp)$}

\label{sec:Existence}

Recall from \dfnref{forgetful} that we write $(-)^\natural$ for the forgetful functor. The next is the key definition in this paper.

\begin{dfn}
  \label{dfn:E}
  Let $Q$ satisfy the conditions in \stpref{Bbbk} and let
  $\Bbbk$ be Gorenstein. We set
  \begin{equation*}
    \lFin{Q} \,=\, 
    \{X \in \lMod{Q} \,|\, \pd{Q}{X}<\infty \} \,=\,
    \{X \in \lMod{Q} \,|\, \id{Q}{X}<\infty \}\;,
  \end{equation*}
  where the last equality holds by \thmref{Q-Gorenstein}\,/\,\rmkref{opposite} and \dfnref{Gorenstein-cat}. We also set
  \begin{equation*}
    \mathscr{E} \,=\, \{ X \in \lMod{Q,\alg} \,|\, X^\natural \in \lFin{Q} \}\;.
  \end{equation*}  
  The objects in $\mathscr{E}$ are said to be \emph{exact}; this terminology is justified by \thmref{E-characterization-hereditary}.
\end{dfn}

Our first goal is to investigate when there exist (hereditary) cotorsion pairs $({}^\perp\mathscr{E},\mathscr{E})$ and $(\mathscr{E},\mathscr{E}^\perp)$ in $\lMod{Q,\alg}$. As the next remark shows, a necessary condition for the existence of such cotorsion pairs is that $\alg$ has finite projective/injective dimension as a $\Bbbk$-module. As proved in \thmref{cotorsion-pairs-existence}, this condition is (perhaps surprisingly) also sufficient.

\begin{rmk}
  \label{rmk:necessary}
We claim that if there exists a cotorsion pair of the form $(\mathscr{E},\mathscr{E}^\perp)$ in $\lMod{Q,\alg}$ and $Q$ is not empty, then $\alg$ must have finite projective/injective dimension as a $\Bbbk$-module.

To see this, choose any $q \in Q$. As $X=Q(q,-) \otimes_\Bbbk \alg$ is in $\lPrj{Q,\alg}$, see \mbox{\prpref{lfp}}, it belongs to the left half, $\mathscr{E}$, of the assumed cotorsion pair. By \dfnref{E} this means that $X^\natural = Q(q,-) \otimes_\Bbbk \alg^\natural$ belongs to $\lFin{Q}$, that is,  $X^\natural$ has finite projective dimension in $\lMod{Q}$. Let $0 \to P_n \to \cdots \to P_0 \to X^\natural \to 0$ be an augmented projective resolution of $X^\natural$ in $\lMod{Q}$. The evalution functor $\Eq{q} \colon \lMod{Q} \to \lMod{\Bbbk}$ has by \corref{adjoint-evaluation} (with $\alg=\Bbbk$) a right adjoint $\Gq{q}$. By \corref{adjoint-evaluation}\prtlbl{b} and condition \eqref{Homfin} in \stpref{Bbbk}, the functor $\Gq{q}$ is exact; whence $\Eq{q}$ preserves projective objects. As $\Eq{q}$ is also exact, $0 \to P_n(q) \to \cdots \to P_0(q) \to X^\natural(q) \to 0$ is a projective resolution of $X^\natural(q)$ in $\lMod{\Bbbk}$. Thus, the $\Bbbk$-module $X^\natural(q) = Q(q,q) \otimes_\Bbbk \alg^\natural$ has finite projective (equivalently, finite injective) dimension. By condition \eqref{retraction} in \mbox{\stpref{Bbbk}} the $\Bbbk$-module $Q(q,q)$ has a direct summand isomorphic to $\Bbbk$, and thus $Q(q,q) \otimes_\Bbbk \alg^\natural$ has a direct summand isomorphic to $\Bbbk \otimes_\Bbbk \alg^\natural \cong \alg^\natural$. Hence also $\alg^\natural$ has finite projective dimension.
\end{rmk}

\begin{lem}
  \label{lem:Ext}
  Assume that $Q$ satisfies the conditions in \stpref{Bbbk},  
  $\Bbbk$ is Gorenstein, and $\alg$ has finite projective/injective dimension as a $\Bbbk$-module. For every $X \in \lMod{Q,\alg}$ there are isomorphisms of\, $\Bbbk$-modules:
\begin{prt}
\item $\Ext{Q,\alg}{i}{G \otimes_\Bbbk \alg}{X} \cong \Ext{Q}{i}{G}{X^\natural}$ for every $G \in \lGPrj{Q}$ and $i \geqslant 0$.

\item $\Ext{Q,\alg}{i}{X}{\Hom{\Bbbk}{\alg}{H}} \cong \Ext{Q}{i}{X^\natural}{H}$ for every $H \in \lGInj{Q}$ and $i \geqslant 0$.
\end{prt}  
\end{lem}

\begin{proof}
\proofoftag{a} Let $G$ be in $\lGPrj{Q}$ and let $\cpx{P} = \cdots \to P_1 \to P_0 \to 0$ be a projective of  resolution of $G$ in $\lMod{Q}$. We show that $\cpx{P} \otimes_\Bbbk \alg$ is a projective resolution of $G \otimes_\Bbbk \alg$ in $\lMod{Q,\alg}$.

By \corref{adjoint-forgetful} the functor $- \otimes_\Bbbk \alg \colon \lMod{Q} \to \lMod{Q,\alg}$ has a right adjoint, namely the forgetful functor $(-)^\natural$, which is exact. Hence the functor $- \otimes_\Bbbk \alg$ preserves projective objects, so each $P_i \otimes_\Bbbk \alg$ is a projective object in $\lMod{Q,\alg}$. It remains to see that the se\-quen\-ce $\cdots \to P_1 \otimes_\Bbbk \alg \to P_0 \otimes_\Bbbk \alg \to G \otimes_\Bbbk \alg \to 0$ is  exact, i.e.~that 
\begin{equation}
  \label{eq:PGA}
\cdots \longrightarrow P_1(q) \otimes_\Bbbk \alg \longrightarrow P_0(q) \otimes_\Bbbk \alg \longrightarrow G(q) \otimes_\Bbbk \alg \longrightarrow 0
\end{equation}  
is an exact sequence of modules for every $q \in Q$. We know that the sequence 
\begin{equation}
  \label{eq:PG}
  \cdots \longrightarrow P_1(q) \longrightarrow P_0(q) \longrightarrow G(q) \longrightarrow 0
\end{equation}  
is exact. As noted in \rmkref{necessary}, each $\Bbbk$-module $P_i(q)$ is projective, so \eqref{PG} is an augmented projective resolution of $G(q)$ in $\lMod{\Bbbk}$. Since the $\Bbbk$-module $G(q)$ is Gorenstein projective by   \thmref{G-description} and one has $\id{\Bbbk}{\alg}<\infty$ by assumption, we have $\Tor{\Bbbk}{i}{G(q)}{\alg}=0$ for all $i>0$ by \cite[Cor.~10.3.10 and Thm.~10.3.8, (1)$\Leftrightarrow$(9)]{rha}. Hence \eqref{PGA} is exact.

The asserted isomorphism now follows from the computation below, where the first and last isomorphisms hold by the definition of Ext and the middle isomorphism holds as $(-)^\natural$ is the right adjoint of $-\otimes_\Bbbk \alg$ by \corref{adjoint-forgetful},
\begin{align*}
  \Ext{Q,\alg}{i}{G \otimes_\Bbbk \alg}{X} 
  \,\cong\, 
  \operatorname{H}^i\Hom{Q,\alg}{\cpx{P} \otimes_\Bbbk \alg}{X} 
  \,\cong\, 
  \operatorname{H}^i\Hom{Q}{\cpx{P}}{X^\natural} 
  \,\cong\,
  \Ext{Q}{i}{G}{X^\natural}\;.
\end{align*}
\proofoftag{b} Dual to the proof of part \prtlbl{a}.
\end{proof}

\begin{thm}
  \label{thm:cotorsion-pairs-existence}
  Assume that $Q$ satisfies the conditions in \stpref{Bbbk},  
  $\Bbbk$ is Gorenstein, and $\alg$ has finite projective/injective dimension as a $\Bbbk$-module. The following assertions hold.
\begin{prt}
\item $({}^\perp\mathscr{E},\mathscr{E})$ is a cotorsion pair in $\lMod{Q,\alg}$; in fact, it is the cotorsion pair generated by $\{G \otimes_\Bbbk \alg\,|\,G \in \lGPrj{Q}\}$, that is, $\{G \otimes_\Bbbk \alg\,|\,G \in \lGPrj{Q}\}^\perp = \mathscr{E}$. 

Moreover, the cotorsion pair $({}^\perp\mathscr{E},\mathscr{E})$ is hereditary and one has ${}^\perp\mathscr{E} \cap \mathscr{E} = \lPrj{Q,\alg}$. 

\item $(\mathscr{E},\mathscr{E}^\perp)$ is a cotorsion pair in $\lMod{Q,\alg}$; in fact, it is the cotorsion pair cogene\-ra\-ted by $\{\Hom{\Bbbk}{\alg}{H}\,|\,H \in \lGInj{Q}\}$, that is, ${}^\perp\{\Hom{\Bbbk}{\alg}{H}\,|\,H \in \lGInj{Q}\} = \mathscr{E}$.

Moreover, the cotorsion pair $(\mathscr{E},\mathscr{E}^\perp)$ is hereditary and one has $\mathscr{E} \cap \mathscr{E}^\perp = \lInj{Q,\alg}$.
\end{prt}  
Furthermore, the class $\mathscr{E}$ is \emph{\emph{thick}}, i.e.~it is closed under direct summands and if two out of three objects in a short exact sequence $0 \to X' \to X \to X'' \to 0$ in $\lMod{Q,\alg}$ belong to $\mathscr{E}$, then so does the third.
\end{thm}

\begin{proof}
The final assertion in the theorem follows directly from \dfnref{E} and the well-known fact that 
$\lFin{Q}$ is a thick subcategory of $\lMod{Q}$. 

\proofoftag{a} The equality $\{G \otimes_\Bbbk \alg\,|\,G \in \lGPrj{Q}\}^\perp = \mathscr{E}$ follows from \lemref{Ext}\prtlbl{a}, the fact that $(\lGPrj{Q},\lFin{Q})$ is a cotorsion pair in $\lMod{Q}$ (see \thmref{Q-Gorenstein}\,/\,\rmkref{opposite} and \thmref{G-cotorsion-pairs}), and
the definition of the class $\mathscr{E}$ (see \dfnref{E}).

To see that $({}^\perp\mathscr{E},\mathscr{E})$ is hereditary we must argue that the class $\mathscr{E}$ in the right half of the cotorsion pair is coresolving. However, this is a special case of the final assertion in the theorem, which has already been proved. We now prove the equality ${}^\perp\mathscr{E} \cap \mathscr{E} = \lPrj{Q,\alg}$.

``$\supseteq$'': Evidently, ${}^\perp\mathscr{E} \supseteq \lPrj{Q,\alg}$. To prove $\mathscr{E} \supseteq \lPrj{Q,\alg}$ it suffices by \prpref{lfp}\prtlbl{a} to show that $\Fq{q}(\alg) = Q(q,-) \otimes_\Bbbk \alg$ is in $\mathscr{E}$ for every $q \in Q$. Thus we must argue that the object $\Fq{q}(\alg)^\natural = Q(q,-) \otimes_\Bbbk \alg^\natural = \Fq{q}(\alg^\natural)$ has finite projective dimension in $\lMod{Q}$ (here the first $\Fq{q}$ is viewed as a functor $\lMod{\alg} \to \lMod{Q,\alg}$ and the second as a functor $\lMod{\Bbbk} \to \lMod{Q}$). By assumption, $\alg^\natural$ has finite projective dimension in $\lMod{\Bbbk}$. Since $\Fq{q} \colon \lMod{\Bbbk} \to \lMod{Q}$ is exact and preserves projective objects, see \corref{adjoint-evaluation}\prtlbl{a} (and condition \eqref{Homfin} in \stpref{Bbbk}) and \lemref{Fq-Gq-preserve}, it follows that $\Fq{q}(\alg^\natural)$ has finite projective dimension in $\lMod{Q}$.

``$\subseteq$'': Assume that $X$ belongs to ${}^\perp\mathscr{E} \cap \mathscr{E}$. Take an exact sequence $0 \to Y \to P \to X \to 0$ in $\lMod{Q,\alg}$ with $P$ projective. As argued above one has $P \in \mathscr{E}$. Since $X \in \mathscr{E}$ by assumption, it follows from the final assertion in the theorem that $Y \in \mathscr{E}$ as well. We also have $X \in {}^\perp\mathscr{E}$; consequently $\Ext{Q,\alg}{1}{X}{Y}=0$ and the sequence $0 \to Y \to P \to X \to 0$ splits. Hence $X$ is a direct summand in the projective object $P$, so $X$ is projective too.

\proofoftag{b} Dual to the proof of part \prtlbl{a}. 
\end{proof}

\section{Completeness of the Cotorsion Pairs $({}^\perp\mathscr{E},\mathscr{E})$ and $(\mathscr{E},\mathscr{E}^\perp)$}

\label{sec:Completeness}

In this section we prove completeness of the cotorsion pairs established in \thmref{cotorsion-pairs-existence} (see \thmref[Theorems~]{completeness-1} and \thmref[]{completeness-2} below). This completeness relies, in part, on some general~properties of the class of objects of injective dimension $\leqslant n$ (for some fixed $n \in \mathbb{N}_0$) in a locally noetherian Grothendieck category, which we now establish.

\begin{dfn}
  For a Grothendieck category $\mathcal{A}$ and a natural number $n \in \mathbb{N}_0$ we set
  \begin{equation*}
    \mathcal{I}_n \,=\, \mathcal{I}_n(\mathcal{A}) \,=\,
    \{X \in \mathcal{A} \,|\, \id{\mathcal{A}}{X} \leqslant n \}\;.
  \end{equation*}
\end{dfn}

\begin{lem}
\label{lem:idlen}
Let $\mathcal{A}$ be a locally finitely presented Grothendieck category. For $X \in \mathcal{A}$ and $n \in \mathbb{N}_0$ one has $X \in \mathcal{I}_n$ if and only if\, $\Ext{\mathcal{A}}{n+1}{F}{X}=0$ for every finitely~ge\-ne\-rated $F \in \mathcal{A}$.
\end{lem}

\begin{proof}
  Let $0 \to X \to I^0 \to I^1 \to \cdots$ be an augmented injective resolution of $X$ in $\mathcal{A}$. Write $\upOmega^i(X) = \Ker{(I^i \to I^{i+1})}$ for the $i^\mathrm{th}$ cosyzygy of $X$. Note that $X \in \mathcal{I}_n$ if and only if $\upOmega^n(X)$ is injective. By Baer's criterion in $\mathcal{A}$, see Krause \cite[Lem.~2.5]{MR1612398}, and by dimension shifting, this happens if and only if $\Ext{\mathcal{A}}{n+1}{F}{X} \cong \Ext{\mathcal{A}}{1}{F}{\upOmega^n(X)}=0$ for every finitely generated object $F \in \mathcal{A}$.
\end{proof}

In the special case where $\mathcal{A}$ is the category of modules over a (noetherian) ring, parts \prtlbl{a} and \prtlbl{b} in the next result can be found in \cite[Thm.~4.1.7]{GobelTrlifaj} and \cite[Lem.~9.1.5]{rha}.

\begin{prp}
  \label{prp:In}
  Let $\mathcal{A}$ be a locally finitely presented Grothendieck category and $n \in \mathbb{N}_0$.
\begin{prt}
\item If $\mathcal{A}$ has enough projectives, then $({}^\perp\mathcal{I}_n,\mathcal{I}_n)$ is a hereditary cotorsion pair in $\mathcal{A}$; this cotorsion pair is generated by a set and hence it is complete.

\item If $\mathcal{A}$ is generated by a set of projective noetherian objects, then $\mathcal{I}_n$ is closed under pure subobjects and pure quotients. This means that for every pure exact sequence $0 \to X' \to X \to X'' \to 0$ in $\mathcal{A}$ (see \ref{purity}) with $X \in \mathcal{I}_n$, one has $X',X'' \in \mathcal{I}_n$.
\end{prt}   
\end{prp}

\begin{proof}
\proofoftag{a} By assumption, $\mathcal{A}$ is generated by a set $\mathcal{X}$ of finitely presented (and hence~fi\-nite\-ly generated) objects. An object in $\mathcal{A}$ is finitely generated if and only if it is a quotient of a finite direct sum of objects from $\mathcal{X}$, see e.g.~Breitsprecher \cite[Satz~(1.6)]{SBr70}. Thus, up to isomorphism, there is only a set, $\mathcal{F}$, of finitely generated objects in $\mathcal{A}$. As $\mathcal{A}$ has enough projectives we can choose for every $F \in \mathcal{F}$ a projective resolution \mbox{$\cdots \to P_1^F \to P_0^F \to F \to 0$}. Let $\upOmega_i(F) = \Coker{(P_{i+1}^F \to P_i^F)}$ be the $i^\mathrm{th}$ syzygy of $F$ in this resolution. For $X \in \mathcal{A}$ one has \smash{$\Ext{\mathcal{A}}{n+1}{F}{X} \cong \Ext{\mathcal{A}}{1}{\upOmega_n(F)}{X}$} by dimension shifting. In view of this and \lemref{idlen} it follows that the set $\upOmega_n(\mathcal{F}) = \{ \upOmega_n(F) \,|\, F \in \mathcal{F}\}$ satisfies $\upOmega_n(\mathcal{F})^\perp = \mathcal{I}_n$. Thus $({}^\perp\mathcal{I}_n,\mathcal{I}_n)$ is the cotorsion pair generated by the set $\upOmega_n(\mathcal{F})$. Hence Saor{\'{\i}}n and {\v{S}}{\v{t}}ov{\'{\i}}{\v{c}}ek \cite[Cor.~2.15(3)]{SaorinStovicek} yields that $({}^\perp\mathcal{I}_n,\mathcal{I}_n)$ is complete. Evidently $\mathcal{I}_n$ is coresolving so $({}^\perp\mathcal{I}_n,\mathcal{I}_n)$ is hereditary.

\proofoftag{b} Since $\mathcal{A}$ is generated by a set of projective noetherian objects we may assume that all the projective objects $P_i^F$ and all the syzygies $\upOmega_i(F)$ in the proof of part \prtlbl{a} are finitely presented (= finitely generated = noetherian, as the category $\mathcal{A}$ is locally noetherian). Let $F \in \mathcal{F}$ be given. Applying the functor $\Hom{\mathcal{A}}{\upOmega_n(F)}{-}$ to the given pure exact sequence, we get an exact sequence
\begin{equation*}
  \Hom{\mathcal{A}}{\upOmega_n(F)}{X} 
  \twoheadrightarrow
  \Hom{\mathcal{A}}{\upOmega_n(F)}{X''}
  \to
  \Ext{\mathcal{A}}{1}{\upOmega_n(F)}{X'}
  \to
  \Ext{\mathcal{A}}{1}{\upOmega_n(F)}{X} = 0\;.
\end{equation*}
In this exact sequence, the first homomorphism is surjective as $0 \to X' \to X \to X'' \to 0$ is pure exact and $\upOmega_n(F)$ is finitely presented. Furthermore, $\Ext{\mathcal{A}}{1}{\upOmega_n(F)}{X}=0$ as $X \in \mathcal{I}_n$, cf.~the proof of part \prtlbl{a}. It follows that $\Ext{\mathcal{A}}{1}{\upOmega_n(F)}{X'}$ and hence $X' \in \mathcal{I}_n$. Since $\mathcal{I}_n$ is coresolving we also get $X'' \in \mathcal{I}_n$.
\end{proof}

  In a locally noetherian Grothendieck category $\mathcal{A}$, the class of injective objects is closed under coproducts; see Gabriel \cite[Chap.~IV\S2, Prop.~6, p.~387]{PGb62} or Stenstr{\"o}m \cite[V\S4, Prop.~4.3]{Stenstrom}. It follows from Roos \cite[Thm.~1, p.~201]{MR0407092} that $\mathcal{A}$ has a \emph{strict cogenerator}, so Simson \cite{MR364398} yields the next result, which also follows easily from \prpref{In}\prtlbl{b}.

\begin{thm}[{\cite[Cor.~p.~163]{MR364398}}]
  \label{thm:Simson}
  Let $\mathcal{A}$ be a locally noetherian Grothendieck category and let $n \in \mathbb{N}_0$. The subcategory $\mathcal{I}_n$ is closed under direct limits. \qed
\end{thm}

\begin{thm}
  \label{thm:completeness-1}
  Assume that $Q$ satisfies the conditions in \stpref{Bbbk}, $\Bbbk$ is Gorenstein, and $\alg$ has finite projective/injective dimension as a $\Bbbk$-module. The cotorsion pair $({}^\perp\mathscr{E},\mathscr{E})$ in $\lMod{Q,\alg}$ from \thmref{cotorsion-pairs-existence}\prtlbl{a} is generated by a set; whence it is complete.
\end{thm}

\begin{proof}
Note that \prpref{In}\prtlbl{a} applies to $\mathcal{A}=\lMod{Q}$ by \prpref{lfp} (with $\alg=\Bbbk$). Let $n$ be the finitistic injective dimension of this category, which is finite by \dfnref{Gorenstein-cat} and \thmref{Q-Gorenstein}\,/\,\rmkref{opposite}. Now the class $\mathcal{I}_n$ is equal to $\lFin{Q} = \{ X \in \lMod{Q} \,|\, \id{Q}{X}<\infty\}$, so the cotorsion pairs $({}^\perp\mathcal{I}_n,\mathcal{I}_n)$ and $(\lGPrj{Q},\lFin{Q})$ (see \thmref{G-cotorsion-pairs}) in $\lMod{Q}$ coincide. By \prpref{In}\prtlbl{a} there is a set $\mathcal{G} \subseteq \lGPrj{Q}$ with $\mathcal{G}^\perp = \lFin{Q}$. It follows from \lemref{Ext}\prtlbl{a} that one has $\{G \otimes_\Bbbk \alg\,|\,G \in \mathcal{G}\}^\perp = \mathscr{E}$ in $\lMod{Q,\alg}$, cf.~the proof of \thmref{cotorsion-pairs-existence}\prtlbl{a}, so the cotorsion pair $({}^\perp\mathscr{E},\mathscr{E})$ in $\lMod{Q,\alg}$ is generated by the set $\{G \otimes_\Bbbk \alg\,|\,G \in \mathcal{G}\}$. Hence \cite[Cor.~2.15(3)]{SaorinStovicek} yields completeness of the cotorsion pair $({}^\perp\mathscr{E},\mathscr{E})$.
\end{proof}

In the language of relative homological algebra, the result above shows that  ${}^\perp\mathscr{E}$ is \emph{special precovering} and $\mathscr{E}$ is \emph{special preenveloping} in $\lMod{Q,\alg}$. In general,  ${}^\perp\mathscr{E}$ is not \emph{covering} and $\mathscr{E}$ is not \emph{enveloping}, see \cite[Thm.~3.4]{MR1679688} by Enochs and Garc\'{\i}a Rozas.

\begin{lem}
  \label{lem:locally-noetherian}
  Assume that a $\Bbbk$-preadditive category $Q$ satisfies the following requirements.
  \begin{rqm}
  \item Each $\Bbbk$-module $Q(q,r)$ is finitely generated.
  \item For every $q \in Q$ the set $\operatorname{N}_+(q)=\{r \in Q \,|\,Q(q,r) \neq 0\}$ is finite.
  \item The ring $\Bbbk$ is noetherian.
  \end{rqm}
  In this situation, every $Q(q,-)$ is a noetherian object in $\lMod{Q}$. In particular, the category $\lMod{Q}$ is generated by a set of projective noetherian objects.
\end{lem}

\begin{proof}
Let $q \in Q$ be given. We must show that every subobject $I$ of $Q(q,-)$ is finitely~generated; see \cite[Chap. V\S4, Prop.~4.1]{Stenstrom}. Such a subobject $I$ is the same as a left ideal in the the category $Q$ at $q$, that is, each~$I(r)$ is a $\Bbbk$-submodule of $Q(q,r)$ and for every $h \in Q(r,p)$ and $g \in I(r)$ one has $hg \in I(p)$. By \prtlbl{2} the set $\operatorname{N}_+(q)$ is finite, say, $\operatorname{N}_+(q)=\{r_1,\ldots,r_n\}$. By \prtlbl{1} each $\Bbbk$-module $Q(q,r_i)$ is finitely generated, and as $\Bbbk$ is noetherian by \prtlbl{3}, the $\Bbbk$-submodule $I(r_i) \subseteq Q(q,r_i)$ is finitely generated as well, say, $I(r_i) = \Bbbk g_{i1} + \cdots +\Bbbk g_{i\ell(i)}$. Consider the morphism
\begin{equation*}
  \textstyle
  \tau \colon \bigoplus_{i=1}^n \bigoplus_{j=1}^{\ell(i)} \, Q(r_i,-) \longrightarrow Q(q,-)
\end{equation*} 
given by $Q(g_{ij},-) \colon Q(r_i,-) \to Q(q,-)$ on the component corresponding to $i \in \{1,\ldots,n\}$ and $j \in \{1,\ldots,\ell(i)\}$. We will show that $\Im{\tau} = I$, and hence $I$ is finitely generated by \cite[Chap.~V\S3, Lem.~3.1]{Stenstrom} as each of the finitely many objects $Q(r_i,-)$ is finitely generated. To prove the equality $\Im{\tau} = I$ we must show that for every $p \in Q$ one has 
\begin{equation}
  \label{eq:sum-sum}
  \textstyle
  \sum_{i=1}^n \sum_{j=1}^{\ell(i)} \, \Im{Q(g_{ij},p)} \,=\, I(p)\;.
\end{equation}  
Let $p \in Q$ be given. To prove the equality displayed above, we argue as follows.

``$\subseteq$'': Consider any of the morphisms $Q(g_{ij},p) \colon Q(r_i,p) \to Q(q,p)$. For any $h \in Q(r_i,p)$ one has $Q(g_{ij},p)(h) = hg_{ij}$, which belongs to $I(p)$ as $g_{ij}$ is in $I(r_i)$ and $I$ is a left ideal.

``$\supseteq$'': If $p \notin \operatorname{N}_+(q)$, then $Q(q,p)=0$ and hence $I(p)=0$. So we may assume that $p=r_i$ for some $i \in \{1,\ldots,n\}$. To prove that $I(p) = I(r_i)$ is contained in the left-hand side of \eqref{sum-sum} (with $p=r_i$), it suffices to show that every generator $g_{i1},\ldots,g_{i\ell(i)}$ of $I(r_i)$ is in the left-hand side. But $g_{ij}$ is in the image of 
$Q(g_{ij},r_i) \colon Q(r_i,r_i) \to Q(q,r_i)$ as $Q(g_{ij},r_i)(\mathrm{id}_{r_i}) = g_{ij}$.
\end{proof}

\begin{lem}
  \label{lem:forgetful-preserves}
  The forgetful functor $(-)^\natural \colon \lMod{Q,\alg} \to \lMod{Q}$ preserves colimits and pure~exact sequences (see \ref{purity}).
\end{lem}

\begin{proof}
  The forgetful functor $(-)^\natural$ is a left adjoint by \corref{adjoint-forgetful}, so it preserves colimits. By \cite[Prop.~3]{AR04} (with $\lambda = \aleph_0$)~the pure exact sequences are precisely the sequences that are direct limits of split exact sequen\-ces. As $(-)^\natural$ preserves direct limits and split exact sequences, it preserves pure exact se\-quen\-ces as well. Here is a more direct argument, which does not use \cite[Prop.~3]{AR04}: Let $\xi = 0 \to X' \to X \to X'' \to 0$ be a pure exact sequence in $\lMod{Q,\alg}$ and let $F \in \lMod{Q}$ be finitely presented. Notice that $F \otimes_\Bbbk \alg$ is finitely presented in $\lMod{Q,\alg}$ as the functor $\Hom{Q,\alg}{F \otimes_\Bbbk \alg}{-} \cong \Hom{Q}{F}{(-)^\natural}$ preserves direct limits. Thus the sequence $\Hom{Q,\alg}{F \otimes_\Bbbk \alg}{\xi} \cong \Hom{Q}{F}{\xi^\natural}$ is exact, so $\xi^\natural$ is pure exact in $\lMod{Q}$.
\end{proof}

The results in \cite{MR1850650} by Aldrich, Enochs, Garc\'{\i}a~Rozas, and Oyonarte are valid for a~Gro\-then\-dieck category \textsl{with enough projectives} (see the Introduction of \emph{loc.~cit.}). In view of the definition of a (minimal) generator of $\mathrm{Ext}$, see \cite[Dfn.~2.8]{MR1850650}, the following is a simply a reformulation of \cite[Thm.~2.9]{MR1850650} in the language of relative homological algebra.

\begin{thm}[{\cite[Thm.~2.9]{MR1850650}}]
  \label{thm:envelope}
  Let $\mathcal{A}$ be a Grothendieck category with enough projectives and  $\mathcal{F}$ be a class of objects in $\mathcal{A}$. If $\mathcal{F}$ is closed under well-ordered direct limits and every object in $\mathcal{A}$ has a special $\mathcal{F}^\perp$-preenvelope, then every object in $\mathcal{A}$ has an $\mathcal{F}^\perp$-envelope. \qed
\end{thm}

\begin{thm}
  \label{thm:completeness-2}
    Assume that $Q$ satisfies the conditions in \stpref{Bbbk}, $\Bbbk$ is Gorenstein, and $\alg$ has finite projective/injective dimension as a $\Bbbk$-module. The cotorsion pair $(\mathscr{E},\mathscr{E}^\perp)$ in $\lMod{Q,\alg}$ from \thmref{cotorsion-pairs-existence}\prtlbl{b} is complete, in fact, it is \emph{\emph{perfect}}, meaning that every object in $\lMod{Q,\alg}$ has an $\mathscr{E}$-cover and an $\mathscr{E}^\perp$-envelope.
\end{thm}

\begin{proof}
We start by proving the following assertions:
\begin{prt}
\item[($*$)] $\lFin{Q}$ is closed under pure subobjects and pure quotients in $\lMod{Q}$.

\item[($**$)] $\lFin{Q}$ is closed under direct limits in $\lMod{Q}$.
\end{prt}

Ad ($*$): As in the proof of \thmref{completeness-1} we have $\mathcal{I}_n = \lFin{Q}$ where $n$ is the finitistic~injec\-tive dimension of $\lMod{Q}$. By \lemref{locally-noetherian} (and \stpref{Bbbk}) the category $\mathcal{A}=\lMod{Q}$ is generated by a set of projective noetherian objects, so \prpref{In}\prtlbl{b} yields the conclusion.

Ad ($**$): By applying \thmref{Simson} to $\mathcal{A}=\lMod{Q}$ (which is locally noetherian, as just noted), it follows that the class $\mathcal{I}_n = \lFin{Q}$ is closed under direct limits.

Next note that $\mathcal{M}=\lMod{Q,\alg}$ is a locally finitely presented Grothendieck category by \prpref{lfp}. We show that $\mathcal{F}=\mathscr{E}$ satisfies the requirements \rqmlbl{1}, \rqmlbl{2} in \thmref{pure-sub-quo}.
  
  Ad \rqmlbl{1}: It follows from ($*$) above, \dfnref{E}, and \lemref{forgetful-preserves}, that $\mathscr{E}$ is closed under pure subobjects and pure quotients in $\lMod{Q,\alg}$.
  
  Ad \rqmlbl{2}: The class $\mathscr{E}$ contains all projective objects in $\lMod{Q,\alg}$ and hence also a projective generator of $\lMod{Q,\alg}$, which exists by \prpref{lfp}\prtlbl{a}. It follows from ($**$) above, \dfnref{E}, and \lemref{forgetful-preserves} that $\mathscr{E}$ is closed under direct limits (and hence coproducts).
  
We conclude from \thmref{pure-sub-quo} that the cotorsion pair $(\mathscr{E},\mathscr{E}^\perp)$ in $\lMod{Q,\alg}$ is complete and that every object in $\lMod{Q,\alg}$ has an $\mathscr{E}$-cover. It remains to prove that every object has an $\mathscr{E}^\perp$-envelope, however, this follows immediately from \thmref{envelope} applied to $\mathcal{F}=\mathscr{E}$.
\end{proof}

\section{The Projective and Injective Model Structures}
\label{sec:model-structures}

We show that for any category $Q$ satisfying the conditions in \stpref{Bbbk} and any ring $\alg$, the category $\lMod{Q,\alg}$ from \dfnref{Q-A-Mod} has two different model structures with the same homotopy category $\QSD{Q}{\alg} := \operatorname{Ho}(\lMod{Q,\alg})$. Indeed, this is a special case of \thmref{model-structures} with $\Bbbk=\mathbb{Z}$. The proof is easy: we essentially just have to combine results from the~pre\-vious \secref[Sections~]{Existence} and \secref[]{Completeness} with Gillespie/Hovey's theory of abelian model categories~\cite{MR3608936,Hovey02}.

To parse the next result, recall that $\mathscr{E}$ denotes the class of exact objects in $\lMod{Q,\alg}$ in the sense of \dfnref{E}.

\begin{thm}
  \label{thm:model-structures}
  Assume that $Q$ satisfies the conditions in \stpref{Bbbk}, $\Bbbk$ is Gorenstein, and $\alg$ has finite projective/injective dimension as a $\Bbbk$-module (e.g.~$\Bbbk=\mathbb{Z}$ and $\alg$ is any ring).
  \begin{prt}
  \item There exists an abelian model structure on $\lMod{Q,\alg}$ where ${}^\perp\mathscr{E}$ is the class of cofibrant objects, $\mathscr{E}$ is the
class of trivial objects, and every object is fibrant.
  
  \item There exists an abelian model structure on $\lMod{Q,\alg}$ where $\mathscr{E}^\perp$ is the class of fibrant objects, $\mathscr{E}$ is the
class of trivial objects, and every object is cofibrant.
  \end{prt}
\end{thm}

\begin{proof}
  \proofoftag{a} We claim that $(\mathcal{C},\mathcal{W},\mathcal{F}) = ({}^\perp\mathscr{E},\mathscr{E},\lMod{Q,\alg})$ is a \emph{Hovey triple} in $\lMod{Q,\alg}$:
\begin{itemlist}  
\item The class $\mathcal{W}=\mathscr{E}$ is thick by the last assertion in \thmref{cotorsion-pairs-existence}.

\item By \thmref{cotorsion-pairs-existence}\prtlbl{a} one has $(\mathcal{C} \cap \mathcal{W},\mathcal{F}) = ({}^\perp\mathscr{E} \cap \mathscr{E}, \lMod{Q,\alg}) = (\lPrj{Q,\alg},\lMod{Q,\alg})$.~It is a complete cotorsion pair as $\lMod{Q,\alg}$ has enough projectives by \prpref{lfp}\prtlbl{a}.

\item Clearly, $(\mathcal{C}, \mathcal{W} \cap \mathcal{F}) = ({}^\perp\mathscr{E}, \mathscr{E})$. This is a complete cotorsion pair by \thmref{completeness-1}.
\end{itemlist}  
The desired conclusion now follows from Hovey \cite[Thm.~2.2]{Hovey02}.
    
\proofoftag{b} By arguing as above, it follows from \thmref[Theorems~]{cotorsion-pairs-existence}\prtlbl{b} and \thmref[]{completeness-2} that $(\mathcal{C},\mathcal{W},\mathcal{F}) = (\lMod{Q,\alg},\mathscr{E},\mathscr{E}^\perp)$ is a Hovey triple in $\lMod{Q,\alg}$.  Now apply \cite[Thm.~2.2]{Hovey02} once more.
\end{proof}

\begin{dfn}
  \label{dfn:prj-inj-ms}
Following the terminology in \cite[Dfn.~4.5 and Lem.~4.6]{MR2811572} by Gillespie, we refer to the model structures on $\lMod{Q,\alg}$ established in parts \prtlbl{a} and \prtlbl{b} in \thmref{model-structures} as the \emph{projective model structure} and the \emph{injective model structure}.
\end{dfn}

Having established the model structures above, the general theory of abelian model~cat\-e\-go\-ries provides us with rich information about the associated homotopy categories. We recall (with appropriate references) the most important facts below.

\begin{prp}
  \label{prp:we}
    Assume that $Q$ satisfies the conditions in \stpref{Bbbk}, $\Bbbk$ is Gorenstein, and $\alg$ has finite projective/injective dimension as a $\Bbbk$-module. The two model categories
\begin{equation*}
  (\lMod{Q,\alg}, \text{projective model structure}) \quad \text{and} \quad (\lMod{Q,\alg}, \text{injective model structure})
\end{equation*}    
have the same weak equivalences, in fact, a morphism $\varphi$ in $\lMod{Q,\alg}$ is a weak equivalence in either of the two model structures if and only if it satisfies the next equivalent conditions:
\begin{eqc}
\item $\varphi = \pi\iota$ where $\iota$ is monic with $\Coker{\iota} \in \mathscr{E}$ and $\pi$ is epic with $\Ker{\pi} \in \mathscr{E}$.

\item $\varphi = \pi\iota$ where $\iota$ is monic with $\Coker{\iota} \in \lPrj{Q,\alg}$ and $\pi$ is epic with $\Ker{\pi} \in \mathscr{E}$.

\item $\varphi = \pi\iota$ where $\iota$ is monic with $\Coker{\iota} \in \mathscr{E}$ and $\pi$ is epic with $\Ker{\pi} \in \lInj{Q,\alg}$.
\end{eqc}
\end{prp}

\begin{proof}
As seen in the proof of \thmref{model-structures}, the projective model structure on $\lMod{Q,\alg}$ is given by the Hovey triple $(\mathcal{C}_\mathrm{p},\mathcal{W}_\mathrm{p},\mathcal{F}_\mathrm{p}) = ({}^\perp\mathscr{E},\mathscr{E},\lMod{Q,\alg})$, and the injective model structure is given by $(\mathcal{C}_\mathrm{i},\mathcal{W}_\mathrm{i},\mathcal{F}_\mathrm{i}) = (\lMod{Q,\alg},\mathscr{E},\mathscr{E}^\perp)$. Thus by definition, see \cite[Dfn.~5.1]{Hovey02}, the morphisms described in \eqclbl{ii}, respectively, \eqclbl{iii}, are precisely the weak equivalences in the projective, respectively, injective, model structure on $\lMod{Q,\alg}$. Evidently one has \eqclbl{ii}$\,\Rightarrow\,$\eqclbl{i} and \eqclbl{iii}$\,\Rightarrow\,$\eqclbl{i} as 
$\lPrj{Q,\alg} \subseteq \mathscr{E}$ and $\lInj{Q,\alg} \subseteq \mathscr{E}$. By \cite[Lem.~5.8]{Hovey02} a monic with cokernel in $\mathcal{W}_\mathrm{p} = \mathscr{E} = \mathcal{W}_\mathrm{i}$ is a weak equivalence in both model structures and so is an epic with kernel in $\mathscr{E}$. Hence, if \eqclbl{i} holds, then $\varphi$ is a composition of weak two equivalences in either model structure, and as already noted this means that \eqclbl{ii} and \eqclbl{iii} hold.
\end{proof}

\begin{dfn}
  \label{dfn:homotopy-cat}
    Assume that $Q$ satisfies the conditions in \stpref{Bbbk}, $\Bbbk$ is Gorenstein, and $\alg$ has finite projective/injective dimension as a $\Bbbk$-module. The \emph{$Q$-shaped derived category of $\alg$} is defined to be the homoptopy category of the model category $\lMod{Q,\alg}$ (see \thmref{model-structures} and \prpref{we}), that is,
\begin{equation*}
  \QSD{Q}{\alg} \,:=\, \operatorname{Ho}(\lMod{Q,\alg}) \,=\, \{\text{weak equivalences}\}^{-1}(\lMod{Q,\alg})\;.
\end{equation*}
\end{dfn}
  
\begin{thm}
  \label{thm:Frobenius}
    Assume that $Q$ satisfies the conditions in \stpref{Bbbk}, $\Bbbk$ is Gorenstein, and $\alg$ has finite projective/injective dimension as a $\Bbbk$-module. The category ${}^\perp\mathscr{E}$, respectively, $\mathscr{E}^\perp$, is Frobenius with  $\lPrj{Q,\alg}$, respectively, $\lInj{Q,\alg}$, as the class of pro-injective objects, and there are equivalences of categories,
\begin{equation*}
  \frac{{}^\perp\mathscr{E}}{\lPrj{Q,\alg}} \ \simeq \ 
  \QSD{Q}{\alg} \ \simeq \
  \frac{\mathscr{E}^\perp}{\lInj{Q,\alg}}\;.
\end{equation*}  
Here the leftmost, respectively, rightmost, category is the stable category of the Frobenius category ${}^\perp\mathscr{E}$, respectively, $\mathscr{E}^\perp$. In particular, $\QSD{Q}{\alg}$ is triangulated.
\end{thm}

\begin{proof}
  The cotorsion pairs associated with the Hovey triples that define the projective and injective model structures (see the proof of \thmref{model-structures}) are hereditary by \thmref{cotorsion-pairs-existence}. Hence the assertions follow from Gillespie \cite[Prop.~4.2 and Thm.~4.3]{MR3608936}.
\end{proof}

\section{(Co)homology}
\label{sec:cohomology}

Our goal in this section is to obtain tractable descriptions of the trivial objects and the weak equivalences in the projective and injective model structures on $\lMod{Q,\alg}$. At this point the only available descriptions of these objects and morphisms come from \dfnref{E} and \prpref{we}, which are not particularly enlightening. 

Condition \eqref{retraction} in \stpref{Bbbk} is called the Retraction Property. In \dfnref{srp} below we introduce a stron\-ger condition, called the Strong Retraction Property, which holds in most natural examples. This property allows us to define the pseudo-radical $\mathfrak{r}$ (\lemref{ideal}) and (co)homology functors $\cH[i]{q}$ and $\hH[i]{q}$ for every $q \in Q$ and $i \geqslant 0$ (\dfnref{cH-hH}). The main results in this section are \thmref[Theorems~]{E-characterization-hereditary} and \thmref[]{quiso} below. The first result \mbox{characterizes}~the~tri\-vial (= exact) objects as the objects with vanishing (co)homology; the second one characterizes the weak equivalences as the morphisms that are isomorphisms in (co)homology.

The proofs of these theorems require preparations that take up most of this section. \thmref[Theorems~]{E-characterization-hereditary} and \thmref[]{quiso} themselves are pro\-ved towards the end of the section.

\begin{thm}
 \label{thm:E-characterization-hereditary}
 Assume that the following hold.
 \begin{itemlist}
 \item $Q$ satisfies conditions \eqref{Homfin}, \eqref{locbd}, \eqref{Serre} in \stpref{Bbbk} and condition \eqmref{strong-retraction} in \dfnref{srp}.
  \item The pseudo-radical $\mathfrak{r}$ is nilpotent, that is, $\mathfrak{r}^N=0$ for some $N \in \mathbb{N}$.
 \item The ring $\Bbbk$ is noetherian and hereditary (e.g.~$\Bbbk=\mathbb{Z}$).
 \end{itemlist}
 For every $\Bbbk$-algebra $\alg$ and object $X \in \lMod{Q,\alg}$ the following conditions are equivalent:
\begin{eqc}
\item $X \in \mathscr{E}$.
\item $\cH[i]{q}(X)=0$ for every $q \in Q$ and $i>0$.
\item[\eqclbl{ii'}] $\cH{q}(X)=0$ for every $q \in Q$.
\item $\hH[i]{q}(X)=0$ for every $q \in Q$ and $i>0$.
\item[\eqclbl{iii'}] $\hH{q}(X)=0$ for every $q \in Q$.
\end{eqc}
\end{thm}

\begin{thm}
  \label{thm:quiso}
  Adopt the setup of \thmref{E-characterization-hereditary}. For every $\Bbbk$-algebra $\alg$ and morphism $\varphi$ in $\lMod{Q,\alg}$ the following conditions are equivalent:
  \begin{eqc}
  \item $\varphi$ is a weak equivalence.
  \item $\cH[i]{q}(\varphi)$ is an isomorphism for every $q \in Q$ and $i>0$.
  \item[\eqclbl{ii'}] $\cH{q}(\varphi)$ and $\cH[2]{q}(\varphi)$ are isomorphisms for every $q \in Q$.
  \item $\hH[i]{q}(\varphi)$ is an isomorphism for every $q \in Q$ and $i>0$.
  \item[\eqclbl{iii'}] $\hH{q}(\varphi)$ and $\hH[2]{q}(\varphi)$ are isomorphisms for every $q \in Q$.
  \end{eqc}
\end{thm}

Looking at \thmref{quiso} (in comparison with \thmref{E-characterization-hereditary}) one may wonder if a morphism $\varphi$ for which $\cH{q}(\varphi)$ (or 
$\hH{q}(\varphi)$) is an isomorphism for every $q \in Q$ necessarily~must be a weak equivalence. In general, this is \emph{not} the case; a counterexample is given in \exaref{counter}. However, it \emph{is} true under the assumption $\mathfrak{r}^2=0$, see \prpref{quiso-2}.
   
\begin{dfn}
\label{dfn:srp}
For a small $\Bbbk$-preadditive category $Q$, we consider the condition below.
\begin{rqmm}
\setcounter{rqmm}{3}
\item \label{eq:strong-retraction} \emph{Strong Retraction Property}: For every object $q \in Q$ the unit map $\Bbbk \to Q(q,q)$, given by $x \mapsto x\cdot\mathrm{id}_q$, has a $\Bbbk$-module retraction and there exists a collection $\{\mathfrak{r}_q\}_{q \in Q}$ of complements, i.e.~$\Bbbk$-modules $\mathfrak{r}_q$ such that 
\begin{equation*}
  Q(q,q) \,=\, (\Bbbk\cdot\mathrm{id}_q) \oplus \mathfrak{r}_q\;,
\end{equation*}  
compatible with composition in $Q$ in the following sense:
\begin{itemlist}
\item[($\dagger$)] $\mathfrak{r}_q \circ \mathfrak{r}_q \subseteq \mathfrak{r}_q$ for all $q \in Q$.
\item[($\ddagger$)] $Q(q,p) \circ Q(p,q) \subseteq \mathfrak{r}_p$ for all $p \neq q$ in $Q$.
\end{itemlist}
\end{rqmm}
\end{dfn}

\begin{rmk}
\label{rmk:not-unique}
If the unit map $\Bbbk \to Q(q,q)$ given by $x \mapsto x\cdot\mathrm{id}_q$ has a $\Bbbk$-module retraction, then there exists a $\Bbbk$-module decomposition $Q(q,q) = (\Bbbk\cdot\mathrm{id}_q) \oplus \mathfrak{r}_q$, but the complement $\mathfrak{r}_q$ it is not unique! Indeed, for every retraction (i.e.~left inverse) $\tau$ of the unit map, the kernel $\Ker{\tau}$ is a complement of $\Bbbk\cdot\mathrm{id}_q$ in $Q(q,q)$. The content of the Strong Retraction Property is that one can choose a collection $\{\mathfrak{r}_q\}_{q \in Q}$ of $\Bbbk$-submodules, where each $\mathfrak{r}_q$ is a complement of $\Bbbk\cdot\mathrm{id}_q$ in $Q(q,q)$, which is compatible with composition in $Q$ as described in ($\dagger$)~and~($\ddagger$). As the next example shows, this is not always possible, i.e.~the Strong Retraction Property does not always hold. When it does hold there may be more than one possible choice of such a compatible collection $\{\mathfrak{r}_q\}_{q \in Q}$. In this section, we consider the situation where $Q$ satisfies the Strong Retraction Property (\stpref{srp-setup}), and we tacitly assume that some fixed choice of a compatible collection $\{\mathfrak{r}_q\}_{q \in Q}$ has been made. Notice that the pseudo-radical  $\mathfrak{r}$ (\lemref{ideal}), the stalk functors $\stalkco{q}, \stalkcn{q}$ (\dfnref{stalk}), and the (co)homology functors $\cH[i]{q}, \cH[i]{q}$ (\dfnref{cH-hH}) all depend on this fixed choice. Of course, \thmref[Theorems~]{E-characterization-hereditary} and~\thmref[]{quiso} are true no matter which choice is made.
\end{rmk}

\begin{exa}
  \label{exa:PJ-counter}
  Even in the presence of the conditions Hom-finiteness, Local Boundedness, and Existence of a Serre Functor (conditions \eqref{Homfin}, \eqref{locbd}, and \eqref{Serre} in \stpref{Bbbk}), the Retraction Property (condition \eqref{retraction} in \stpref{Bbbk}) is strictly weaker than the Strong Retraction Property (condition \eqmref{strong-retraction} in \dfnref{srp}).
  
  Indeed, let $\Bbbk=\mathbb{R}$ and let $Q$ be the $\mathbb{R}$-preadditive category with one object, $*$, and hom set $Q(*,*)=\mathbb{C}$. Composition in $Q$ is given by multiplication of complex numbers. This~category satisfies all four conditions in \stpref{Bbbk} (the Serre functor is given by \mbox{$\Serre(*) = *$}). The unit map $\mathbb{R} \to Q(*,*)=\mathbb{C}$ is just the inclusion map and any $\mathbb{R}$-linear retraction of this map has the form $\tau^{\mspace{1mu}a} \colon \mathbb{C} \to \mathbb{R}$ given by $x+iy \mapsto x-ay$ for some $a \in \mathbb{R}$. The complement of $\mathbb{R}$ in $\mathbb{C}$ corresponding to this retraction is $\mathfrak{r}^{\mspace{1mu}a} = \Ker{(\tau^{\mspace{1mu}a})} = \mathrm{span}_{\mathbb{R}}\{a+i\} \subseteq \mathbb{C}$. For every $a \in \mathbb{R}$ we have $(a+i)^2 \notin \mathrm{span}_{\mathbb{R}}\{a+i\}$ and hence $\mathfrak{r}^{\mspace{1mu}a} \circ \mathfrak{r}^{\mspace{1mu}a} \nsubseteq \mathfrak{r}^{\mspace{1mu}a}$. Thus, $Q$ does not satisfy the Strong Retraction Property.
  
  From our viewpoint, this example is a bit ``exotic'', and as we shall see in \thmref[Theorems~]{A-dou} and \thmref[]{A-rep}, the Strong Retraction Property does hold in more ``natural'' examples. 
\end{exa}

\begin{rmk}
The Strong Retraction Property (condition \eqmref{strong-retraction} in \dfnref{srp}) implies that if two objects $p,q \in Q$ are isomorphic, then they are equal. Indeed, suppose for contradiction that there exists an isomorphism $f \colon p \to q$ with~$p \neq q$. By ($\ddagger$) one has  $\mathrm{id}_p = f^{-1}f \in \mathfrak{r}_p$, and clearly $\mathrm{id}_p \in \Bbbk\cdot\mathrm{id}_p$. By assumption we have $(\Bbbk\cdot\mathrm{id}_p) \cap \mathfrak{r}_p = 0$, and hence $\mathrm{id}_p = 0$. But this is impossible as the map $x \mapsto x\cdot\mathrm{id}_p$ is (even split) monic.
\end{rmk}

\begin{lem}
  \label{lem:ideal}
  Assume that $Q$ satisfies condition \eqmref{strong-retraction} in \dfnref{srp}.  The $\Bbbk$-modules
  \begin{equation*}
  \mathfrak{r}(p,q) \,=\, 
  \left\{\mspace{-5mu}
  \begin{array}{cl}
    \mathfrak{r}_q & \text{if $p = q$}
    \\
    Q(p,q) & \text{if $p \neq q$}
  \end{array}
  \right.
  \qquad (p,q \in Q)
\end{equation*}
constitute an ideal $\mathfrak{r}$ in $Q$; this ideal is called the \emph{\emph{pseudo-radical}} of $Q$.
\end{lem}

\begin{proof}
  We must prove the inclusions $Q(q,r) \circ \mathfrak{r}(p,q) \subseteq \mathfrak{r}(p,r)$ and $\mathfrak{r}(q,r) \circ Q(p,q) \subseteq \mathfrak{r}(p,r)$ for all $p,q,r \in Q$. For $p \neq r$ we have $\mathfrak{r}(p,r)=Q(p,r)$ and the inclusions are trivial. For $p=r$ the inclusions read $Q(q,p) \circ \mathfrak{r}(p,q) \subseteq \mathfrak{r}_p$ and $\mathfrak{r}(q,p) \circ Q(p,q) \subseteq \mathfrak{r}_p$. These inclusions hold for $p \neq q$ by ($\ddagger$). For $p=q$ the inclusions read $Q(q,q) \circ \mathfrak{r}_q \subseteq \mathfrak{r}_q$ and $\mathfrak{r}_q \circ Q(q,q) \subseteq \mathfrak{r}_q$. By using the equality $Q(q,q) = (\Bbbk\cdot\mathrm{id}_q) \oplus \mathfrak{r}_q$ both inclusions become $\Bbbk\cdot\mathfrak{r}_q + (\mathfrak{r}_q \circ \mathfrak{r}_q) \subseteq \mathfrak{r}_q$, and this inclusion holds by ($\dagger$).  
\end{proof}

\begin{exa}
Assume that $\Bbbk$ is a field and that each endomorphism $\Bbbk$-algebra $Q(q,q)$ is local with Jacobson radical $\mathfrak{r}_q := \operatorname{rad} Q(q,q)$. With this choice of $\mathfrak{r}_q$ ($q \in Q$) the \mbox{requirements} in the Strong Retraction Property are met and the ideal $\mathfrak{r}$ defined in \lemref{ideal} is precisely the \emph{radical}, $\operatorname{rad}_Q$, of the category $Q$ in the sense of Kelly \cite{Kelly64}. See e.g.~Assem, Simson,~and Skowro{\'n}ski \cite[Appn.~A.3, Prop.~3.5]{ASS1} for details. In more general situations,  $\mathfrak{r}$ need not~be the radical of $Q$, however, our terminology ``pseudo-radical'' for the ideal $\mathfrak{r}$ is inspired by the situation just mentioned.
\end{exa}

\begin{dfn}
  \label{dfn:stalk}
  Assume that $Q$ satisfies condition \eqmref{strong-retraction} in \dfnref{srp}. The \emph{stalk functors} (this name is explained by \lemref{stalk} below) at an object $q \in Q$ are defined to be
\begin{equation*}
  \stalkco{q}\,=\, Q(q,-)/\mathfrak{r}(q,-) \,\in\, \lMod{Q}
  \qquad \text{and} \qquad
  \stalkcn{q} \,=\, Q(-,q)/\mathfrak{r}(-,q) \,\in\, \rMod{Q}\,.
\end{equation*}
\end{dfn}

\noindent
Note that these stalk functors generalize the functors introduced in \cite[Setup 3.1]{MR4013804}.

\begin{lem}
  \label{lem:stalk}
Assume that $Q$ satisfies condition \eqmref{strong-retraction} in \dfnref{srp}. Let $q \in Q$ be given. For every object $p \in Q$ one has
\begin{equation*}
\stalkco{q}(p) \,=\, 
\left\{\mspace{-5mu}
\begin{array}{cl}
 \Bbbk & \text{if $p = q$}
 \\
 0 & \text{if $p \neq q$}\;.
\end{array}
\right.
\end{equation*}
For every morphism $f$ in $Q$ one has
\begin{equation*}
\stalkco{q}(f) \,=\, 
\left\{\mspace{-5mu}
\begin{array}{cl}
 x\cdot\mathrm{id}_\Bbbk & \text{if \ $f = x\cdot\mathrm{id}_q+g \in (\Bbbk\cdot\mathrm{id}_q) \oplus \mathfrak{r}_q = Q(q,q)$}
 \\
 0 & \text{otherwise}\;.
\end{array}
\right.
\end{equation*}
The (contravariant) functor $\stalkcn{q}$ can be described similarly. 
\end{lem}

\begin{proof}
  The assertions follow from the definition, \dfnref[]{stalk}, of the stalk functors. Note that one has $\stalkco{q}(q) = Q(q,q)/\mathfrak{r}_q \cong \Bbbk$ by the $\Bbbk$-module decomposition
  $Q(q,q) = (\Bbbk\cdot\mathrm{id}_q) \oplus \mathfrak{r}_q$.
\end{proof}

The stalk functors allow us to define a notion of (co)homology for objects in $\lMod{Q,\alg}$. 

\begin{dfn}
  \label{dfn:cH-hH}
  Assume that $Q$ satisfies condition \eqmref{strong-retraction} in \dfnref{srp}. 
  Let $X \in \lMod{Q,\alg}$. For $q \in Q$ and $i \geqslant 0$ we define the \emph{$i^\mathrm{th}$ (co)homology} of $X$ at $q$ as follows:
  \begin{equation*}
    \cH[i]{q}(X) \,=\, \Ext{Q}{i}{\stalkco{q}}{X}
    \qquad \text{and} \qquad
    \hH[i]{q}(X) \,=\, \Tor{Q}{i}{\stalkcn{q}}{X}\;.
  \end{equation*}
  Notice that $\cH[i]{q}$ and $\hH[i]{q}$ are functors $\lMod{Q,\alg} \to \lMod{\alg}$.
\end{dfn}

\begin{rmk}
  \label{rmk:cH-hH}
  The Ext and Tor functors in the definition above are the right and left derived functors of $\operatorname{Hom}_Q$ and $\ORt{}{}$ from \eqref{il} and \eqref{ic}, and the functors \smash{$\cH[*]{q} = \Ext{Q}{*}{\stalkco{q}}{-}$} and \smash{$\hH[*]{q} = \Tor{Q}{*}{\stalkcn{q}}{-}$} are computed via projective resolutions of $\stalkco{q} \in \lMod{Q}$ and $\stalkcn{q} \in \rMod{Q}$. Notice that we can also consider \dfnref{cH-hH} in the special case where $\alg=\Bbbk$. It follows that for $X \in \lMod{Q,\alg}$ one has
  \begin{equation*}
    \cH[i]{q}(X)^\natural \,=\, \Ext{Q}{i}{\stalkco{q}}{X}^\natural
    \,=\, \Ext{Q}{i}{\stalkco{q}}{X^\natural} \,=\, \cH[i]{q}(X^\natural)\;,
  \end{equation*}
  and similarly, $\hH[i]{q}(X)^\natural = \hH[i]{q}(X^\natural)$. So (co)homology commutes with the forgetful functor.
\end{rmk}

\begin{stp}
  \label{stp:srp-setup}
Throughout this section, we asume without further mention that the $\Bbbk$-pre\-ad\-ditive category $Q$ satisfies the Strong Retraction Property (condition \eqmref{strong-retraction} in \dfnref{srp}) such that the pseudo-radical $\mathfrak{r}$ (\lemref{ideal}), the stalk functors $\stalkco{q}, \stalkcn{q}$ (\dfnref{stalk}), and the (co)homology functors $\cH[i]{q}, \hH[i]{q}$ (\dfnref{cH-hH}) are defined. If further conditions (e.g.~the ones from \stpref{Bbbk}) need to be imposed on $Q$, this will be explicitly men\-tioned.
\end{stp}

Recall the functors $\Fq{q}$ and $\Gq{q}$ from \corref{adjoint-evaluation}. The following result provides examples of objects in $\lMod{Q,\alg}$ with trivial (co)homology.

\begin{lem}
  \label{lem:FGH}
  Assume that the category $Q$ satisfies condition \eqref{Homfin} in \stpref{Bbbk}. For every $p \in Q$ and $M \in \lMod{\alg}$ one has:
  \begin{prt}
  \item $\hH[i]{q}(\Fq{p}(M))=0$ for every $q \in Q$ and $i>0$.
  \item $\cH[i]{q}(\Gq{p}(M))=0$ for every $q \in Q$ and $i>0$.
  \end{prt}
\end{lem}

\begin{proof}
  \proofoftag{a} Let $\cpx{P} = \cdots \to P_1 \to P_0 \to 0$ be a projective of  resolution of $\stalkcn{q}$ in $\lMod{Q}$. By \dfnref{cH-hH} the homology of the complex $\ORt{\cpx{P}}{\Fq{p}(M)}$ computes $\hH[*]{q}(\Fq{p}(M))$. In the following computation, the first isomorphism holds by the definition of the functor $\Fq{p}$, the middle isomorphism follows from \lemref{associativity}, and the last isomorphism is established in the proof of \corref{adjoint-evaluation}:
\begin{equation*}
  \ORt{\cpx{P}}{\Fq{p}(M)} \,\cong\, \ORt{\cpx{P}}{(Q(p,-) \otimes_\Bbbk M)}
  \,\cong\,
  (\ORt{\cpx{P}}{Q(p,-)}) \otimes_\Bbbk M
  \,\cong\, \cpx{P}(p) \otimes_\Bbbk M\;.
\end{equation*}
The evaluation functor $\Eq{p}$ is exact. Its right adjoint $\Gq{p}$ is also exact by \corref{adjoint-evaluation}\prtlbl{b}, as $Q$ satisfies condition \eqref{Homfin} in \stpref{Bbbk}, so $\Eq{p}$ preserves projective objects. Hence the complex $\Eq{p}(\cpx{P}) = \cpx{P}(p)$ is a projective resolution of the $\Bbbk$-module $\Eq{p}(\stalkcn{q}) = \stalkcn{q}(p)$. But this $\Bbbk$-module is projective, as it is either $\Bbbk$ or $0$ by \lemref{stalk}, and thus the complex $\cpx{P}(p)$ is homotopy equivalent to either $\Bbbk$ or $0$. Thus the complex $\cpx{P}(p) \otimes_\Bbbk M$ is homotopy equivalent to either $M$ or $0$, in particular, it has zero homology in all degrees $i>0$.
  
  \proofoftag{b} Dual to the proof of part \prtlbl{a}.
\end{proof}

\begin{prp}
  \label{prp:adjoint-stalk}
For every $q \in Q$ there is an adjoint triple $(\Cq{q},\Sq{q},\Kq{q})$ as follows:
  \begin{equation*}
  \begin{gathered}
  \xymatrix@C=4pc{
    \lMod{\alg}
    \ar[r]^-{\Sq{q}} &
    \lMod{Q,\alg}
    \ar@/_1.8pc/[l]_-{\Cq{q}}
    \ar@/^1.8pc/[l]^-{\Kq{q}}    
  }  \qquad \textnormal{given by} \qquad
  {\setlength\arraycolsep{1.5pt}
   \renewcommand{\arraystretch}{1.2}
  \begin{array}{rcl}
  \Cq{q}(X) &=& \ORt{\stalkcn{q}}{X} \\
  \Sq{q}(M) &=& \stalkco{q} \otimes_\Bbbk M \\ 
  \Kq{q}(X) &=& \Hom{Q}{\stalkco{q}}{X}\;.
  \end{array}
  }
  \end{gathered}
  \end{equation*}
\end{prp}

\begin{proof}
  Apply \prpref{adjunctions-2} with $U=\stalkco{q}$ and $W=\stalkcn{q}$. As $\Sq{q} = \stalkco{q} \otimes_\Bbbk -$ its right adjoint is $\Hom{Q}{\stalkco{q}}{-}$. From the descriptions of $\stalkco{q}$ and $\stalkcn{q}$ in \lemref{stalk}, we see that $\Sq{q}$ is also given by $\Sq{q} \cong \Hom{\Bbbk}{\stalkcn{q}}{-}$, so its left adjoint is $\ORt{\stalkcn{q}}{-}$. 
\end{proof}

Note that $\cH[i]{q}$ is the $i^\mathrm{th}$ right derived functor of $\Kq{q}$ and  $\hH[i]{q}$ is the $i^\mathrm{th}$ left derived functor of $\Cq{q}$; see \dfnref{cH-hH}. We wish to give more hands-on descriptions of the functors $\Kq{q}$ and $\Cq{q}$, and to this end, we first find concrete projective presentations of $\stalkco{q}$ and $\stalkcn{q}$.

\begin{dfn}
  \label{dfn:sets}
  For each $q \in Q$ we define sets of morphisms in $Q$ as follows:
  \begin{equation*}
    \textstyle
    \setco{q} \,=\, \bigcup_{r \,\in\, Q} \,\mathfrak{r}(q,r)   
    \qquad \text{and} \qquad
    \setcn{q} \,=\, \bigcup_{p \,\in\, Q} \,\mathfrak{r}(p,q)\;.    
  \end{equation*}
  So $\setco{q}$, respectively, $\setcn{q}$, contains all morphisms in $\mathfrak{r}$ with domain, respectively, codomain, $q$.
\end{dfn}

\begin{lem}
  \label{lem:projective-presentations}
  For every $q \in Q$ there are two exact sequences:
  \begin{eqnarray}
    \label{eq:seq-co}    
    \xymatrix{
    \bigoplus_{g \,\in\, \setco{q}} Q(\cod{g},-)
    \ar[r]^-{\mapco{q}} & Q(q,-) \ar[r] & \stalkco{q} \ar[r] & 0
    } 
    &\text{in \ \ $\lMod{Q}$\,\phantom{.}}& 
    \\
    \label{eq:seq-cn}
    \xymatrix{
    \bigoplus_{f \,\in\, \setcn{q}} Q(-,\dom{f})
    \ar[r]^-{\mapcn{q}} & Q(-,q) \ar[r] & \stalkcn{q} \ar[r] & 0
    }
    &\text{in \ \ $\rMod{Q}$\,.}&
  \end{eqnarray}
  Here $\mapco{q}$ is the unique morphism in $\lMod{Q}$ given by $Q(g,-)$ on the component correspon\-ding to $g \in \setco{q}$ and $\mapcn{q}$ is the unique morphism in $\rMod{Q}$ given by $Q(-,f)$ on the component correspon\-ding to $f \in \setcn{q}$.
\end{lem}

\begin{proof}
  We only show exactness of the sequence \eqref{seq-co} as exactness of \eqref{seq-cn} is proved similar\-ly. By the definition, \dfnref[]{stalk}, of $\stalkco{q}$ we must argue that one has $\Im{\mapco{q}} = \mathfrak{r}(q,-)$, and by the definition of $\mapco{q}$ this is equivalent to proving that for every $r \in Q$ one has: 
\begin{equation}  
  \label{eq:sum}
  \textstyle\sum_{\,g \,\in\, \setco{q}} \Im{Q(g,r) \,=\, \mathfrak{r}(q,r)}\;.
\end{equation}  
    
  ``$\subseteq$'': For every $g \in \setco{q}$ one has $\Im{Q(g,r)} \subseteq \mathfrak{r}(q,r)$. Indeed, for any $h \in Q(\cod{g},r)$ the morphism $Q(g,r)(h) = hg$ is in $\mathfrak{r}(q,r)$ as $g$ belongs to $\mathfrak{r}(q,\cod{g})$ and $\mathfrak{r}$ is an ideal.
  
  ``$\supseteq$'': If $h \in \mathfrak{r}(q,r)$, then $h$ is in $\setco{q}$ and hence $\Im{Q(h,r)}$ is contained in the left-hand side of \eqref{sum}. Clearly $h$ is in the image of $Q(h,r) \colon Q(r,r) \to Q(q,r)$ as $Q(h,r)(\mathrm{id}_r) = h$.
\end{proof}

We now give some explicit formulae for the functors $\Kq{q}$ and $\Cq{q}$ from \prpref{adjoint-stalk}.

\begin{prp}
  \label{prp:K-C-formulae}
  For every $q \in Q$ and $X \in \lMod{Q,\alg}$ there are isomorphisms:
  \begin{align}
    \label{eq:K-formula}    
    \Kq{q}(X) &\,\cong\, 
    \textstyle
    \Ker{\big(\, X(q) \xrightarrow{\,\Mapco{q}{X}\,} \prod_{g \,\in\, \setco{q}}\, X(\cod{g}) \big) }
    \,=\
    \bigcap_{g \,\in\, \setco{q}} \Ker{X(g)}
    \\
    \label{eq:C-formula}
    \Cq{q}(X) &\,\cong\, 
    \textstyle
    \Coker{\big(\bigoplus_{f \,\in\, \setcn{q}}\, X(\dom{f}) \xrightarrow{\,\Mapcn{q}{X}} X(q)\big)} 
    \,=\ 
    X(q)\big/\big(\textstyle\sum_{f \,\in\, \setcn{q}} \Im{X(f)} \big)     
  \end{align}
 Here \smash{$\Mapco{q}{X}$} is the unique morphism whose coordinate map corresponding to $g \in \setco{q}$ is $X(g)$,~and $\Mapcn{q}{X}$ is the unique morphism given by $X(f)$ on the component correspon\-ding to $f \in \setcn{q}$.
\end{prp}

\begin{proof}
  The last equalities in \eqref{K-formula} and \eqref{C-formula} are evident from the definitions of \smash{$\Mapco{q}{X}$} and~\smash{$\Mapcn{q}{X}$}. To show the first iso\-mor\-phism in \eqref{K-formula}, apply the left exact functor $\Hom{Q}{?}{X}$ to the exact sequence \eqref{seq-co} to obtain the exact sequence
\begin{equation*}
  \xymatrix@C=1.3pc{
    0 \ar[r]
    &
    \Hom{Q}{\stalkco{q}}{X}
    \ar[r]
    & 
    \Hom{Q}{Q(q,-)}{X}
    \ar[rr]^-{\Hom{Q}{\mapco{q}}{X}}
    & &
    \Hom{Q}{\bigoplus_{g \,\in\, \setco{q}} Q(\cod{g},-)}{X}    
  }.
\end{equation*}
The functor $\Hom{Q}{?}{X}$ converts coproducts to products. Furthermore, for every \mbox{$p \in Q$}~one has $\Hom{Q}{Q(p,-)}{X} \cong X(p)$ by Yoneda's lemma. In view of this and the definition, \prpref[]{adjoint-stalk}, of the functor $\Kq{q}$, it follows that the sequence above is isomorphic to 
\begin{equation*}
  \xymatrix@C=1.5pc{
    0 \ar[r]
    &
    \Kq{q}(X)
    \ar[r]
    & 
    X(q)
    \ar[r]^-{\Mapco{q}{X}}
    & 
    \prod_{g \,\in\, \setco{q}}\, X(\cod{g})
  }.
\end{equation*}
Thus one has $\Kq{q}(X) \cong \Ker{\Mapco{q}{X}}$, as claimed. Similarly, by applying the right exact functor $\ORt{{?}}{X}$, which preserves coproducts, to the exact sequence \eqref{seq-cn} and using the isomorphism $\ORt{Q(-,p)}{X} \cong X(p)$ (see the proof of \corref{adjoint-evaluation}), it follows that $\Cq{q}(X) \cong \Coker{\Mapcn{q}{X}}$.
\end{proof}

Next we give a useful criterion to check if an object in $\lMod{Q,\alg}$ is zero.

\begin{prp}
  \label{prp:zero-criterion}
  Assume that the pseudo-radical $\mathfrak{r}$ is nilpotent, that is, $\mathfrak{r}^N=0$ for some $N \in \mathbb{N}$. For every $X \in \lMod{Q,\alg}$ the following conditions are equivalent:
  \begin{eqc}
  \item $X=0$, that is, $X(q)=0$ for every $q \in Q$.
  \item $\Kq{q}(X)=0$ for every $q \in Q$.
  \item $\Cq{q}(X)=0$ for every $q \in Q$.
  \end{eqc}
\end{prp}

\begin{proof}
  As $\Kq{q}(X)$ is a submodule and $\Cq{q}(X)$ is a quotient module of $X(q)$, see \prpref{K-C-formulae}, it is clear that \eqclbl{i} implies both \eqclbl{ii} and \eqclbl{iii}. We now show that \eqclbl{ii}\,$\Rightarrow$\,\eqclbl{i}; the proof of \eqclbl{iii}\,$\Rightarrow$\,\eqclbl{i} is similar. Assume \eqclbl{ii} and supppose for contradiction that $X \neq 0$. Choose any $q_1 \in Q$ with $X(q_1) \neq 0$. Since $\Kq{q_1}(X) = 0$ we have \smash{$X(q_1) \nsubseteq \Kq{q_1}(X) = \bigcap_{g \,\in\, \setco{q_1}} \Ker{X(g)}$}, so there exists some morphism $g_1 \colon q_1 \to q_2$ in $\mathfrak{r}$ for which $X(q_1) \nsubseteq \Ker{X(g_1)}$. This means that the map $X(g_1) \colon X(q_1) \to X(q_2)$ is non-zero. Since $0 \neq \Im{X(g_1)} \subseteq X(q_2)$ and $\Kq{q_2}(X) = 0$ we have \smash{$\Im{X(g_1)} \nsubseteq \Kq{q_2}(X) = \bigcap_{g \,\in\, \setco{q_2}} \Ker{X(g)}$}, so there exists some $g_2 \colon q_2 \to q_3$ in $\mathfrak{r}$ for which $\Im{X(g_1)} \nsubseteq \Ker{X(g_2)}$. This means that $X(g_2)X(g_1) = X(g_2g_1)$ is non-zero. By continuing in this manner, we obtain a sequence of morphisms,
\begin{equation*}
  \xymatrix@C=1.5pc{
    q_1 \ar[r]^-{g_1} & q_2 \ar[r]^-{g_2} & q_3 \ar[r]^-{g_3} & \cdots  
  },
\end{equation*} 
where each $g_i$ is in $\mathfrak{r}$ and $X(g_1), X(g_2g_1), X(g_3g_2g_1),\ldots$ are all non-zero. But $g_N\cdots g_2g_1$ is in $\mathfrak{r}^N=0$; in particular, $X(g_N\cdots g_2g_1)=0$, which is a contradiction.
\end{proof}

\begin{prp}
  \label{prp:stalk-ses}
  For every $X \in \lMod{Q,\alg}$ the following assertions hold:
  \begin{prt}
  \item There is a short exact sequence in $\lMod{Q,\alg}$,
  \begin{equation*}
      \textstyle
      0 
      \longrightarrow 
      \bigoplus_{q \in Q} \, \Sq{q}\Kq{q}(X)
      \longrightarrow
      X \longrightarrow X' \longrightarrow 0\;. 
  \end{equation*}
  Now assume that the next conditions are satisfied:
  \begin{itemlist}
  \item Each hom $\Bbbk$-module $Q(q,r)$ is finitely generated.
  \item For every $q \in Q$ the set $\operatorname{N}_+(q)=\{r \in Q \,|\,Q(q,r) \neq 0\}$ is finite.
  \item $\mathcal{G}$ is a class of left $\alg$-modules closed under extensions and submodules.
  \end{itemlist}
  If $X(q) \in \mathcal{G}$ for every $q \in Q$, then $\Kq{q}(X), X'(q) \in \mathcal{G}$ for every $q \in Q$.
  \item There is a short exact sequence in $\lMod{Q,\alg}$,
  \begin{equation*}
      \textstyle
      0 
      \longrightarrow 
      X''
      \longrightarrow
      X \longrightarrow \prod_{q \in Q} \, \Sq{q}\Cq{q}(X) \longrightarrow 0\;. 
  \end{equation*}
  Now assume that the next conditions are satisfied:
  \begin{itemlist}
  \item Each hom $\Bbbk$-module $Q(p,q)$ is finitely generated.
  \item For every $q \in Q$ the set $\operatorname{N}_-(q)=\{p \in Q \,|\,Q(p,q) \neq 0\}$ is finite.
  \item $\mathcal{H}$ is a class of left $\alg$-modules closed under extensions and quotient modules.
  \end{itemlist}
  If $X(q) \in \mathcal{H}$ for every $q \in Q$, then $\Cq{q}(X), X''(q) \in \mathcal{H}$ for every $q \in Q$.
  \end{prt}
\end{prp}

\begin{proof}
  \proofoftag{a} For $q \in Q$ we consider the counit $\eta_q^X \colon \Sq{q}\Kq{q}(X) \to X$ of the adjunction $(\Sq{q}, \Kq{q})$ from \prpref{adjoint-stalk}. By the universal property of the coproduct, we get an induced morphism \smash{$\eta^X \colon \bigoplus_{q \in Q} \Sq{q}\Kq{q}(X) \to X$} whose cokernel we denote by $X' = \Coker{\eta^X}$. To establish the asserted short exact sequence, we need to show that $\eta^X$ is monic. For $p \in Q$ we have 
\begin{equation*}  
  \Sq{q}\Kq{q}(X)(p) \,=\, \stalkco{q}(p) \otimes_\Bbbk \Kq{q}(X)
  \,=\,
  \left\{\mspace{-5mu}
\begin{array}{cl}
 \Kq{p}(X) & \text{if $p = q$}
 \\
 0 & \text{if $p \neq q$}
\end{array}
\right.
\end{equation*}  
by the definitions of $\Sq{q}$ and $\stalkco{q}$; see \prpref{adjoint-stalk} and \lemref{stalk}. Hence $\eta^X(p)$ is the ca\-no\-ni\-cal map $\Kq{p}(X) \to X(p)$, which is monic by \prpref{K-C-formulae}. Thus, $\eta^X$ is monic.

Next assume that the three conditions (marked with bullets in the Proposition) are satisfied and that $X(q) \in \mathcal{G}$ for every $q \in Q$.

\prpref{K-C-formulae} shows that $\Kq{q}(X)$ is a submodule of $X(q)$. Since one has $X(q) \in \mathcal{G}$ and $\mathcal{G}$ is closed under submodules, it follows that $\Kq{q}(X) \in \mathcal{G}$. 

To prove $X'(q) \in \mathcal{G}$ for every $q \in Q$, we argue as follows. Fix $q \in Q$. The set $\operatorname{N}_+(q)$ is finite, say, $\operatorname{N}_+(q)=\{r_1,\ldots,r_m\}$. In particular, $\mathfrak{r}(q,r)=0$ if $r \notin \{r_1,\ldots,r_m\}$. Thus, from the definition, \dfnref[]{sets}, of the set $\setco{q}$ and the formula for $\Kq{q}(X)$ given in \prpref{K-C-formulae}~we~get
\begin{equation}
\label{eq:K-intersection-1}
  \textstyle
  \Kq{q}(X) \,=\, \bigcap_{g \,\in\, \setco{q}} \Ker{X(g)}
  \,=\, 
  \big(\bigcap_{g \,\in\, \mathfrak{r}(q,r_1)} \Ker{X(g)}\big)
  \,\cap\, \ldots \,\cap\, 
  \big(\bigcap_{g \,\in\, \mathfrak{r}(q,r_m)} \Ker{X(g)}\big)\,.
\end{equation}
Each $\Bbbk$-module $Q(q,r_i)$ is finitely generated. Since $\mathfrak{r}(q,r_i)$ is a direct summand in $Q(q,r_i)$, see \lemref{ideal} and condition \eqmref{strong-retraction} in \dfnref{srp}, it follows that $\mathfrak{r}(q,r_i)$ is finitely generated as well, say, $\mathfrak{r}(q,r_i) = \Bbbk\mspace{1mu} g_{i1} + \cdots + \Bbbk\mspace{1mu} g_{i\ell(i)}$. Evidently, one now has
\begin{equation}
\label{eq:K-intersection-2}
  \textstyle
  \bigcap_{g \,\in\, \mathfrak{r}(q,r_i)} \Ker{X(g)}
  \,=\,
  \Ker{X(g_{i1})} \,\cap\, \ldots \,\cap\, \Ker{X(g_{i\ell(i)})}\;.
\end{equation}
Combining \eqref{K-intersection-1} and \eqref{K-intersection-2} we see that there exist finitely many morphisms $g_j \colon q \to \cod{g_j}$, $j=1,\ldots,n$, in the ideal $\mathfrak{r}$ such that $\Kq{q}(X) = \Ker{X(g_1)} \,\cap\, \ldots \,\cap\, \Ker{X(g_n)}$, and hence
\begin{equation*}
  X'(q) \,=\, X(q)\big/\Kq{q}(X) \,=\, 
  X(q)\big/\big(\Ker{X(g_1)} \,\cap\, \ldots \,\cap\, \Ker{X(g_n)}\big)\;,
\end{equation*}  
where the first equality follows from the definition of $X'$; cf.~the first part of the proof. To finish the proof we argue that given finitely many morphisms
$g_j \colon q \to \cod{g_j}$, $j=1,\ldots,n$, in $Q$ (they need not belong to the ideal $\mathfrak{r}$), the module
\begin{equation}
  \label{eq:module}
  X(q)\big/\big(\Ker{X(g_1)} \,\cap\, \ldots \,\cap\, \Ker{X(g_n)}\big)
\end{equation}  
belongs to $\mathcal{G}$. 
We use induction on $n$. For $n=0$ the intersection $\Ker{X(g_1)} \cap \ldots \cap \Ker{X(g_n)}$ is taken over the empty set, so the module in \eqref{module} is the zero module, which is in $\mathcal{G}$.

Next let $n>0$ and set $L = \Ker{X(g_1)} \,\cap\, \ldots \,\cap\, \Ker{X(g_{n-1})}$. By the induction hypothesis we have $X(q)/L \in \mathcal{G}$. We must show that $X(q)/(L \cap \Ker{X(g_n)})$ belongs to $\mathcal{G}$. To this end, consider the short exact sequence
  \begin{equation}
    \label{eq:ses-for-Xprime}
    0 
    \longrightarrow \frac{\Ker{X(g_n)}}{L \cap \Ker{X(g_n)}}
    \longrightarrow \frac{X(q)}{L \cap \Ker{X(g_n)}}
    \longrightarrow \frac{X(q)}{\Ker{X(g_n)}}
    \longrightarrow 0\;.
  \end{equation}
  The module $X(q)/\Ker{X(g_n)}$ is isomorphic to the submodule $\Im{X(g_n)}$ of $X(\cod{g_n})$. Since $X(\cod{g_n}) \in \mathcal{G}$ and $\mathcal{G}$ is closed under submodules, it follows that $X(q)/\Ker{X(g_n)} \in \mathcal{G}$. Noether's second isomorphism theorem shows that $\Ker{X(g_n)}/(L \cap \Ker{X(g_n)})$ is isomorphic to $(L + \Ker{X(g_n)})/L$, which is a submodule of $X(q)/L$. Since $X(q)/L$ belongs to $\mathcal{G}$, so does $\Ker{X(g_n)}/(L \cap \Ker{X(g_n)})$. Consequently, the short exact sequence \eqref{ses-for-Xprime} shows that $X(q)/(L \cap \Ker{X(g_n)})$ is an extension of modules from $\mathcal{G}$, and as $\mathcal{G}$ is closed under extensions, we get $X(q)/(L \cap \Ker{X(g_n)}) \in \mathcal{G}$.
  
\proofoftag{b} Dual to the proof of part \prtlbl{a}.  
\end{proof}

We will apply \prpref{stalk-ses} successively in the following construction.

\begin{con}
  \label{con:X-ell}
  Let $X \in \lMod{Q,\alg}$ be given.
  \begin{prt}
  \item Define $X^0, X^1, X^2,\ldots$ in $\lMod{Q,\alg}$ as follows. Set $X^0 = X$. Having defined $X^\ell$, let $X^{\ell+1}$ be the third term in the next short exact sequence coming from \prpref{stalk-ses}\prtlbl{a}:
  \begin{equation*}
      \textstyle
      0 
      \longrightarrow 
      \bigoplus_{q \in Q} \, \Sq{q}\Kq{q}(X^\ell)
      \longrightarrow
      X^\ell \longrightarrow X^{\ell+1} \longrightarrow 0\;. 
  \end{equation*}
  \item Define $X_0, X_1, X_2,\ldots$ in $\lMod{Q,\alg}$ as follows. Set $X_0 = X$. Having defined $X_\ell$, let $X_{\ell+1}$ be the first term in the next short exact sequence coming from \prpref{stalk-ses}\prtlbl{b}:
  \begin{equation*}
      \textstyle
      0 
      \longrightarrow 
      X_{\ell+1}
      \longrightarrow
      X_\ell \longrightarrow 
      \prod_{q \in Q} \, \Sq{q}\Cq{q}(X_\ell) \longrightarrow 0\;. 
  \end{equation*}
  \end{prt}
  Note that in \prtlbl{a} we use superscripts but in \prtlbl{b} we use subscripts on the constructed objects.
\end{con}

Consider the sequence of objects $X^0, X^1, X^2,\ldots$ from part \prtlbl{a} in the construction above. For each $q \in Q$ and $\ell \geqslant 0$ we can consider the modules $X^\ell(q)$ and $\Kq{q}(X^\ell)$. Below we show how these modules can be computed directly from $X^0=X$ (and from $q$ and $\ell$).

\begin{dfn}
  \label{dfn:ell}
  For $\ell \geqslant 0$ let $\mathfrak{r}^\ell$ be the $\ell^\mathrm{th}$ power of the pseudo-radical ideal $\mathfrak{r}$. We set 
  \begin{equation*}
    \textstyle
    \setco{q}^\ell \,=\, \bigcup_{r \,\in\, Q} \,\mathfrak{r}^\ell(q,r)    
    \qquad \text{and} \qquad
    \setcn{q}^\ell \,=\, \bigcup_{p \,\in\, Q} \,\mathfrak{r}^\ell(p,q)\;.   
  \end{equation*}
  For $X \in \lMod{Q,\alg}$ we also set
  \begin{equation*}
    \textstyle  
    \Kq{q}^\ell(X) \,=\, 
    \bigcap_{g \,\in\, \setco{q}^\ell} \Ker{X(g)}
    \qquad \text{and} \qquad
    \Cq{q}^\ell(X) \,=\, 
    X(q)\big/\big(\textstyle\sum_{f \,\in\, \setcn{q}^\ell} \Im{X(f)} \big)\;.
  \end{equation*}
\end{dfn}

\begin{rmk}
  \label{rmk:ell}
  For the sets $\setco{q}^\ell$ and the functors $\Kq{q}^\ell$ we observe the following:
  \begin{itemlist}
  \item $\setco{q}^0 \supseteq \setco{q}^1 \supseteq \setco{q}^2 \supseteq \cdots$ and hence $\Kq{q}^0(X) \rightarrowtail \Kq{q}^1(X) \rightarrowtail \Kq{q}^2(X) \rightarrowtail \cdots \rightarrowtail X(q)$.
  \item As $\mathfrak{r}^0(-,-)=Q(-,-)$ we have $\mathrm{id}_q \in \setco{q}^0$ and hence $\Kq{q}^0=0$.
  \item $\setco{q}^1=\setco{q}$ (see \dfnref{sets}) and $\Kq{q}^1=\Kq{q}$ (see \prpref{K-C-formulae}).
  \end{itemlist}
  Dually, for the sets $\setcn{q}^\ell$ and the functors $\Cq{q}^\ell$ one has:
  \begin{itemlist}
  \item $\setcn{q}^0 \supseteq \setcn{q}^1 \supseteq \setcn{q}^2 \supseteq \cdots$ and hence $X(q) \twoheadrightarrow \cdots \twoheadrightarrow \Cq{q}^2(X) \twoheadrightarrow \Cq{q}^1(X) \twoheadrightarrow \Cq{q}^0(X)$.
  \item $\Cq{q}^0=0$.
  \item $\setcn{q}^1=\setcn{q}$ and $\Cq{q}^1=\Cq{q}$.
  \end{itemlist}
\end{rmk}  

\begin{lem}
  \label{lem:ell-computation}
  Adopt the notation from \conref{X-ell}. For all $q \in Q$ and $\ell \geqslant 0$ one~has:
\begin{prt}  
\item $X^\ell(q) \,=\, \Coker{(\Kq{q}^\ell(X) \rightarrowtail X(q))}$ \,and\, $\Kq{q}(X^\ell) \,=\, \Coker{(\Kq{q}^\ell(X) \rightarrowtail \Kq{q}^{\ell+1}(X))}$.
\item $X_\ell(q) \,=\, \Ker{(X(q) \twoheadrightarrow \Cq{q}^\ell(X))}$ \,and\, $\Cq{q}(X_\ell) \,=\, \Ker{(\Cq{q}^{\ell+1}(X) \twoheadrightarrow \Cq{q}^\ell(X))}$.
\end{prt}  
\end{lem}

\begin{proof}
  \proofoftag{a} By induction on $\ell$. As $X^0=X$, $\Kq{q}^0 = 0$, and $\Kq{q}^1 = \Kq{q}$, see \rmkref{ell}, the formulae hold for $\ell=0$. Assume that they hold for some $\ell$. In the next commutative diagram, the upper row is exact by the induction hypothesis \mbox{$\Kq{q}(X^\ell) = \Coker{(\Kq{q}^\ell(X) \rightarrowtail \Kq{q}^{\ell+1}(X))}$}.
\begin{equation*}
  \xymatrix{
  0 \ar[r] & \Kq{q}^\ell(X) \ar@{>->}[d] \ar[r] & \Kq{q}^{\ell+1}(X) \ar@{>->}[d] \ar[r] & \Kq{q}(X^\ell) \ar[d] \ar[r] & 0
  \\
  0 \ar[r] & X(q) \ar[r]^-{=} & X(q) \ar[r] & 0 \ar[r] & 0  
  }
\end{equation*}  
Applying the Snake Lemma and the induction hypothesis $X^\ell(q) = \Coker{(\Kq{q}^\ell(X) \rightarrowtail X(q))}$ to this diagram, we get the short exact sequence
\begin{equation*}
  0 \longrightarrow 
  \Kq{q}(X^\ell) \longrightarrow
  X^\ell(q) \longrightarrow
  \Coker{(\Kq{q}^{\ell+1}(X) \rightarrowtail X(q))} \longrightarrow 0\;.
\end{equation*}
By definition of $X^{\ell+1}$ (see also the first part of the proof of \prpref{stalk-ses}), the module $X^{\ell+1}(q)$ is precisely the cokernel of the homomorphism $\Kq{q}(X^\ell) \rightarrowtail X^\ell(q)$, so we conclude that $X^{\ell+1}(q) = \Coker{(\Kq{q}^{\ell+1}(X) \rightarrowtail X(q))}$. Hence, the first of the asserted formulae hold for $\ell+1$. For a morphism $g \colon q \to \cod{g}$ in $Q$ the homomorphism 
\begin{equation*}   
  \xymatrix@C=4pc{
    X(q) \big/ \Kq{q}^{\ell+1}(X) \,=\, X^{\ell+1}(q) \ar[r]^-{X^{\ell+1}(g)} & 
    X^{\ell+1}(\cod{g}) \,=\, X(\cod{g}) \big/ \Kq{\cod{g}}^{\ell+1}(X)
  }  
\end{equation*}  
    is induced by $X(g) \colon X(q) \to X(\cod{g})$. Consequently, there is an equality
\begin{equation*}
  \Ker{X^{\ell+1}(g)} \,=\, \big(\, X(g)^{-1}\big(\Kq{\cod{g}}^{\ell+1}(X)\big) \,\big) \big/ \Kq{q}^{\ell+1}(X)\;.
\end{equation*}  
It follows that 
\begin{equation*}
  \textstyle
  \Kq{q}(X^{\ell+1}) \,=\, \bigcap_{g \,\in\, \setco{q}} \Ker{X^{\ell+1}(g)}
  \,=\,
  \big(\, \bigcap_{g \,\in\, \setco{q}} X(g)^{-1}\big(\Kq{\cod{g}}^{\ell+1}(X)\big) \,\big) 
  \big/ \Kq{q}^{\ell+1}(X)\;.
\end{equation*}
To finish the proof we must show that the numerator in the last expression above is equal to $\Kq{q}^{\ell+2}(X)$. To this end, consider the following computation:
\begin{align*}
  \textstyle
  \bigcap_{g \,\in\, \setco{q}} X(g)^{-1}\big(\Kq{\cod{g}}^{\ell+1}(X)\big)
  &\,=\  
  \textstyle
  \bigcap_{g \,\in\, \setco{q}} X(g)^{-1}
  \big(\, \bigcap_{h \,\in\, \setco{\cod{g}}^{\ell+1}} \Ker{X(h)} \,\big)
  \\
  &\,=\
  \textstyle
  \bigcap_{g \,\in\, \setco{q}} \ \bigcap_{h \,\in\, \setco{\cod{g}}^{\ell+1}} 
  X(g)^{-1}(\Ker{X(h)})
  \\
  &\,=\
  \textstyle
  \bigcap_{g \,\in\, \setco{q}} \ \bigcap_{h \,\in\, \setco{\cod{g}}^{\ell+1}} 
  \Ker{X(hg)} \,=:\, M\;.
\end{align*}
For $g \in \setco{q}$ and \smash{$h \in \setco{\cod{g}}^{\ell+1}$} we have by definition $g \in \mathfrak{r}(q,\cod{g})$ and $h \in \mathfrak{r}^{\ell+1}(\cod{g},\cod{h})$ and thus $hg \in \mathfrak{r}^{\ell+2}(q,\cod{h})$, that is, $hg \in \setco{q}^{\ell+2}$. Hence there is an inclusion,
\begin{equation}
  \label{eq:M}
\textstyle
M \ \supseteq\ \bigcap_{k \,\in\, \setco{q}^{\ell+2}} \,\Ker{X(k)} \,=\, \Kq{q}^{\ell+2}(X)\;.
\end{equation}
On the other hand, every morphism \smash{$k \in \setco{q}^{\ell+2}$} belongs to $\mathfrak{r}^{\ell+2}(q,r)$ for some $r \in Q$. By definition of the ideal $\mathfrak{r}^{\ell+2}$, the morphism $k$ is therefore a $\Bbbk$-linear combination of morphisms of the form $g_{\ell+2}\cdots g_2g_1$, i.e.~compositions,
\begin{equation*}
  \xymatrix@C=1.5pc{
    q \ar[r]^-{g_1} & p_1 \ar[r]^-{g_2} & p_2 \ar[r] & 
    \ \cdots \ \ar[r] 
    & p_{\ell+1} \ar[r]^-{g_{\ell+2}} & r
  },
\end{equation*}
where each morphism $g_i$ is in $\mathfrak{r}$. With $h=g_{\ell+2}\cdots g_2$ and $g=g_1$ we have $g_{\ell+2}\cdots g_2g_1 = hg$ where $g \in \setco{q}$ and \smash{$h \in \setco{p_1}^{\ell+1} = \setco{\cod{g}}^{\ell+1}$}. Therefore, $k$ is a $\Bbbk$-linear combination of morphisms of the form $hg$ where $g \in \setco{q}$ and \smash{$h \in \setco{\cod{g}}^{\ell+1}$}. Thus, equality holds in \eqref{M}, as desired.

\proofoftag{b} Dual to the proof of part \prtlbl{a}.
\end{proof}

\begin{thm}
  \label{thm:E-characterization}
 Assume that the following hold.
 \begin{itemlist}
 \item $Q$ satisfies condition \eqref{Homfin}, \eqref{locbd}, \eqref{Serre} in \stpref{Bbbk} and condition \eqmref{strong-retraction} in \dfnref{srp}.
  \item The pseudo-radical $\mathfrak{r}$ is nilpotent, that is, $\mathfrak{r}^N=0$ for some $N \in \mathbb{N}$.
 \item The ring $\Bbbk$ is $1$-Gorenstein, that is, $\Bbbk$ is noetherian and $\id{\Bbbk}{\Bbbk} \leqslant 1$.
 \item The $\Bbbk$-algebra $\alg$ has finite projective/injective dimension as a $\Bbbk$-module.
 \end{itemlist}
 For every $X \in \lMod{Q,\alg}$ the following conditions are equivalent.
\begin{eqc}
\item $X \in \mathscr{E}$ (see \dfnref{E}).
\item $\Ext{Q}{i}{\Sq{q}(G)}{X^\natural}=0$ for every $G \in \lGPrj{\Bbbk}$, $q \in Q$, and $i>0$.
\item[\eqclbl{ii'}] $\Ext{Q}{1}{\Sq{q}(G)}{X^\natural}=0$ for every $G \in \lGPrj{\Bbbk}$ and $q \in Q$.
\item $\Ext{Q}{i}{X^\natural}{\Sq{q}(H)}=0$ for every $H \in \lGInj{\Bbbk}$, $q \in Q$, and $i>0$.
\item[\eqclbl{iii'}] $\Ext{Q}{1}{X^\natural}{\Sq{q}(H)}=0$ for every $H \in \lGInj{\Bbbk}$ and $q \in Q$.
\end{eqc}
\end{thm}

\begin{proof}
Let $\mathcal{Y}$ be the class of $Y \in \lMod{Q}$ for which the $\Bbbk$-module $Y(q)$ is Gorenstein projective for every $q \in Q$. By our  assumptions and \thmref{G-description} we have $\mathcal{Y}=\lGPrj{Q}$, so $(\mathcal{Y},\lFin{Q})$ is a hereditary cotorsion pair in $\lMod{Q}$ by 
by \thmref{Q-Gorenstein}\,/\,\rmkref{opposite} and \thmref{G-cotorsion-pairs}. We now prove the equivalence between \eqclbl{i}, \eqclbl{ii}, and \eqclbl{ii'} in the theorem.

\proofofimp{i}{ii} If $X \in \mathscr{E}$ then, by \dfnref{E}, one has $X^\natural \in \lFin{Q}$. As $(\mathcal{Y},\lFin{Q})$ is a hereditary cotorsion pair, this means that $\Ext{Q}{i}{Y}{X^\natural}=0$ for every $Y \in \mathcal{Y}$ and $i>0$. Since $\Sq{q}(G) \in \mathcal{Y}$ for every $G \in \lGPrj{\Bbbk}$, we conclude that \eqclbl{ii} holds.

\proofofimp{ii}{ii'} This implication is trivial.

\proofofimp{ii'}{i} Since $(\mathcal{Y},\lFin{Q})$ is a cotorsion pair, we know that $\mathcal{Y}^\perp = \lFin{Q}$. The assumption in \eqclbl{ii'} is that $X^\natural$ is in $\{\, \Sq{q}(G) \,|\ G \in \lGPrj{\Bbbk} \,\}^\perp$. We need to show $X \in \mathscr{E}$, that is, $X^\natural \in \lFin{Q}$, so the desired implication follows if we can prove the next inclusion in $\lMod{Q}$:
\begin{equation}
  \label{eq:inclusion}
  \{\, \Sq{q}(G) \,|\ G \in \lGPrj{\Bbbk} \,\}^\perp \,\subseteq\, \mathcal{Y}^\perp\;.
\end{equation}
To this end, consider any $Y \in \mathcal{Y}$. Applying \conref{X-ell}\prtlbl{a} (with $\alg=\Bbbk$) to this object, we get objects $Y^0,Y^1,Y^2,\ldots$ with $Y^0=Y$ and short exact sequences in $\lMod{Q}$,
   \begin{equation*}
    \xymatrix@R=0.5pc@C=1.5pc{
      \text{(0)} &
      0 \ar[r] & \bigoplus_{q \in Q} \, \Sq{q}\Kq{q}(Y^0) \ar[r] &
      Y^0 \ar[r] & Y^1 \ar[r] & 0
      \\
      \text{(1)} & 0 \ar[r] & \bigoplus_{q \in Q} \, \Sq{q}\Kq{q}(Y^1) \ar[r] &
      Y^1 \ar[r] & Y^2 \ar[r] & 0
      \\
      \vdots & {} & \vdots & \vdots & \vdots & {}
    } 
  \end{equation*}
  As $\Bbbk$ is $1$-Gorenstein, the class $\mathcal{G}=\lGPrj{\Bbbk}$ is closed under submodules. Indeed, by \cite[Thm.~10.2.14]{rha} every $\Bbbk$-module has Gorenstein projective dimension $\leqslant 1$, so the claim follows from \cite[Thm.~2.20]{HHl04a}. The class $\mathcal{G}=\lGPrj{\Bbbk}$ is always closed under extensions by \cite[Thm.~2.5]{HHl04a}. By assumption, $Y(q) \in \lGPrj{\Bbbk}$ for every $q \in Q$, so it follows from successive applications of \prpref{stalk-ses}\prtlbl{a} that $\Kq{q}(Y^\ell),Y^\ell(q) \in \lGPrj{\Bbbk}$
for every $q \in Q$ and $\ell \geqslant 0$. In particular, every  $\Sq{q}\Kq{q}(Y^\ell)$ belongs to the class $\{\, \Sq{q}(G) \,|\ G \in \lGPrj{\Bbbk} \,\}$.

Consider the sets $\setco{q}^\ell$ and the functors $\Kq{q}^\ell$ from \dfnref{ell}. As $\mathfrak{r}^N=0$ it follows that for every $\ell \geqslant N$ the set \smash{$\setco{q}^\ell$} consists of all zero morphisms in $Q$ with domain $q$, and therefore $\Kq{q}^\ell(Y)=Y(q)$. Thus \lemref{ell-computation}\prtlbl{a} yields that for every $q \in Q$ one has:
\begin{equation*}
  \Kq{q}(Y^N) \ = \ 
  \Coker{(\Kq{q}^N(Y) \rightarrowtail \Kq{q}^{N+1}(Y))} \ = \
  \Coker{(Y(q) \xrightarrow{\mathrm{id}} Y(q))} \ = \ 0\;.
\end{equation*}
Now \prpref{zero-criterion} implies $Y^N=0$ and thus the short exact sequence number ($N-1$) in the display above shows that $\bigoplus_{q \in Q} \, \Sq{q}\Kq{q}(Y^{N-1}) \cong Y^{N-1}$. Consequently, the short exact sequence number ($N-2$) reads:
   \begin{equation*}
    \xymatrix@C=1.0pc{
      \text{($N-2$)} &
      0 \ar[r] & \bigoplus_{q \in Q} \, \Sq{q}\Kq{q}(Y^{N-2}) \ar[r] &
      Y^{N-2} \ar[r] & \bigoplus_{q \in Q} \, \Sq{q}\Kq{q}(Y^{N-1}) \ar[r] & 0
    }.
  \end{equation*}
As already mentioned, every  $\Sq{q}\Kq{q}(Y^\ell)$ belongs to the class $\{\, \Sq{q}(G) \,|\ G \in \lGPrj{\Bbbk} \,\}$. Hence, if an object $U \in \lMod{Q}$ belongs to the left-hand side in \eqref{inclusion}, then 
\begin{equation*}
  \textstyle
  \Ext{Q}{1}{\bigoplus_{q \in Q} \, \Sq{q}\Kq{q}(Y^\ell)}{U} \,=\,
  \prod_{q \in Q}\Ext{Q}{1}{\Sq{q}\Kq{q}(Y^\ell)}{U} \,=\, 0
\end{equation*}
for all $\ell \geqslant 0$. Thus the short exact sequence number ($N-2$) shows that $\Ext{Q}{1}{Y^{N-2}}{U}=0$. Then the short exact sequence number ($N-3$) shows that \smash{$\Ext{Q}{1}{Y^{N-3}}{U}=0$}. Continuing in this way, we arrive at the conclusion that \smash{$\Ext{Q}{1}{Y^{0}}{U}=0$}; and since $Y^0=Y$, which was an arbitrary object in $\mathcal{Y}$, we get $U \in \mathcal{Y}^\perp$. This proves the desired inclusion \eqref{inclusion}.

The equivalence between \eqclbl{i}, \eqclbl{iii}, and \eqclbl{iii'} is proved by arguments dual to the ones above, using \prpref[Propositions~]{zero-criterion} and \prpref[]{stalk-ses}\prtlbl{b}, \conref{X-ell}\prtlbl{b}, and \lemref{ell-computation}\prtlbl{b}. In this case one needs to know that the class $\mathcal{H}=\lGInj{\Bbbk}$ is closed under quotient modules. As $\Bbbk$ is $1$-Gorenstein this follows from \cite[Thm.~10.1.13(1)]{rha} and \cite[Thm.~2.22]{HHl04a}. The class $\mathcal{H}=\lGInj{\Bbbk}$ is closed under extensions by \cite[Thm.~2.6]{HHl04a} (see also \cite[Thm.~10.1.4]{rha}).
\end{proof}

At this point we are finally ready to prove the main theorems of this section.

\begin{proof}[Proof of \thmref{E-characterization-hereditary}]
  As $\Bbbk$ is noetherian and hereditary it is, in particular, $1$-Gorenstein and $\alg$ has projective/injective dimension at most $1$ as a $\Bbbk$-module. Thus we can apply \thmref{E-characterization}. Furthermore, in this case the class $\lGPrj{\Bbbk}$ (respectively, $\lGInj{\Bbbk}$) of Gorenstein projective (respectively, Gorenstein injective) $\Bbbk$-modules coincides with the class $\lPrj{\Bbbk}$ (respectively, $\lInj{\Bbbk}$) of projective (respectively, injective) $\Bbbk$-modules by \cite[Props.~10.2.3 and 10.1.2]{rha}. By \prpref{adjoint-stalk} and its proof we have $\Sq{q} = \stalkco{q} \otimes_\Bbbk - \cong \Hom{\Bbbk}{\stalkcn{q}}{-}$, and thus conditions \eqclbl{ii}, \eqclbl{ii'}, \eqclbl{iii}, and \eqclbl{iii'} in \thmref{E-characterization} take the form:
\begin{eqc}
\item[\eqclbl{ii\phantom{'}$_{\mspace{-7mu}*}$}] $\Ext{Q}{i}{\stalkco{q} \otimes_\Bbbk P}{X^\natural}=0$ for every $P \in \lPrj{\Bbbk}$, $q \in Q$, and $i>0$.

\item[\eqclbl{ii'$_{\mspace{-7mu}*}$}] $\Ext{Q}{1}{\stalkco{q} \otimes_\Bbbk P}{X^\natural}=0$ for every $P \in \lPrj{\Bbbk}$ and $q \in Q$.

\item[\eqclbl{iii\phantom{'}$_{\mspace{-7mu}*}$}] $\Ext{Q}{i}{X^\natural}{\Hom{\Bbbk}{\stalkcn{q}}{I}}=0$ for every $I \in \lInj{\Bbbk}$, $q \in Q$, and $i>0$.

\item[\eqclbl{iii'$_{\mspace{-7mu}*}$}] $\Ext{Q}{1}{X^\natural}{\Hom{\Bbbk}{\stalkcn{q}}{I}}=0$ for every $I \in \lInj{\Bbbk}$ and $q \in Q$.
\end{eqc}
As every projective $\Bbbk$-module $P$ is a direct summand in a coproduct of copies of $\Bbbk$, it~follows that, for fixed $q \in Q$ and $i>0$, the vanishing of $\Ext{Q}{i}{\stalkco{q} \otimes_\Bbbk P}{X^\natural}$ for every projective $\Bbbk$-module $P$ is equivalent to the vanishing of 
\begin{equation*}
  \Ext{Q}{i}{\stalkco{q} \otimes_\Bbbk \Bbbk}{X^\natural} \,\cong\,
  \cH[i]{q}(X^\natural) \,=\, \cH[i]{q}(X)^\natural\;;
\end{equation*}
cf.~\dfnref{cH-hH} and \rmkref{cH-hH}. Hence conditions \eqclbl{ii\phantom{'}$_{\mspace{-7mu}*}$} and \eqclbl{ii'$_{\mspace{-7mu}*}$} above are equivalent to conditions \eqclbl{ii} and \eqclbl{ii'} in the theorem we are proving. 

For every injective $\Bbbk$-module $I$, $q \in Q$, and $i>0$ there are isomorphisms
\begin{equation*}
  \Ext{Q}{i}{X^\natural}{\Hom{\Bbbk}{\stalkcn{q}}{I}} \,\cong\,
  \Hom{\Bbbk}{\Tor{Q}{i}{\stalkcn{q}}{X^\natural}}{I} \,=\,
  \Hom{\Bbbk}{\hH[i]{q}(X)^\natural}{I}\;,
\end{equation*}
so conditions \eqclbl{iii\phantom{'}$_{\mspace{-7mu}*}$} and \eqclbl{iii'$_{\mspace{-7mu}*}$} above are equivalent to \eqclbl{iii} and \eqclbl{iii'} in this theorem. 
\end{proof}

\begin{proof}[Proof of \thmref{quiso}.]
  We start by proving the equivalence between \eqclbl{i}, \eqclbl{ii}, and \eqclbl{ii'}.
  
  \proofofimp{i}{ii} Assume that $\varphi \colon X \to Y$ is a weak equivalence. By \prpref{we} there exists a factorization $\varphi = \pi\iota$ where $\iota \colon X \rightarrowtail Z$ is monic with $\Coker{\iota} \in \lPrj{Q,\alg}$ and $\pi \colon Z \twoheadrightarrow Y$ is epic with $\Ker{\pi} \in \mathscr{E}$. The short exact sequence \smash{$0 \to X \xrightarrow{\iota} Z \to \Coker{\iota} \to 0$} is split exact, and hence so is the sequence
\begin{equation*}
  \xymatrix@C=1.7pc{
  0 \ar[r] & 
  \cH[i]{q}(X) \ar[r]^-{\cH[i]{q}(\iota)} & 
  \cH[i]{q}(Z) \ar[r] & 
  \cH[i]{q}(\Coker{\iota}) \ar[r] & 
  0
  }
\end{equation*}  
for every $q \in Q$ and $i>0$. As $\Coker{\iota} \in \lPrj{Q,\alg} \subseteq \mathscr{E}$ we have $\cH[i]{q}(\Coker{\iota})=0$ by \thmref{E-characterization-hereditary}, so the short exact sequence above shows that $\cH[i]{q}(\iota)$ is an isomorphism. The short exact sequence \smash{$0 \to \Ker{\pi} \to Z \xrightarrow{\pi} Y \to 0$} induces a long exact Ext-sequence,
\begin{equation*}
  \xymatrix@C=1.8pc{
  \cdots \ar[r] &
  \cH[i]{q}(\Ker{\pi}) \ar[r] & 
  \cH[i]{q}(Z) \ar[r]^-{\cH[i]{q}(\pi)} & 
  \cH[i]{q}(Y) \ar[r] & 
  \cH[i+1]{q}(\Ker{\pi}) \ar[r] & 
  \cdots
  }
\end{equation*}  
For $q \in Q$ and $i>0$ we have $\cH[i]{q}(\Ker{\pi})=0$ by \thmref{E-characterization-hereditary} as $\Ker{\pi} \in \mathscr{E}$, so the long exact sequence shows that \smash{$\cH[i]{q}(\pi)$} is an isomorphism. Having proved that \smash{$\cH[i]{q}(\iota)$} and \smash{$\cH[i]{q}(\pi)$} are isomorphisms, it follows that $\cH[i]{q}(\varphi) = \cH[i]{q}(\pi)\cH[i]{q}(\iota)$ is an isomorphism too.

\proofofimp{ii}{ii'} This implication is trivial.

\proofofimp{ii'}{i} Let $\varphi \colon X \to Y$ be a morphism in 
$\lMod{Q,\alg}$ such that $\cH[i]{q}(\varphi)$ is an isomorphism for every $q \in Q$ and $i=1,2$. As the factorization axiom \cite[Dfn.~1.1.3]{modcat} holds in any model category, $\varphi$ admits a factorization $\varphi = \pi\iota$ where $\iota \colon X \rightarrowtail Z$ is a cofibration and $\pi \colon Z \twoheadrightarrow Y$ is a trivial fibration. We will show that $\iota$ is a weak equivalence (and hence a trivial cofibration), as this implies that the composite $\varphi = \pi\iota$ is a weak equivalence.

By assumption, \smash{$\cH[i]{q}(\varphi)$} is an isomorphism for every $q \in Q$ and $i=1,2$. The already established implication \eqclbl{i}\,$\Rightarrow$\,\eqclbl{ii}, applied to the weak equivalence $\pi$, yields that $\cH[i]{q}(\pi)$ is an isomorphism for every $q \in Q$ and $i>0$. By the identity \smash{$\cH[i]{q}(\varphi) = \cH[i]{q}(\pi)\cH[i]{q}(\iota)$} it thus follows that $\cH[i]{q}(\iota)$ is an isomorphism every $q \in Q$ and $i=1,2$. As $\iota$ is an monomorphism (the cofibrations in any abelian model structure are, in particular, monomorphisms by \cite[Dfn.~5.1]{Hovey02}), it follows from \cite[Lem.~5.8]{Hovey02} that $\iota$ is a weak equivalence if and only if $\Coker{\iota}$ belongs to $\mathscr{E}$, equivalently, $\cH{q}(\Coker{\iota})=0$ for every $q \in Q$ by \thmref{E-characterization-hereditary}. To show this we consider the following part of the long exact Ext-sequence induced by the short exact sequence \smash{$0 \to X \xrightarrow{\iota} Y \to \Coker{\iota} \to 0$},
\begin{equation*}
  \xymatrix{
  \cH{q}(X) \ar[r]^-{\cH{q}(\iota)}_{\cong} & 
  \cH{q}(Y) \ar[r] & 
  \cH{q}(\Coker{\iota}) \ar[r] & 
  \cH[2]{q}(X) \ar[r]^-{\cH[2]{q}(\iota)}_{\cong} & 
  \cH[2]{q}(Y)
  }.
\end{equation*}  
As $\cH{q}(\iota)$ and $\cH[2]{q}(\iota)$ are isomorphisms, we get $\cH{q}(\Coker{\iota})=0$ as desired.

  The equivalence between \eqclbl{i}, \eqclbl{iii}, and \eqclbl{iii'} is proved similarly.
\end{proof}

\enlargethispage{1ex}

This concludes the proofs of the main results, \thmref[Theorems~]{E-characterization-hereditary} and \thmref[]{quiso}, of this section; note that  Theorem~D in the Introduction is a special case of these results. We end this section with a strengthening of \thmref{quiso} in the special case where the pseudo-radical squared is zero, see \prpref{quiso-2} below; we also prove Theorem~E from the Introduction, see \thmref{Prj-Inj} below.

It follows from \lemref{stalk} that a necessary condition for $X \in \lMod{Q}$ to be a coproduct of copies of stalk functors $\stalkco{*}$ is that the functor $X$ takes values in the (sub)category of free $\Bbbk$-modules. In some cases, this condition is also sufficient, as the next result shows.

\begin{lem}
  \label{lem:coproduct-stalk}
  Let $\ell \geqslant 0$. Assume that the $\Bbbk$-module $(\mathfrak{r}^\ell/\mathfrak{r}^{\ell+1})(p,q)$ is free for all $p,q \in Q$. Then the following assertions hold.
\begin{prt}
\item For every $p \in Q$ there exists a collection $\{U(q)\}_{q \in Q}$ of index sets and an isomorphism $(\mathfrak{r}^\ell/\mathfrak{r}^{\ell+1})(p,-) \,\cong\, \bigoplus_{q \in Q} \stalkco{q}^{(U(q))}$ in $\lMod{Q}$.

\item For every $q \in Q$ there exists a collection $\{V(p)\}_{p \in Q}$ of index sets and an isomorphism $(\mathfrak{r}^\ell/\mathfrak{r}^{\ell+1})(-,q) \,\cong\, \bigoplus_{p \in Q} \stalkcn{p}^{(V(p))}$ in $\rMod{Q}$.

\end{prt}  
\end{lem}

\begin{proof}
  \proofoftag{a} Fix  $p \in Q$. For every $q \in Q$ let \smash{$\{\varepsilon_{q,u}\}_{u \mspace{1mu}\in\mspace{1mu} U(q)}$} be a subset of $\mathfrak{r}^\ell(p,q)$ such that $\{\bar{\varepsilon}_{q,u}\}_{u \mspace{1mu}\in\mspace{1mu} U(q)}$ is a basis of the free $\Bbbk$-module $(\mathfrak{r}^\ell/\mathfrak{r}^{\ell+1})(p,q)$; here $\bar{\varepsilon}_{q,u}$ denotes~the image of $\varepsilon_{q,u}$ in $(\mathfrak{r}^\ell/\mathfrak{r}^{\ell+1})(p,q)$. As $\mathfrak{r}^\ell$ is an ideal in $Q$ containing $\varepsilon_{q,u}$, the image of the natural transformation \mbox{$Q(\varepsilon_{q,u},-) \colon Q(q,-) \to Q(p,-)$} is contained in $\mathfrak{r}^\ell(p,-)$ and $Q(\varepsilon_{q,u},-)$ maps the subfunctor $\mathfrak{r}(q,-)$ to $\mathfrak{r}^{\ell+1}(p,-)$. Thus, $Q(\varepsilon_{q,u},-)$ induces a natural transformation,
\begin{equation*}
  \xymatrix{
    \stalkco{q} \,=\, Q(q,-)/\mathfrak{r}(q,-)
    \ar[r]^-{\tau_{q,u}} & 
    (\mathfrak{r}^\ell/\mathfrak{r}^{\ell+1})(p,-)
  }.
\end{equation*}
By the universal property of the coproduct, we get an induced natural transformation,
\begin{equation*}
  \xymatrix{
    \bigoplus_{q \in Q} \stalkco{q}^{(U(q))}
    \ar[r]^-{\tau} & 
    (\mathfrak{r}^\ell/\mathfrak{r}^{\ell+1})(p,-)
  }.
\end{equation*}
We claim that $\tau$ is an isomorphism, i.e.~for every $r \in Q$ the $\Bbbk$-module homomorphism below is an isomorphism (the equalities in this display follow from \lemref{stalk}),
\begin{equation*}
  \xymatrix{
    \Bbbk^{(U(r))} \,=\,
    \stalkco{r}(r)^{(U(r))} \,=\,     
    \bigoplus_{q \in Q} \stalkco{q}(r)^{(U(q))}
    \ar[r]^-{\tau^r} & 
    (\mathfrak{r}^\ell/\mathfrak{r}^{\ell+1})(p,r)
  }.
\end{equation*}
Indeed, by definition, $\tau^r$ maps an element $(x_u)_{u \mspace{1mu}\in\mspace{1mu} U(r)}$ in \smash{$\Bbbk^{(U(r))}$} to the sum $\sum_{u \mspace{1mu}\in\mspace{1mu} U(r)}x_u\mspace{1mu} \bar{\varepsilon}_{r,u}\mspace{2mu}$, and this map is an isomorphism by construction.

\proofoftag{b} Similar to the proof of part \prtlbl{a}.
\end{proof}

\begin{prp}
  \label{prp:quiso-2}
   Adopt the setup of \thmref{E-characterization-hereditary} but assume that $N=2$, that is, $\mathfrak{r}^2=0$, and that $\Bbbk$ is a PID (e.g.~$\Bbbk=\mathbb{Z}$). For every $\Bbbk$-algebra $\alg$ and morphism $\varphi$ in $\lMod{Q,\alg}$ the following conditions are equivalent:
  \begin{eqc}
  \item $\varphi$ is a weak equivalence.
  \item $\cH{q}(\varphi)$ is an isomorphism for every $q \in Q$.
  \item $\hH{q}(\varphi)$ is an isomorphism for every $q \in Q$.
  \end{eqc}
\end{prp}

\begin{proof}
  We know from \thmref{quiso} that \eqclbl{i} implies both \eqclbl{ii} and \eqclbl{iii}.
  
  \proofofimp{ii}{i} By \thmref{quiso} it suffices to show that also \smash{$\cH[2]{q}(\varphi)$} is an isomorphism for every $q \in Q$. For every $p,q \in Q$ the $\Bbbk$-module $Q(p,q)$ is projective (= free, as $\Bbbk$ is a PID) by condition \eqref{Homfin} in \stpref{Bbbk}, and hence so is the submodule $\mathfrak{r}(p,q)$. Since $\mathfrak{r}^2=0$ we have  $(\mathfrak{r}/\mathfrak{r}^2)(p,q) = \mathfrak{r}(p,q)$, so we can apply
\lemref{coproduct-stalk}\prtlbl{a} with $\ell=1$ to get an isomorphism $\mathfrak{r}(q,-) \,\cong\, \bigoplus_{r \in Q} \stalkco{r}^{(U(r))}$ in $\lMod{Q}$ for suitable index sets $U(r)$. This yields the third equality below; the first and last equalities holds by \dfnref{cH-hH}, and the second equality holds by dimension shifting, as $\mathfrak{r}(q,-)$ is a first syzygy of $\stalkco{q}$ by \dfnref{stalk}.
\begin{equation*}
  \cH[2]{q}(?)
  \,=\,
  \Ext{Q}{2}{\stalkco{q}}{?}
  \,=\,
  \Ext{Q}{1}{\mathfrak{r}(q,-)}{?}  
  \,=\,
  \textstyle
  \prod_{r \in Q} \Ext{Q}{1}{\stalkco{r}}{?}^{U(r)}
  \,=\,
  \textstyle
  \prod_{r \in Q} \cH{r}(?)^{U(r)}
\end{equation*}
From this identity we see that if $\cH{r}(\varphi)$ is an ismorphism for every $r \in Q$, then $\cH[2]{q}(\varphi)$ is also an ismorphism for every $q \in Q$, as desired.

  \proofofimp{iii}{i} Similar to the proof of the implication \eqclbl{ii}\,$\Rightarrow$\,\eqclbl{i}. 
\end{proof}

We end this section by proving \thmref{Prj-Inj} below, which gives a useful characterization of the projective and injective objects in $\lMod{Q,\alg}$. Recall from \corref{adjoint-evaluation} the~functors $\Fq{q}$ and $\Gq{q}$ and from \prpref{adjoint-stalk} the functors $\Cq{q}$ and $\Kq{q}$.

\begin{lem}
  \label{lem:CF}
  For every $p,q \in Q$ the following assertions hold.
\begin{prt}
\item $\Cq{p}\Fq{p} = \mathrm{id}$ and $\Cq{p}\Fq{q}=0$ if $p \neq q$.

\item $\Kq{p}\Gq{p} = \mathrm{id}$ and $\Kq{p}\Gq{q}=0$ if $p \neq q$.
\end{prt}
Here \textnormal{``$\mathrm{id}$''} denotes the identity functor on $\lMod{\alg}$.
\end{lem}

\begin{proof}
\proofoftag{a} In the computation below, the first isomorphism holds by the definitions of $\Cq{p}$ and $\Fq{q}$, the middle isomorphism follows from \lemref{associativity}, and the last isomorphism is already mentioned in the proof of \corref{adjoint-evaluation}.
\begin{equation*}
   \Cq{p}\Fq{q}(?) 
   \,\cong\, 
   \ORt{\stalkcn{p}}{(Q(q,-) \otimes_\Bbbk {?})}
   \,\cong\,
   (\ORt{\stalkcn{p}}{Q(q,-)}) \otimes_\Bbbk {?}
   \,\cong\,
   \stalkcn{p}(q) \otimes_\Bbbk {?}
\end{equation*}
The desired conclusion now follows from \lemref{stalk}.

\proofoftag{b} Similar to the proof of part \prtlbl{a}.
\end{proof}

\begin{thm}
  \label{thm:Prj-Inj}
  Assume that $Q$ satisfies condition \eqref{Homfin} in \stpref{Bbbk} and that the pseudo-radical $\mathfrak{r}$ is nilpotent, that is, $\mathfrak{r}^N=0$ for some $N \in \mathbb{N}$. For every $X \in \lMod{Q,\alg}$ one has:
\begin{prt}
\item $X \in \lPrj{Q,\alg}$ if and only if\, $\hH{q}(X)=0$ and $\Cq{q}(X) \in \lPrj{\alg}$ for every $q \in Q$.
\item $X \in \lInj{Q,\alg}$ if and only if\, $\cH{q}(X)=0$ and $\Kq{q}(X) \in \lInj{\alg}$ for every $q \in Q$.
\end{prt}
\end{thm}

\begin{proof}
\proofoftag{a} ``Only if'': It is immediate from \dfnref{cH-hH} and \prpref{adjoint-stalk} that the functors $\hH{q}$ and $\Cq{q}$ preserve coproducts. Thus, to show the ``only if'' part we can~by~\prpref{lfp}\prtlbl{a} assume that $X$ has the form $X=\Fq{p}(\alg)$ for some $p \in Q$. In this case we have $\Cq{q}(X) = \Cq{q}\Fq{p}(\alg)$, which is either $\alg$ or $0$ by \lemref{CF}\prtlbl{a}; in particular this $\alg$-module belongs to $\lPrj{\alg}$. We also have $\hH{q}(X)=\hH{q}(\Fq{p}(\alg))=0$ by \lemref{FGH}\prtlbl{a}.

``If'': Let $q \in Q$. Consider the canonical epimorphism $\pi^X_q \colon X(q)  \twoheadrightarrow \Cq{q}(X)$, see \prpref{K-C-formulae} (and its proof). Since the $\alg$-module $\Cq{q}(X)$ is projective, \smash{$\pi^X_q$} has a right inverse, say, $\iota_q \colon \Cq{q}(X) \rightarrowtail X(q)$. We define $\varphi_q$ be the composite
\begin{equation*}
  \xymatrix{
    \Fq{q}\Cq{q}(X) \ar[r]^-{\Fq{q}(\iota_q)} &
    \Fq{q}(X(q)) = \Fq{q}\Eq{q}(X) \ar[r]^-{\varepsilon^X_q} &
    X
  },
\end{equation*}
where $\varepsilon^X_q$ is the counit of the adjunction $(\Fq{q},\Eq{q})$ from \corref{adjoint-evaluation}. Note that $\Cq{q}(\varepsilon^X_q)=\pi^X_q$ and 
$\Cq{q}\Fq{q}(\iota_q) = \iota_q$, see \lemref{CF}\prtlbl{a}, and thus \smash{$\Cq{q}(\varphi_q) = \pi^X_q\,\iota_q = \mathrm{id}_{\Cq{q}(X)}$}. By the universal property of the coproduct, there is a unique morphism, 
\begin{equation*}
  \textstyle
  \varphi \colon \bigoplus_{q \in Q} \Fq{q}\Cq{q}(X) \longrightarrow X\;,
\end{equation*}
induced by the family $\{\varphi_q\}_{q \in Q}$. For every $p \in Q$ the functor $\Cq{p}$ is a left adjoint by \prpref{adjoint-stalk}, so it preserves coproducts. By \lemref{CF}\prtlbl{a} and the fact that \smash{$\Cq{p}(\varphi_p) =  \mathrm{id}_{\Cq{p}(X)}$}, it follows that $\Cq{p}(\varphi)$ is an isomorphism, in fact, it is the identity on $\Cq{p}(X)$.

Now, applying the right exact functor $\Cq{p}$ to the exact sequence
\begin{equation*}
  \xymatrix{
    \bigoplus_{q \in Q} \Fq{q}\Cq{q}(X) \ar[r]^-{\varphi} &
    X \ar[r] & \Coker{\varphi} \ar[r] & 0
  }
\end{equation*}  
and using that $\Cq{p}(\varphi)$ is surjective, it follows that $\Cq{p}(\Coker{\varphi})$ for every $p \in Q$, and therefore $\Coker{\varphi}=0$ by \prpref{zero-criterion}. Thus there is a short exact sequence,
\begin{equation}
  \label{eq:Ker}
  \xymatrix{
    0 \ar[r] & \Ker{\varphi} \ar[r] & 
    \bigoplus_{q \in Q} \Fq{q}(\Cq{q}(X)) \ar[r]^-{\varphi} &
    X \ar[r] & 0
  }.
\end{equation}
For every $p \in Q$ one has \smash{$\Cq{p} = \ORt{\stalkcn{p}}{-}$} and \smash{$\hH{p} = \Tor{Q}{1}{\stalkcn{p}}{-}$}, and as it is assumed that $\hH{p}(X)=0$, the functor $\Cq{p}$ leaves the~sequence \eqref{Ker} exact. As the homomorphism $\Cq{p}(\varphi)$ is injective, we get $\Cq{p}(\Ker{\varphi})=0$ for every $p \in Q$, and thus $\Ker{\varphi}=0$; again by \prpref{zero-criterion}. We have now shown that $\varphi$ is an isomorphism. Since $\bigoplus_{q \in Q} \Fq{q}\Cq{q}(X)$ is a projective object in $\lMod{Q,\alg}$ by \lemref{Fq-Gq-preserve}, we conclude that $X$ is projective too.

\proofoftag{b} Dual to the proof of part \prtlbl{a}.
\end{proof}

\section{Stable Translation Quivers}
\label{sec:Example}

In this section, we investigate natural examples of (small) $\Bbbk$-preadditive categories that satisfy conditions \eqref{Homfin}, \eqref{locbd}, \eqref{Serre} in \stpref{Bbbk} and condition \eqmref{strong-retraction} in \dfnref{srp} with a nilpotent pseudo-radical. Recall that for such a category, $Q$, \thmref[Theorems~]{model-structures}, \thmref[]{E-characterization-hereditary}, and \thmref[]{quiso} show that for any ring $\alg$ (if we take $\Bbbk=\mathbb{Z}$), the category $\lMod{Q,\alg}$ has two different model structures where the trivial objects and the weak equivalences can be naturally characterized by the (co)homology functors from \dfnref{cH-hH}. The examples we have in mind are mesh categories of (suitably nice) stable translation quivers.

Recall that a \emph{stable translation quiver} is a triple $(\upGamma,\tr,\str)$ where $\upGamma=(\upGamma_0,\upGamma_1)$ is a quiver and $\tr \colon \upGamma_0 \to \upGamma_0$ (the \emph{translation}) and $\str \colon \upGamma_1 \to \upGamma_1$ (the \emph{semitranslation}) are bijections such that for every arrow $a \colon p \to q$ in $\upGamma$ the arrow $\sigma(a) \colon \tau(q) \to p$ goes from $\tau(q)$ to $p$. The sets $\upGamma_0$ and $\upGamma_1$ may be infinite (this will always be the case in \exaref{rep} below if $\upDelta_0$ and $\upDelta_1$ are non-empty), but we assume that $\upGamma$ is \emph{locally finite}, that is, for every vertex $q \in \upGamma_0$ there are only finitely many arrows with target $q$. Note that $\upGamma$ may have \emph{oriented cycles}, that is,  paths of length \mbox{$\geqslant\mspace{-2mu}1$} starting and ending at the same vertex; an oriented cycle of length one is called a \emph{loop} (in fact, the stable translation quivers arising from \exaref{dou} below will always have oriented cycles of length \mbox{$>\mspace{-2mu}1$} if $\upDelta_1$ is non-empty).

For every $q \in \upGamma_0$ the \emph{mesh} associated with $q$ is the diagram:
\begin{equation}
  \label{eq:mesh}
  \begin{gathered}
  \xymatrix@!=1ex{
    {} & p_1 \ar[dr]^-{a_1} & {}
    \\
    \tr(q)
    \ar[ur]^-{\str(a_1)}
    \ar[dr]_-{\str(a_n)}
    & \vdots & q
    \\
    {} & p_n \ar[ur]_-{a_n} & {}
  }
  \end{gathered}
\end{equation}
where $a_1,\ldots,a_n$ are all the finitely many different arrows in $\upGamma$ with target $q$. It follows that $\str(a_1),\ldots,\str(a_n)$ are all the arrows in $\upGamma$ with source $\tr(q)$. 

There are a couple of standard ways to obtain a stable translation quiver from an ordinary quiver, which we now describe.

\begin{exa}[the double quiver]
  \label{exa:dou}
  Let $\upDelta=(\upDelta_0,\upDelta_1)$ be a quiver. The \emph{double} quiver $\upGamma = \dou{\upDelta}$ of $\upDelta$ has the same vertices as $\upDelta$, that is, $\upGamma_0 = \upDelta_0$, but twice as many arrows. More precisely, $\upGamma$ has all the original arrows of $\upDelta$ but also an arrow $x^* \colon q \to p$ for every arrow $x \colon p \to q$ in $\upDelta$ (note that $x^*$ goes in the opposite direction of $x$); in symbols:
\begin{equation*}
  \upGamma_1 \;=\; \upDelta_1 \,\uplus\,\{ x^* \colon q \to p \ |\, (x \colon p \to q) \in \upDelta_1 \}\;.
\end{equation*}  
  The double quiver $\upGamma = \dou{\upDelta}$ is a stable translation quiver: The translation $\tr \colon \upGamma_0 \to \upGamma_0$ is the identity and the semitranslation $\str \colon \upGamma_1 \to \upGamma_1$ is given by $\sigma(x) = x^*$ and $\sigma(x^*)=x$ for every arrow $x$ in $\upDelta$.
\end{exa}

\begin{exa}[the repetitive quiver]
  \label{exa:rep}
  Let $\upDelta=(\upDelta_0,\upDelta_1)$ be a quiver. The \emph{repetitive} quiver $\upGamma = \rep{\upDelta}$ (even though the symbol $\mathbb{Z}\upDelta$ is commonly used for this quiver, we avoid it in this paper) of $\upDelta$ has the vertex set $\upGamma_0 = \upDelta_0 \times \mathbb{Z}$ and arrows 
\begin{equation*}  
  x_i \colon (p,i) \longrightarrow (q,i)
  \qquad \text{and} \qquad 
  x^*_i \colon (q,i) \longrightarrow (p,i-1)
\end{equation*}  
for every arrow $x \colon p \to q$ in $\upDelta$. The repetitive quiver $\upGamma = \rep{\upDelta}$ is a stable translation quiver with translation \mbox{$\tr \colon \upGamma_0 \to \upGamma_0$} given by $\tr(q,i) = (q,i+1)$ for $(q,i) \in \upGamma_0 = \upDelta_0 \times \mathbb{Z}$ and semitranslation $\str \colon \upGamma_1 \to \upGamma_1$ given by $\sigma(x_i) = x^*_{i+1}$ and $\sigma(x^*_i)=x_i$ for every arrow $x$ in $\upDelta$.
\end{exa}

\begin{rmk}
  \label{rmk:graph}
A graph $G$ can be turned it into a quiver, $\upDelta = \vec{G}$, by choosing some orientation of its vertices. If $G$ is a \emph{tree}, i.e.~any two vertices in $G$ are connected by exactly one path (equivalently, $G$ is a connected acyclic graph), then the repetitive quiver \smash{$\rep{\upDelta} = \rep{(\vec{G})}$} does not depend (up to isomorphism of stable translation quivers) on the chosen orientation of $G$; see Happel \cite[p.~53]{Happel}. The same it true for for the double quiver \smash{$\dou{\upDelta} = \dou{(\vec{G})}$}.
\end{rmk}

The following definitions are standard.

\begin{dfn}
  \label{dfn:mesh}
Let $(\upGamma,\tr,\str)$ be a stable translation quiver and $\Bbbk$ be a commutative ring.

The \emph{path category} of $\upGamma$ over $\Bbbk$ is the $\Bbbk$-preadditive category $\Bbbk \upGamma$ whose objects are the vertices of $\upGamma$ and whose hom sets $\Bbbk \upGamma(p,q)$ are the free $\Bbbk$-modules with basis the set of paths in $\upGamma$ from $p$ to $q$. (Of course, this definition works for any quiver $\upGamma$.)

The \emph{mesh relation} associated with a vertex $q$ in $\upGamma$ is the (formal) sum $\mu_q$ below of paths from $\tr(q)$ to $q$, cf.~\eqref{mesh}. Note that $\mu_q \colon \tr(q) \to q$ is a morphism in the path category $\Bbbk \upGamma$.
\begin{equation*}
   \mu_q \,=\, a_1 \,\str(a_1) + \cdots + a_n \,\str(a_n)
\end{equation*}

 Let $\mathfrak{m}$ be the \emph{mesh ideal}, i.e.~the two-sided ideal in the $\Bbbk$-preadditive category $\Bbbk \upGamma$ generated by all mesh relations; in symbols, $\mathfrak{m} = \langle\, \mu_q \,|\, q \in \upGamma_0 \,\rangle$. The \emph{mesh category} of $\upGamma$ over $\Bbbk$ is the quotient category:
\begin{equation*}
  \mesh{\upGamma} \,=\, (\Bbbk \upGamma)/\mathfrak{m}\;. 
\end{equation*}  
\end{dfn}

\begin{rmk}
  \label{rmk:rep}
  Notice that for $Q=\mesh{\upGamma}$ the category $\lMod{Q,\alg}$ is equivalent to the category of $\lMod{\alg}$-valued representations of $\upGamma$ that satisfy the mesh relations.
\end{rmk}

As we now prove, the mesh category of a stable translation quiver always satisfies the Strong Retraction Property. In general, the associated pseudo-radical $\mathfrak{r}$ need not be nilpotent (as assumed in \thmref[Theorems~]{E-characterization-hereditary} and \thmref[]{quiso}); however, in several natural examples (see \thmref[Theorems~]{A-dou} and \thmref[]{A-rep}), it will be.

\begin{lem}
  \label{lem:arrow-ideal}
  The mesh category $Q=\mesh{\upGamma}$, over any commutative ring $\Bbbk$, 
  of a stable translation quiver $\upGamma$ satisfies condition \eqmref{strong-retraction} in \dfnref{srp}. The \emph{\emph{arrow ideal}}, i.e.~the ideal in $Q$ generated by all arrows in $\upGamma$, serves as a pseudo-radical.   
\end{lem}

\begin{proof}
  Let $P = \Bbbk\upGamma$ be the path category of $\upGamma$. For every $\ell \geqslant 0$ let $P^\ell(p,q)$ and $Q^\ell(p,q)$~the $\Bbbk$-submodules of $P(p,q)$ and $Q(p,q)$ generated by all paths in $\upGamma$ of length $\ell$; here we~have included the trivial paths $\mathrm{id}_q$ for $q \in \upGamma_0$, which have length \mbox{$\ell=0$}. There is a direct sum decomposition of $\Bbbk$-modules, \smash{$P(p,q) =  \bigoplus_{\ell \geqslant 0} P^\ell(p,q)$}, and as the mesh relations are homogeneous (of degree two) and only involve unit coefficients $\pm 1 \in \Bbbk$, we get an induced direct sum decomposition for the quotient category $Q=P/\mathfrak{m}$, that is,
\begin{equation}  
  \label{eq:Q-decomp}
  Q(p,q) \,=\, \bigoplus_{\ell \geqslant 0}\, Q^\ell(p,q)\;.
\end{equation}  
Let $\mathfrak{r}$ be the arrow ideal in $Q$; thus $\mathfrak{r}(p,q)$ is the $\Bbbk$-submodule of $Q(p,q)$ generated by all paths in $\upGamma$ of length $\geqslant 1$. That is, $\mathfrak{r}(p,q) = \bigoplus_{\ell \geqslant 1}\, Q^\ell(p,q)$, and hence
\begin{equation*}
  Q(p,q) \,=\, Q^0(p,q) \oplus \mathfrak{r}(p,q) \,=\,
    \left\{\mspace{-5mu}
  \begin{array}{cl}
    (\Bbbk\cdot\mathrm{id}_q) \oplus \mathfrak{r}(q,q) & \text{if $p = q$}
    \\
    \mathfrak{r}(p,q) & \text{if $p \neq q$}
  \end{array}
  \right.
\end{equation*}  
in view of \eqref{Q-decomp}. It is evident that the $\Bbbk$-submodules $\mathfrak{r}_q := \mathfrak{r}(q,q)$ satisfy the requirements in condition \eqmref{strong-retraction} in \dfnref{srp}, and the pseudo-radical associated to these $\Bbbk$-submodules, in the sense of \lemref{ideal}, is precisely the arrow ideal $\mathfrak{r}$ we started with.
\end{proof}

For $n \geqslant 2$ we now consider the Dynkin graph $G=\mathbb{A}_n$ with \emph{linear orientation}, that is,
\begin{equation}
  \label{eq:An}
  \vec{\mathbb{A}}_n \;=\; 
  \xymatrix{
    \underset{1}{\vtx} \ar[r]^-{a^{}_1} &
    \underset{2}{\vtx} \ar[r]^-{a^{}_2} &
    \ \cdots \ \ar[r] &
    \underset{n-1}{\vtx} \ar[r]^-{a^{}_{n-1}} &
    \underset{n}{\vtx} 
  }.
\end{equation}
Below we study the double quiver $\dou{(\vec{\mathbb{A}}_n)}$ and the repetitive quiver $\dou{(\vec{\mathbb{A}}_n)}$ of $\vec{\mathbb{A}}_n$. As noted in \rmkref{graph}, these stable translation quivers do not depend on the chosen orientation.

\begin{exa}
  \label{exa:A-dou}
  Consider the double quiver of $\vec{\mathbb{A}}_n$ from \eqref{An}, that is,
  \begin{equation*}
  \dou{(\vec{\mathbb{A}}_n)} \;=\; 
  \xymatrix{
    \underset{1}{\vtx} \ar@<0.6ex>[r]^-{a^{}_1} &
    \underset{2}{\vtx} \ar@<0.6ex>[l]^-{a^*_1} \ar@<0.6ex>[r]^-{a^{}_2} &
    \ \cdots \ \ar@<0.6ex>[r] \ar@<0.6ex>[l]^-{a^*_2} &
    \underset{n-1}{\vtx} \ar@<0.6ex>[l] \ar@<0.6ex>[r]^-{a^{}_{n-1}} &
    \underset{n}{\vtx} \ar@<0.6ex>[l]^-{a^*_{n-1}}
  }.
  \end{equation*}
  By \dfnref{mesh} the mesh relations are:
  \begin{equation*}
     \mu_1 \,=\, a^*_1a^{}_1\;, \quad 
     \mu_q \,=\, a^{}_{q-1}a^*_{q-1}+a^*_{q}a^{}_{q} \ \, \text{ for } \, \ 1<q<n\;,
     \quad \text{and} \quad
     \mu_n \,=\, a^{}_{n-1}a^*_{n-1}\;.      
  \end{equation*}
\end{exa}
    
\begin{thm}
  \label{thm:A-dou}
  Let $\Bbbk$ be any commutative ring. The mesh category
  \begin{equation*}
  Q \,=\, \mesh{\dou{(\vec{\mathbb{A}}_n)}}
  \end{equation*}
  over $\Bbbk$ of the double quiver of $\vec{\mathbb{A}}_n$ satisfies conditions \eqref{Homfin}, \eqref{locbd}, \eqref{Serre} in \stpref{Bbbk} and condition \eqmref{strong-retraction} in \dfnref{srp}. More precisely, the following assertions hold.
  \begin{prt}
   \item Every hom $\Bbbk$-module $Q(p,q)$ is free of dimension
  \begin{equation*}
    d(p,q) \,=\, \min\{p,\,q,\,n+1-p,\,n+1-q\} \in \mathbb{N}\;.
  \end{equation*}  

  \item The Serre functor is given by $\Serre(q) = n+1-q$ on objects and its action on morphisms is determined by the formulae $\Serre(a^{}_q) = (-1)^q \mspace{1mu} a^*_{n-q}$ and $\Serre(a^*_q) = (-1)^{n-q} \mspace{1mu} a^{}_{n-q}$.
  
  \end{prt}
  Moreover, the arrow ideal $\mathfrak{r}$ (which by \lemref{arrow-ideal} is the pseudo-radical in $Q$) is nilpotent; in fact, one has $\mathfrak{r}^n=0$. 
\end{thm}  

\begin{proof}
  See \appref{proof}.
\end{proof}

\begin{cor}
  Let $\alg$ be any ring. Consider the category of $\lMod{\alg}$-valued repre\-sen\-tations of the double quiver of \smash{$\vec{\mathbb{A}}_n$} that satisfy the mesh relations. This category has two different abelian model structure where the trivial objects and the weak equivalences can be~characterized by the (co)homology functors $\cH[i]{q}$ and $\hH[i]{q}$ as in \thmref[Theorems~]{E-characterization-hereditary} and \thmref[]{quiso}.
\end{cor}
 
\begin{proof}
  Set $\Bbbk=\mathbb{Z}$. As noted in \rmkref{rep}, the category in question is nothing but $\lMod{Q,\alg}$ where \smash{$Q = \mesh{\dou{(\vec{\mathbb{A}}_n)}}$}. By \thmref{A-dou} this 
  $Q$ satisfies conditions \eqref{Homfin}, \eqref{locbd}, \eqref{Serre} in \stpref{Bbbk} and condition \eqmref{strong-retraction} in \dfnref{srp} with a nilpotent pseudo-radical. The assertion therefore follows directly from \thmref[Theorems~]{model-structures}, \thmref[]{E-characterization-hereditary},  and \thmref[]{quiso}.
\end{proof}

\begin{exa}
  \label{exa:A-rep}
    We consider the repetitive quiver $\rep{(\vec{\mathbb{A}}_n)}$ of the quiver $\vec{\mathbb{A}}_n$ from \eqref{An}. For e.g.~$n=5$ it looks as follows (where the $\vtx$'s have been left out):
  \begin{equation*}
    \xymatrix@!=0.9pc{
      \text{\small $(1,\mspace{1mu}i)$} 
      \ar[dr]^(0.6){\mspace{-15mu}a^{}_{1,\mspace{1mu}i}}
      &
      {}
      &
      \text{\small $(1,\mspace{1mu}i-1)$}    
      \ar[dr]^(0.6){\mspace{-15mu}a^{}_{1,\mspace{1mu}i-1}}
      &
      {}
      &
      \text{\small $(1,\mspace{1mu}i-2)$}    
      \ar[dr]^(0.6){\mspace{-15mu}a^{}_{1,\mspace{1mu}i-2}}
      &
      {}
      &
      \text{\small $(1,\mspace{1mu}i-3)$}    
      \ar[dr]^(0.6){\mspace{-15mu}a^{}_{1,\mspace{1mu}i-3}}
      &
      {}
      &
      \text{\small $(1,\mspace{1mu}i-4)$}    
      \\
      \mspace{-50mu}\cdots
      &
      \text{\small $(2,\mspace{1mu}i)$}    
      \ar[dr]^(0.6){\mspace{-15mu}a^{}_{2,\mspace{1mu}i}}
      \ar@{->}[ur]_(0.59){\mspace{-13mu}a^*_{1,\mspace{1mu}i}}
      &
      {}
      &
      \text{\small $(2,\mspace{1mu}i-1)$}    
      \ar[dr]^(0.6){\mspace{-15mu}a^{}_{2,\mspace{1mu}i-1}}
      \ar@{->}[ur]_(0.59){\mspace{-13mu}a^*_{1,\mspace{1mu}i-1}}
      &
      {}
      &
      \text{\small $(2,\mspace{1mu}i-2)$}    
      \ar[dr]^(0.6){\mspace{-15mu}a^{}_{2,\mspace{1mu}i-2}}
      \ar@{->}[ur]_(0.59){\mspace{-13mu}a^*_{1,\mspace{1mu}i-2}}
      &
      {}
      &
      \text{\small $(2,\mspace{1mu}i-3)$}    
      \ar[dr]^(0.6){\mspace{-15mu}a^{}_{2,\mspace{1mu}i-3}}
      \ar@{->}[ur]_(0.59){\mspace{-13mu}a^*_{1,\mspace{1mu}i-3}}
      &
      \mspace{60mu}\cdots
      \\
      \text{\small $(3,\mspace{1mu}i+1)$}    
      \ar[dr]^(0.6){\mspace{-15mu}a^{}_{3,\mspace{1mu}i+1}}
      \ar@{->}[ur]_(0.59){\mspace{-13mu}a^*_{2,\mspace{1mu}i+1}}
      &
      {}
      &
      \text{\small $(3,\mspace{1mu}i)$}    
      \ar[dr]^(0.6){\mspace{-15mu}a^{}_{3,\mspace{1mu}i}}
      \ar@{->}[ur]_(0.59){\mspace{-13mu}a^*_{2,\mspace{1mu}i}}
      &
      {}
      &
      \text{\small $(3,\mspace{1mu}i-1)$}    
      \ar[dr]^(0.6){\mspace{-15mu}a^{}_{3,\mspace{1mu}i-1}}
      \ar@{->}[ur]_(0.59){\mspace{-13mu}a^*_{2,\mspace{1mu}i-1}}
      &
      {}
      &
      \text{\small $(3,\mspace{1mu}i-2)$}    
      \ar[dr]^(0.6){\mspace{-15mu}a^{}_{3,\mspace{1mu}i-2}}
      \ar@{->}[ur]_(0.59){\mspace{-13mu}a^*_{2,\mspace{1mu}i-2}}
      &
      {}
      &
      \text{\small $(3,\mspace{1mu}i-3)$}    
      \\
      \mspace{-50mu}\cdots
      &
      \text{\small $(4,\mspace{1mu}i+1)$}    
      \ar[dr]^(0.6){\mspace{-15mu}a^{}_{4,\mspace{1mu}i+1}}
      \ar@{->}[ur]_(0.59){\mspace{-13mu}a^*_{3,\mspace{1mu}i+1}}
      &
      {}
      &
      \text{\small $(4,\mspace{1mu}i)$}    
      \ar[dr]^(0.6){\mspace{-15mu}a^{}_{4,\mspace{1mu}i}}
      \ar@{->}[ur]_(0.59){\mspace{-13mu}a^*_{3,\mspace{1mu}i}}
      &
      {}
      &
      \text{\small $(4,\mspace{1mu}i-1)$}    
      \ar[dr]^(0.6){\mspace{-15mu}a^{}_{4,\mspace{1mu}i-1}}
      \ar@{->}[ur]_(0.59){\mspace{-13mu}a^*_{3,\mspace{1mu}i-1}}
      &
      {}
      &
      \text{\small $(4,\mspace{1mu}i-2)$}    
      \ar[dr]^(0.6){\mspace{-15mu}a^{}_{4,\mspace{1mu}i-2}}
      \ar@{->}[ur]_(0.59){\mspace{-13mu}a^*_{3,\mspace{1mu}i-2}}
      &
      \mspace{60mu}\cdots
      \\
      \text{\small $(5,\mspace{1mu}i+2)$}    
      \ar@{->}[ur]_(0.59){\mspace{-13mu}a^*_{4,\mspace{1mu}i+2}}      
      &
      {}
      &
      \text{\small $(5,\mspace{1mu}i+1)$}    
      \ar@{->}[ur]_(0.59){\mspace{-13mu}a^*_{4,\mspace{1mu}i+1}}      
      &
      {}
      &
      \text{\small $(5,\mspace{1mu}i)$}    
      \ar@{->}[ur]_(0.59){\mspace{-13mu}a^*_{4,\mspace{1mu}i}}      
      &
      {}
      &
      \text{\small $(5,\mspace{1mu}i-1)$}    
      \ar@{->}[ur]_(0.59){\mspace{-13mu}a^*_{4,\mspace{1mu}i-1}}      
      &
      {}
      &
      \text{\small $(5,\mspace{1mu}i-2)$}
    }
  \end{equation*}
  By \dfnref{mesh} the mesh relations are, for every $i \in \mathbb{Z}$,
  \begin{align*}
    \mu_{(1,i)} &\,=\, a^*_{1,\mspace{1mu}i+1}\,a^{}_{1,\mspace{1mu}i+1}\;,
    \\
    \mu_{(q,i)} &\,=\, a^{}_{q-1,\mspace{1mu}i}\,a^*_{q-1,\mspace{1mu}i+1}\,+\,a^*_{q,\mspace{1mu}i+1}\,a^{}_{q,\mspace{1mu}i+1} \ \ \text{ for } \ \ 1<q<n\;, \quad \text{and}
    \\
    \mu_{(n,i)} &\,=\, a^{}_{n-1,\mspace{1mu}i}\,a^*_{n-1,\mspace{1mu}i+1}\;.
  \end{align*}
\end{exa}

\begin{thm}
  \label{thm:A-rep}
  Let $\Bbbk$ be any commutative ring. The mesh category
  \begin{equation*}
  Q \,=\, \mesh{\rep{(\vec{\mathbb{A}}_n)}}
  \end{equation*}
  over $\Bbbk$ of the repetitive quiver of $\vec{\mathbb{A}}_n$ satisfies conditions \eqref{Homfin}, \eqref{locbd}, \eqref{Serre} in \stpref{Bbbk} and con\-dition \eqmref{strong-retraction} in \dfnref{srp}. More precisely, the following assertions hold.
  \begin{prt}
  \item Each hom $\Bbbk$-module $Q((p,i),(q,j))$ is either zero or free of dimension $1$. The latter happens if and only if the point $(q,j)$ lies in, or on the boundary of, the rectangle~\smash{$R_{p,\mspace{1mu}i}$} spanned by the following four vertices\footnote{\,The picture of $R_{p,\mspace{1mu}i}$ should be compared with the diagram in \exaref{A-rep}. Notice that in the picture of $R_{p,\mspace{1mu}i}$ the vertex $(p,\mspace{1mu}i)$ is located higher than $\Serre(p,\mspace{1mu}i)$, which of course is only the case if $p$ is smaller than \smash{$\frac{n+1}{2}$}. For $p=1$ the ``rectangle'' $R_{p,\mspace{1mu}i}$ is actually the line (with negative slope) from $(1,\mspace{1mu}i)$ to $\Serre(1,\mspace{1mu}i) = (n,\mspace{1mu}i)$, and for $p=n$ it is the line (with positive slope) from $(n,\mspace{1mu}i)$ to $\Serre(n,\mspace{1mu}i) = (1,\mspace{1mu}i+1-n)$.}, where \mbox{$\Serre(p,i) = (n+1-p,\,i+1-p)$}:
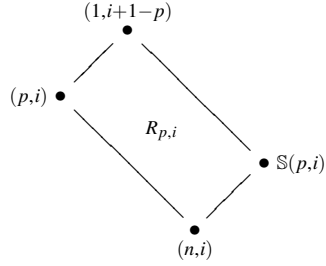
\begin{equation*}
  \xymatrix@!=0.1pc{
    {} & \stackrel{(1,\mspace{1mu}i+1-p)}{\vtx} 
    \ar@{}[dddr]|-{R_{p,\mspace{1mu}i}} \ar@{-}[ddrr] & {} & {}
    \\
    \mspace{-30mu}\text{$\scriptstyle(p,\mspace{1mu}i)$} \ \vtx \ar@{-}[ur] \ar@{-}[ddrr] & {} & {} & {}
    \\
    {} & {} & {} & \vtx \ \text{$\scriptstyle \Serre(p,\mspace{1mu}i)$} \mspace{-40mu}
    \\
    {} & {} & \underset{(n,\mspace{1mu}i)}{\vtx} \ar@{-}[ur] & {} 
  }
\end{equation*}

  \item The Serre functor is given by \mbox{$\Serre(q,i) = (n+1-q,\,i+1-q)$} on objects, and on morphisms it is determined by $\Serre(a^{}_{q,\mspace{1mu}i}) = (-1)^q \mspace{1mu} a^*_{n-q,\mspace{1mu}i+1-q}$ and $\Serre(a^*_{q,\mspace{1mu}i}) = (-1)^{n-q} \mspace{1mu} a^{}_{n-q,\mspace{1mu}i-q}$.
  
  \end{prt}
  Moreover, the arrow ideal $\mathfrak{r}$ (which by \lemref{arrow-ideal} is the pseudo-radical in $Q$) is nilpotent; in fact, one has $\mathfrak{r}^n=0$. 
\end{thm}  

\begin{proof}
  Similar to, but easier than, the proof of \thmref{A-dou}.
\end{proof}

\begin{cor}
  \label{cor:A-rep}
  Let $\alg$ be any ring. Consider the category of $\lMod{\alg}$-valued repre\-sen\-tations of the repetitive quiver of \smash{$\vec{\mathbb{A}}_n$} that satisfy the mesh relations. This category has two different abelian model structure where the trivial objects and the weak equivalences can be~cha\-rac\-terized by the (co)homology functors $\cH[i]{q}$ and $\hH[i]{q}$ as in \thmref[Theorems~]{E-characterization-hereditary} and \thmref[]{quiso}.
\end{cor}
 
\begin{proof}
  Set $\Bbbk=\mathbb{Z}$. As noted in \rmkref{rep}, the category in question is nothing but $\lMod{Q,\alg}$ where \smash{$Q = \mesh{\rep{(\vec{\mathbb{A}}_n)}}$}. By \thmref{A-rep} this 
  $Q$ satisfies conditions \eqref{Homfin}, \eqref{locbd}, \eqref{Serre} in \stpref{Bbbk} and condition \eqmref{strong-retraction} in \dfnref{srp} with a nilpotent pseudo-radical. The assertion therefore follows directly from \thmref[Theorems~]{model-structures}, \thmref[]{E-characterization-hereditary},  and \thmref[]{quiso}.
\end{proof}

\begin{exa}
  \label{exa:derived-cat}
  The category of $\lMod{\alg}$-valued represen\-tations of the repetitive quiver of~\smash{$\vec{\mathbb{A}}_2$} that satisfy the mesh relations is nothing but the category  $\operatorname{Ch}\mspace{1mu}\alg$ of chain complexes of left $\alg$-modules. The model structures on this category provided by \corref{A-rep} are classic and well-known, see e.g.~Hovey \cite[Thms.~2.3.11 and 2.3.13]{modcat}, and the associated homotopy category from \dfnref{homotopy-cat} is the usual derived category $\mathcal{D}(\alg)$.
\end{exa}

In the context of stable translation quivers, it is possible to give very concrete descriptions of the (co)homology functors $\cH{q}$ and $\hH{q}$. This is our goal for the rest of this~section. 

For representations of a stable translations quiver that satisfy the mesh relations, there is a natural notion of homology:

\begin{dfn}
  \label{dfn:mH}
  Let $\upGamma$ be a stable translation quiver and set $Q = \mesh{\upGamma}$.
  Consider the mesh \eqref{mesh} associated with $q \in \upGamma_0$. Since one has $a_1 \,\str(a_1) + \cdots + a_n \,\str(a_n)=0$ in $Q$, every $X \in \lMod{Q,\alg}$  induces a three term complex of left $\alg$-modules,
\begin{equation*}
  \xymatrix@C=6pc{
    X(\tr(q)) \ar[r]^-{
      \left(
      \begin{smallmatrix}
        X(\str(a_1)) \vspace{-0.6ex} \\ \text{\raisebox{2.5ex}{$\vdots$}}\vspace{-1ex} \\ X(\str(a_n)) 
      \end{smallmatrix}
      \right)
    }
    &
    \mspace{-8mu}
    {
    \begin{array}{c}
      X(p_1)
      \\
      \oplus
      \\
      \vdots
      \\
      \oplus
      \\
      X(p_n)
    \end{array}
    }
    \mspace{-8mu}
    \ar[r]^-{
      \big(\!
      \begin{smallmatrix}
        X(a_1) & \cdots & X(a_n) 
      \end{smallmatrix}
      \!\big)    
    } 
    &
    X(q)
  }.
\end{equation*}
We write $\mH{q}(X)$ for the homology of this three term complex and call it the \emph{mesh homology} of $X$ at $q$.
\end{dfn}

\begin{dfn}
  \label{dfn:normal}
  Let $\upGamma$ be a stable translation quiver and set $Q = \mesh{\upGamma}$. We say that $\upGamma$ is \emph{normal \textnormal{(}relative to $\Bbbk$\textnormal{)}} if one has $\mH{q}(Q(p,-)) = 0$ for all $p,q \in \upGamma_0$ (equivalently, every projective object in $\lMod{Q}$ has vanishing mesh homology).
\end{dfn}

Notice that the definition of normality depends on the base ring $\Bbbk$. As far as we know, most stable translation quivers are normal. Here we just note the following two results.

\begin{thm}
  \label{thm:A-dou-normal}
  The stable translation quiver $\dou{(\vec{\mathbb{A}}_n)}$ from \exaref{A-dou} is normal.
\end{thm}

\begin{proof}
  See \appref{proof}.
\end{proof}

\begin{thm}
  \label{thm:A-rep-normal}
  The stable translation quiver $\rep{(\vec{\mathbb{A}}_n)}$ from \exaref{A-rep} is normal.
\end{thm}

\begin{proof}
  Similar to, but easier than, the proof of \thmref{A-dou-normal}.
\end{proof}

The next result and \rmkref{also-cH} compare the mesh homology $\mH{*}$ defined above with the first (co)homology functors $\cH{*}$ and $\hH{*}$ from 
\dfnref{cH-hH}.

\begin{prp}
  \label{prp:2H}
  Let $\upGamma$ be a stable translation quiver with mesh category $Q = \mesh{\upGamma}$. If\, $\upGamma$ is normal, then for every $X \in \lMod{Q,\alg}$ and $q \in \upGamma_0$ there is a natural isomorphism,
  \begin{equation*}
    \hH{q}(X) \,\cong\, \mH{q}(X)\;.
  \end{equation*}  
\end{prp}

\begin{proof}
  By \lemref{arrow-ideal} the Strong Retraction Property holds for $Q$ and the arrow ideal $\mathfrak{r}$ serves as the pseudo-radical. In particular it makes sense to consider $\stalkcn{q} \in \rMod{Q}$ from
\dfnref{stalk}. Consider the mesh \eqref{mesh} associated with the vertex $q$. We claim that there is an exact sequence in $\rMod{Q}$,
\begin{equation*}
  \xymatrix@C=1.75pc{
    Q(-,\tr(q)) \ar[rr]^-{
      \left(
      \begin{smallmatrix}
        Q(-,\,\str(a_1)) \vspace{-0.6ex} \\ \text{\raisebox{2.5ex}{$\vdots$}}\vspace{-1ex} \\ Q(-,\,\str(a_n)) 
      \end{smallmatrix}
      \right)
    }
    & &
    \mspace{-3mu}
    {
    \begin{array}{c}
      Q(-,p_1)
      \\
      \oplus
      \\
      \vdots
      \\
      \oplus
      \\
      Q(-,p_n)
    \end{array}
    }
    \mspace{-8mu}
    \ar[rrr]^-{
      \big(\!
      \begin{smallmatrix}
        Q(-,\,a_1) & \cdots & Q(-,\,a_n) 
      \end{smallmatrix}
      \!\big)    
    } 
    & & &
    Q(-,q) \ar[r] & \stalkcn{q} \ar[r] & 0
  }.
\end{equation*}
Indeed, exactness at $Q(-,p_1) \oplus \cdots \oplus Q(-,p_n)$ follows as $\mH{q}(Q(r,-)) = 0$ for all $r \in \upGamma_0$, as $\upGamma$ is assumed to be normal. To prove exactness at $Q(-,q)$ and at $\stalkcn{q}$ it suffices, by the definition of $\stalkcn{q}$ to show that one has
\begin{equation*}
  \Im{\big(Q(-,a_1) \cdots  Q(-,a_n)\big)} \,=\, \mathfrak{r}(-,q)\;.
\end{equation*}
However, this equality follows from the fact that $\mathfrak{r}$ is the arrow ideal and $a_1,\ldots,a_n$ is the complete list of arrows in $\upGamma$ with target $q$. Consequently, the sequence displayed above, which we write as $P_2 \to P_1 \to P_0 \to \stalkcn{q} \to 0$, is part of an augmented projective resolution of $\stalkcn{q}$ in $\rMod{Q}$. It follows that the left $\alg$-module \smash{$\hH{q}(X) = \Tor{Q}{1}{\stalkcn{q}}{X}$} can be computed as the homology of the three term complex $\ORt{P_2}{X} \to \ORt{P_1}{X} \to \ORt{P_0}{X}$. For every object $r \in Q$ one has $\ORt{Q(-,r)}{X} \cong X(r)$, see the proof of \corref{adjoint-evaluation}, and hence the three term complex $\ORt{P_2}{X} \to \ORt{P_1}{X} \to \ORt{P_0}{X}$ is isomorphic to the three term complex in \dfnref{mH}, whose homology is $\mH{q}(X)$ by definition.
\end{proof}

\begin{rmk}
  \label{rmk:also-cH}
  In \dfnref{mH} we could evidently also have defined the mesh homology $\mH{q}(Y)$ for $Y \in \bMod{\alg}{Q}$, that is, for contravariant $\Bbbk$-linear functors $Y \colon Q \to \lMod{\alg}$. If one extends the definition, \dfnref[]{normal}, of normality to mean that \smash{$\mH{q}(Q(p,-)) = 0 = \mH{q}(Q(-,p))$} for all $p,q \in \upGamma_0$, then one can also easily extend \prpref[Pro\-po\-si\-tion~]{2H} (and its proof) to get isomorphisms $\hH{q}(X) \cong \mH{q}(X) \cong \cH{\tr(q)}(X)$ for every $X \in \lMod{Q,\alg}$. The stable translation quivers \smash{$\dou{(\vec{\mathbb{A}}_n)}$} and \smash{$\rep{(\vec{\mathbb{A}}_n)}$} are also normal in this stronger sense, but this is not important.
\end{rmk}

By \thmref[Theorems~]{A-dou}, \thmref[]{A-rep}, \thmref[]{A-dou-normal}, and \thmref[]{A-rep-normal}, the next result applies e.g.~to the stable translation quivers \smash{$\dou{(\vec{\mathbb{A}}_n)}$} and \smash{$\rep{(\vec{\mathbb{A}}_n)}$}.

\begin{cor}
  \label{cor:E-mesh}
  Let $\upGamma$ be a stable translation quiver with mesh category $Q = \mesh{\upGamma}$ over any commutative ring $\Bbbk$ and assume that the following conditions hold.
 \begin{itemlist}
  \item $Q$ satisfies conditions \eqref{Homfin}, \eqref{locbd}, \eqref{Serre} in \stpref{Bbbk}.
 \item The arrow ideal $\mathfrak{r}$ is nilpotent, that is, $\mathfrak{r}^N=0$ for some $N \in \mathbb{N}$.
 \item The ring $\Bbbk$ is noetherian and hereditary (e.g.~$\Bbbk=\mathbb{Z}$).
  \item $\upGamma$ is normal (\dfnref{normal}).
 \end{itemlist}
 For any $\Bbbk$-algebra $\alg$, the class $\mathscr{E}$ of exact objects in $\lMod{Q,\alg}$ from \dfnref{E} satisfies:
 \begin{equation*}
   \mathscr{E} \,=\, \{\, X \in \lMod{Q,\alg} \ |\ \mH{q}(X)=0 \text{ for every $q \in \upGamma_0$}\, \}\;.
 \end{equation*}
\end{cor}

\begin{proof}
The hypotheses in \thmref{E-characterization-hereditary} are satisfied by  \lemref{arrow-ideal} and the assumptions made in the present result, and hence $\mathscr{E} = \{\, X \in \lMod{Q,\alg} \ |\ \hH{q}(X)=0 \text{ for every $q \in \upGamma_0$}\, \}$. The assertion now follows from \prpref{2H} as $\upGamma$ is assumed to be normal.
\end{proof}

\begin{exa}
  \label{exa:counter}
  Let $Q$ be the category \smash{$\mesh{\rep{(\vec{\mathbb{A}}_3)}}$}. The category $\lMod{Q}$ is the same as the category of $\lMod{\Bbbk}$-valued representations of the quiver \smash{$\rep{(\vec{\mathbb{A}}_3)}$} that satisfy the mesh relations; see \rmkref{rep}. Moreover, in $\lMod{Q}$ the abstract homology functors $\hH{*}$ agree with the mesh homology functors $\mH{*}$ by \thmref{A-rep-normal} and \prpref{2H}. 
  
  We give an elementary example of a morphism $\varphi \colon X \to Y$ in $\lMod{Q}$ such that $\mH{*}(\varphi)$ is an isomorphism but $\varphi$ is \emph{not} a weak equivalence in the projective/injective model structure. Notice that for the category $Q$ one has $\mathfrak{r}^3=0$ but $\mathfrak{r}^2 \neq 0$; thus our example shows that the conditions \eqclbl{i}--\eqclbl{iii} in \prpref{quiso-2} are, in general, not equivalent unless $\mathfrak{r}^2=0$.
  
Now, let $X$ and $Y$ be the representations:
  \begin{equation*}
    \xymatrix@!=0.5pc{
      0    
      \ar[dr]
      &
      {}
      &
      \Bbbk_1
      \ar[dr]^-{=}
      &
      {}
      &
      0_4
      \ar[dr]
      &
      {}
      &
      0
      \\
      \mspace{-50mu}\cdots
      &
      0
      \ar[dr]
      \ar[ur]
      &
      {}
      &
      \Bbbk_3
      \ar[dr]
      \ar[ur]
      &
      {}
      &
      0
      \ar[dr]
      \ar[ur]
      &
      \mspace{60mu}\cdots
      \\
      0
      \ar[ur]
      &
      {}
      &
      \Bbbk_2
      \ar[ur]_-{=}
      &
      {}
      &
      0
      \ar[ur]
      &
      {}
      &
      0   
    }
  \end{equation*}  
and
  \begin{equation*}
    \xymatrix@!=0.5pc{
      0    
      \ar[dr]
      &
      {}
      &
      \Bbbk_1
      \ar[dr]
      &
      {}
      &
      0_4
      \ar[dr]
      &
      {}
      &
      0
      \\
      \mspace{-50mu}\cdots
      &
      0
      \ar[dr]
      \ar[ur]
      &
      {}
      &
      0_3
      \ar[dr]
      \ar[ur]
      &
      {}
      &
      0
      \ar[dr]
      \ar[ur]
      &
      \mspace{60mu}\cdots
      \\
      0
      \ar[ur]
      &
      {}
      &
      0_2
      \ar[ur]
      &
      {}
      &
      0
      \ar[ur]
      &
      {}
      &
      0   
    }
  \end{equation*}  
where the subscripts $1\,\ldots,4$ just refer to four specific vertices in \smash{$\rep{(\vec{\mathbb{A}}_3)}$} with those labels. Let $\varphi$ be the morphism from $X$ to $Y$ where $\varphi(1) \colon X(1) \to Y(1)$ is the identity on $\Bbbk$ and $\varphi(q)=0$ for all vertices $q \neq 1$. Notice that $\mH{q}(X)=0=\mH{q}(Y)$ for all $q \neq 3$ and 
\begin{equation*}
  \mH{3}(X) \,=\, \Ker{\!\big(\Bbbk \oplus \Bbbk \xrightarrow{(1 \ \ 1)} \Bbbk\big)} 
  \qquad \text{and} \qquad
  \mH{3}(Y) \,=\, \Ker{(\Bbbk \longrightarrow 0)}\;.
\end{equation*}
The map $\mH{3}(\varphi)$ sends $(x,-x)$ to $x$, so it is clearly an isomorphism. 

To see that $\varphi$ is not a weak equivalence, notice that $\varphi$ is an epimorphism and that its kernel satisfies $\mH{4}(\Ker{\varphi}) = \Bbbk \neq 0$. Thus, $\Ker{\varphi}$ does not belong to $\mathscr{E}$, see \corref{E-mesh}, and hence $\varphi$ can not be a weak equivalence by \cite[Lem.~5.8]{Hovey02}.
\end{exa}

\appendix

\section{Purity and Kaplansky Classes}
\label{app:Kaplansky}

In this section, $\mathcal{M}$ denotes a Grothendieck category which is \emph{locally finitely presentable} in the sense of Crawley-Boevey \cite[\S1]{WCB94} or Ad{\'a}mek and Rosick{\'y} \cite[\mbox{Dfn.~1.17 with $\lambda=\aleph_0$}]{AdamekRosicky}.

\begin{bfhpg}[Purity]
  \label{purity}
A short exact sequence $0 \to M' \to M \to M'' \to 0$ in $\mathcal{M}$ is called \emph{pure exact} if it stays exact under the functor $\Hom{\mathcal{M}}{K}{-}$ for every finitely presentable object $K$ in $\mathcal{M}$.

The definition above can be found in \cite[\S3]{WCB94}. The more general concept of \emph{$\lambda$-pure exact} sequences has been studied by Ad{\'a}mek and Rosick{\'y} in \cite{AR04} and \cite[Chap.~2.D]{AdamekRosicky} (the situation above corresponds to $\lambda = \aleph_0$).
\end{bfhpg}

To parse the next definition recall from \cite[Dfns.~1.13 and 1.67]{AdamekRosicky} the definitions of \emph{$\kappa$-pre\-sentable} and \emph{$\kappa$-generated} objects, where $\kappa$ is any regular cardinal.

\begin{bfhpg}[Kaplansky classes]
A class $\mathcal{F}$ of objects in $\mathcal{M}$ is said to be
\emph{$\kappa$-Kaplansky} if for every $F \in \mathcal{F}$ and every $\kappa$-generated subobject
$X \subseteq F$ there exists a $\kappa$-presentable sub\-object $Y$ of $F$ such that $X \subseteq Y \subseteq F$ and $Y,\, F/Y \in \mathcal{F}$. One says that $\mathcal{F}$ is a \emph{Kaplansky class} if it is $\kappa$-Kaplansky for some regular cardinal $\kappa$. 

In this generality the definition is due to Gillespie \cite[Dfn.~4.9]{MR2342555}; see also {\v{S}}{\v{t}}ov{\'{\i}}{\v{c}}ek \cite[Dfn.~2.6]{Sto2013a}. Kaplansky classes were introduced and studied by Enochs and L{\'o}pez-Ramos \cite{EEnJLR02} in the special case where $\mathcal{M}$ is the category of (left) modules over a ring.
\end{bfhpg}

In the special case where $\mathcal{M}$ is the category of (left) modules over a ring, the next result follows from \cite[Thm.~3.4]{HJ08}. Note that the assumption in \cite[Thm.~3.4]{HJ08} that $\mathcal{F}$ should be closed under extensions is superfluous, indeed, it follows from the proof of \mbox{\thmref{pure-sub-quo}} below that under the hypotheses in that theorem, the class $\mathcal{F}$ is \emph{deconstructible}, and such a class is automatically closed under extensions by {{\v{S}}{\v{t}}ov{\'{\i}}{\v{c}}ek \cite[Lem.~1.6]{Sto2013a}. 

\begin{thm}
  \label{thm:pure-sub-quo}
  Let $\mathcal{F}$ be a class of objects in $\mathcal{M}$ that satisfies the following conditions: 
\begin{rqm}
\item $\mathcal{F}$ is closed under pure subobjects and pure quotients, meaning that for every pure exact sequence $0 \to M' \to M \to M'' \to 0$ in $\mathcal{M}$ with $M \in \mathcal{F}$, one has $M',M'' \in \mathcal{F}$.

\item $\mathcal{F}$ contains a generator of $\mathcal{M}$ and is closed under coproducts in $\mathcal{M}$.
\end{rqm}  
Then $(\mathcal{F},\mathcal{F}^\perp)$ is a complete cotorsion pair. In fact, every object in $\mathcal{M}$ even has an $\mathcal{F}$-cover.
\end{thm}

\begin{proof}
 We start by showing that condition \rqmlbl{1} implies that $\mathcal{F}$ is a Kaplansky class.
 
 By assumtion, $\mathcal{M}$ is locally finitely presentable (as $\mathcal{M}$ is also cocomplete, this is the same as being \emph{finitely accessible}, see \cite[Exa.~2.3(1)]{AdamekRosicky}), so we may apply \cite[Thm.~2.33 and its Remark]{AdamekRosicky} with $\lambda=\aleph_0$. The conclusion provided by this result holds for arbitrary large regular cardinals $\gamma$ \emph{sharply greater} than $\lambda=\aleph_0$ (in symbols: $\gamma \vartriangleright \aleph_0$, see   \cite[Dfn.~2.12]{AdamekRosicky})---which just means that $\gamma$ is an uncountable regular cardinal, see \cite[Exa.~2.13(1)]{AdamekRosicky}---but for our purpose it suffices to know that the conclusion holds for some regular cardinal $\gamma$. We let $\gamma$ be any such cardinal and we will show that $\mathcal{F}$ is a $\gamma$-Kaplansky class.
  
    Let $F \in \mathcal{F}$ and $X \subseteq F$ be a $\gamma$-generated subobject. Since $\mathcal{M}$ is also locally $\gamma$-presentable, see \cite[Remark after Thm.~1.20]{AdamekRosicky}, there is by \cite[Prop.~1.69(ii)]{AdamekRosicky} an epimorphism \mbox{$f \colon K \twoheadrightarrow X$} where $K$ is $\gamma$-presentable (this also follows from \cite[Lem.~A.3(1)]{Sto2013a} and \cite[Prop.~1.16]{AdamekRosicky}). Applying \cite[Thm.~2.33 and its Remark]{AdamekRosicky} to the composite morphism $K \twoheadrightarrow X \rightarrowtail F$, we get a factorization
\begin{equation*}
  \xymatrix@!=1pc{
     K \ar@{->>}[r]^-{f} 
     \ar@{.>}[dr]_-{g}
     & 
     X \ar@{>->}[r] 
     & 
     F
     \\   
     {}
     &
     \bar{K}
     \ar@{>.>}[ur]_-{\bar{f}}
     &
     {}
  }
\end{equation*}    
where $\bar{K}$ is $\gamma$-presentable and $\bar{f}$ is a $\aleph_0$-pure (mono)morphism in the sense of \cite[Dfn.~2.27 (see also Prop.~2.29)]{AdamekRosicky}. By \cite[Prop.~5(b) and Dfn.~1]{AR04} this means that that the exact sequence
\begin{equation*}
  \xymatrix@C=1.5pc{0 \ar[r] & \bar{K} \ar[r]^-{\bar{f}} & F \ar[r] & \Coker{\bar{f}} \ar[r] & 0}
\end{equation*}
is pure exact in the sense of \ref{purity}. If we set $Y=\Im{\bar{f}}$, then the sequence above is isomorphic to $0 \to Y \to F \to F/Y \to 0$, which is therefore also pure exact. As $F \in \mathcal{F}$ we get $Y,\, F/Y \in \mathcal{F}$ because  $\mathcal{F}$ is closed under pure subobjects and pure quotients. Finally, note that $Y \cong \bar{K}$ is $\gamma$-presentable and that one has $X \subseteq Y$, indeed, $X = \Im{f} = \Im{(\bar{f} \circ g)} \subseteq \Im{\bar{f}} = Y$.

 To finish the proof, note that for any $I$-direct system $\{M_i \to M_j\}$ in $\mathcal{M}$,  the canonical~map \smash{$\bigoplus_{i \in I} M_i \twoheadrightarrow \varinjlim_{i \in I}M_i$} is a pure epimorphism. As $\mathcal{F}$ is assumed to be closed~under pure quotients and coproducts, it follows that $\mathcal{F}$ is closed under direct limits as well. We have seen that $\mathcal{F}$ is a Kaplansky class; thus {\v{S}}{\v{t}}ov{\'{\i}}{\v{c}}ek \cite[Cor.~2.7(2)]{Sto2013a} implies that $\mathcal{F}$ is \emph{deconstructible}. As every split exact sequence is pure exact, $\mathcal{F}$ is also closed under direct summands. By assumption, $\mathcal{F}$ contains a generator, and thus {\v{S}}{\v{t}}ov{\'{\i}}{\v{c}}ek \cite[Cor.~5.17]{Sto2013b} yields that $(\mathcal{F},\mathcal{F}^\perp)$ is a (functorially) complete cotorsion pair. 
 
As noted above, $\mathcal{F}$ is closed under direct limits. As the cotorsion pair $(\mathcal{F},\mathcal{F}^\perp)$ is complete, every object in $\mathcal{M}$ has a (special) $\mathcal{F}$-precover. Thus, El Bashir \cite[Thm.~1.2]{ElBashir} shows that every object in $\mathcal{M}$ has an $\mathcal{F}$-cover. 
\end{proof}

\section{Proofs of \thmref[Theorems~]{A-dou} and \thmref[]{A-dou-normal}}
\label{app:proof}

Recall that we consider the double quiver $\dou{(\vec{\mathbb{A}}_n)}$ from \exaref{A-dou}. For the path and mesh categories of this stable translation quiver we write:
\begin{equation*}
  P \,=\, \Bbbk \dou{(\vec{\mathbb{A}}_n)}
  \qquad \text{and} \qquad
  Q \,=\, \mesh{\dou{(\vec{\mathbb{A}}_n)}}\;.
\end{equation*}
For vertices $p, q$ in \smash{$\dou{(\vec{\mathbb{A}}_n)}$} and $\ell \geqslant 0$ denote by $P^\ell(p,q)$ and $Q^\ell(p,q)$ the $\Bbbk$-submodules of $P(p,q)$ and $Q(p,q)$ generated by all paths in \smash{$\dou{(\vec{\mathbb{A}}_n)}$} of length $\ell$; here we include the trivial paths $\mathrm{id}_q$, which have length \mbox{$\ell=0$}. As we have seen in the proof of \lemref{arrow-ideal}, there is a direct sum decomposition \eqref{Q-decomp}.

  To keep track of all the possible paths in the quiver \smash{$\dou{(\vec{\mathbb{A}}_n)}$}, it is convenient to draw multiple copies of the vertices and arrows in \smash{$\dou{(\vec{\mathbb{A}}_n)}$} as follows (for e.g.~$n=5$):
  
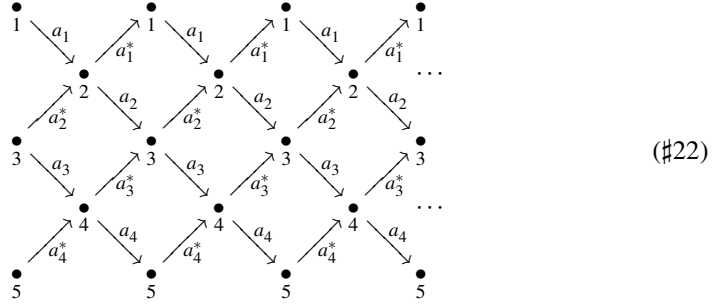
\begin{equation}
    \label{eq:all-paths-2}
    \begin{gathered}
    \xymatrix@!=0.1pc{
      \underset{1}{\vtx} \ar[dr]^(0.65){\mspace{-15mu}a^{}_1} & {} & 
      \underset{1}{\vtx} \ar[dr]^(0.65){\mspace{-15mu}a^{}_1} & {} & 
      \underset{1}{\vtx} \ar[dr]^(0.65){\mspace{-15mu}a^{}_1} & {} & 
      \underset{1}{\vtx}
      \\
      {} & 
      \underset{2}{\vtx} 
      \ar[ur]_(0.6){\mspace{-15mu}a^*_1} 
      \ar[dr]^(0.65){\mspace{-15mu}a^{}_2} & 
      {} & 
      \underset{2}{\vtx} 
      \ar[ur]_(0.6){\mspace{-15mu}a^*_1} 
      \ar[dr]^(0.65){\mspace{-15mu}a^{}_2} & 
      {} & 
      \underset{2}{\vtx} 
      \ar[ur]_(0.6){\mspace{-15mu}a^*_1} 
      \ar[dr]^(0.65){\mspace{-15mu}a^{}_2} & 
      \mspace{12mu} \cdots
      \\
      \underset{3}{\vtx} 
      \ar[ur]_(0.6){\mspace{-15mu}a^*_2}
      \ar[dr]^(0.65){\mspace{-15mu}a^{}_3} & 
      {} & 
      \underset{3}{\vtx} 
      \ar[ur]_(0.6){\mspace{-15mu}a^*_2}
      \ar[dr]^(0.65){\mspace{-15mu}a^{}_3} & 
      {} & 
      \underset{3}{\vtx} 
      \ar[ur]_(0.6){\mspace{-15mu}a^*_2}
      \ar[dr]^(0.65){\mspace{-15mu}a^{}_3} & 
      {} &
      \underset{3}{\vtx}
      \\
      {} & 
      \underset{4}{\vtx} 
      \ar[ur]_(0.6){\mspace{-15mu}a^*_3}
      \ar[dr]^(0.65){\mspace{-15mu}a^{}_4} & 
      {} & 
      \underset{4}{\vtx} 
      \ar[ur]_(0.6){\mspace{-15mu}a^*_3}
      \ar[dr]^(0.65){\mspace{-15mu}a^{}_4} & 
      {} & 
      \underset{4}{\vtx} 
      \ar[ur]_(0.6){\mspace{-15mu}a^*_3}
      \ar[dr]^(0.65){\mspace{-15mu}a^{}_4} & 
      \mspace{12mu} \cdots
      \\
      \underset{5}{\vtx} 
      \ar[ur]_(0.6){\mspace{-15mu}a^*_4} & 
      & 
      \underset{5}{\vtx} 
      \ar[ur]_(0.6){\mspace{-15mu}a^*_4} & 
      &       
      \underset{5}{\vtx} 
      \ar[ur]_(0.6){\mspace{-15mu}a^*_4} & 
      &       
      \underset{5}{\vtx}
    }
    \end{gathered}    
  \end{equation}  
  From this diagram, it is evident that one has:
  \begin{equation}
    \label{eq:P-nonzero}
    P^\ell(p,q) \neq 0 \quad \iff \quad 
    \ell = |p-q|+2t \ \text{ where } \ t \in \mathbb{N}_0\;.
  \end{equation}    
  We now take into account the mesh relations, that is, we consider \eqref{all-paths-2} as a diagram in $Q$.
\begin{itemlist}

\item[($\circ$)] The mesh relations $\mu_1$ and $\mu_n$ mean that $a^*_1a^{}_1=0$ and $a^{}_{n-1}a^*_{n-1}=0$ hold in $Q$.

\item[($\scriptstyle\Diamond$)] The mesh relations $\mu_2, \ldots, \mu_{n-1}$ mean that every square \,\text{\tiny $
\renewcommand{\arraystretch}{0.0}
\setlength{\arraycolsep}{0pt}
\begin{array}{ccccc}
{} & {} & \bullet & {} & {}
\\
{} & \nearrow & {} & \searrow & {}
\\
\bullet & {} & {} & {} & \bullet 
\\
{} & \searrow & {} & \nearrow & {}
\\
{} & {} & \bullet & {} & {}
\end{array}
$}\, is anticommutative.
\end{itemlist}  
Given vertices $p,q$ and $\ell \geqslant 0$, it follows from ($\scriptstyle\Diamond$) and \eqref{all-paths-2} that in $Q$, all paths from $p$ to $q$ of length $\ell$ (of course, there might not exist such a path) are equal up to a sign. Thus:
\begin{equation}
  \label{eq:1-dim}
  \text{If $Q^\ell(p,q) \neq 0$, then $Q^\ell(p,q)$ is a free $\Bbbk$-module of dimension $1$.}
\end{equation}
If $Q^\ell(p,q) \neq 0$, then a basis for the $1$-dimensional free $\Bbbk$-module $Q^\ell(p,q)$ is, for example, $\{ \pi\}$ or $\{-\pi\}$ where $\pi$ is any (possibly trivial) path in \smash{$\dou{(\vec{\mathbb{A}}_n)}$} from $p$ to $q$ of length $\ell$.

Given vertices $p$ and $q$ we now seek to determine which numbers $\ell$ satisfy $Q^\ell(p,q) \neq 0$. 

\begin{prp}
  \label{prp:dim-2}
  With $d(p,q) = \min\{p,\,q,\,n+1-p,\,n+1-q\} \in \mathbb{N}$ one has:
  \begin{equation*}
    Q^\ell(p,q) \neq 0 \quad \iff \quad 
    \ell = |p-q|+2t \ \text{ where } \ 0 \leqslant t < d(p,q)\;.
  \end{equation*}  
\end{prp}

\begin{proof}
Certainly a necessary condition for $Q^\ell(p,q) \neq 0$ is that $P^\ell(p,q) \neq 0$, which by~\eqref{P-nonzero} means that $\ell$ has the form $\ell = |p-q|+2t$ for some $t \in \mathbb{N}_0$. If $t \geqslant \min\{p,q\}$, then there exists a path from $p$ to $q$ of length $\ell = |p-q|+2t$ that contains at least one copy of the product $a^*_1a^{}_1$, namely the following path with $t-\min\{p,q\}+1>0$ copies of $a^*_1a^{}_1$:
\begin{equation*}  
  (a^{}_{q-1}\cdots a^{}_1)(a^*_1a^{}_1)\cdots(a^*_1a^{}_1)(a^*_1\cdots a^*_{p-1})\;.
\end{equation*} 
Here we have used that $(q-1)+2(t-\min\{p,q\}+1)+(p-1)$ is equal to $\ell = |p-q|+2t$. In this case, one has $Q^\ell(p,q)=0$ by ($\circ$) above. Conversely, if $t < \min\{p,q\}$, then there is no path from $p$ to $q$ of length $\ell = |p-q|+2t$ that contains a copy of the product $a^*_1a^{}_1$.

Similarly, if $t \geqslant \min\{n+1-p,n+1-q\}$ then there exists a path from $p$ to $q$ of length $\ell = |p-q|+2t$ that contains at least one copy of the product $a^{}_{n-1}a^*_{n-1}$, namely the following path with $t-\min\{n+1-p,n+1-q\}+1>0$ copies of $a^{}_{n-1}a^*_{n-1}$:
\begin{equation*}
  (a^*_q\cdots a^*_{n-1})(a^{}_{n-1}a^*_{n-1})\cdots(a^{}_{n-1}a^*_{n-1})(a^{}_{n-1}\cdots a^{}_p)\;.
\end{equation*}
Here we have used that $(n-q)+2(t-\min\{n+1-p,n+1-q\}+1)+(n-p)$ is equal to $\ell=|p-q|+2t$. In this case, one has $Q^\ell(p,q)=0$ by ($\circ$) above. Conversely, if one has $t < \min\{n+1-p,n+1-q\}$, then there is no path from $p$ to $q$ of length $\ell = |p-q|+2t$ that contains a copy of the product $a^{}_{n-1}a^*_{n-1}$.

Thus, if both $t < \min\{p,q\}$ and $t < \min\{n+1-p,n+1-q\}$, equivalently, $t<d(p,q)$, then no path from $p$ to $q$ of length $\ell = |p-q|+2t$ contains  $a^*_1a^{}_1$ or $a^{}_{n-1}a^*_{n-1}$, and hence one has \smash{$Q^\ell(p,q) \neq 0$}.
\end{proof}

\begin{cor}
  \label{cor:dimension}
  $Q(p,q)$ is a free $\Bbbk$-module of dimension $d(p,q) \in \mathbb{N}$ for every $p,q \in Q$.
\end{cor}

\begin{proof}
  This follows immediately by combining the direct sum decomposition \eqref{Q-decomp} with the assertion \eqref{1-dim} and \prpref{dim-2}.
\end{proof}

\begin{cor}
  \label{cor:l-geq-n}
  For every $p,q \in Q$ there is a strict inequality, $|p-q|+2(d(p,q)-1)<n$. In particular, one has $Q^\ell(p,q)=0$ for all $\ell \geqslant n$.
\end{cor}

\begin{proof}
  The strict inequality is not hard to prove, and the last assertion now follows from \prpref{dim-2}.
\end{proof}

\begin{cor}
  \label{cor:nonzero-symmetry}
  Let $p,q \in Q$ and set $\Serre(p)=n+1-p$. For every $0 \leqslant \ell < n$ one has 
  \begin{equation*}
     Q^\ell(p,q) \neq 0
     \quad \iff \quad
     Q^{n-1-\ell}(q,\Serre(p)) \neq 0\;.
  \end{equation*}   
\end{cor}

\begin{proof} 
  It follows from the definitions that $d(q,\Serre(p)) = d(p,q)$; denote this number by $\delta$. It is not hard to prove the identity:
  \begin{equation}
    \label{eq:n-plus-1}
    |p-q|+|q-\Serre(p)|+2\delta \,=\, n+1\;.
  \end{equation}
  To prove the equivalence we need by \prpref{dim-2} to argue that one has $\ell = |p-q|+2t$ for some $0 \leqslant t < \delta$ if and only if $n-1-\ell = |q-\Serre(p)|+2s$ for some $0 \leqslant s < \delta$. If $\ell$ has the form $\ell = |p-q|+2t$ for some $0 \leqslant t < \delta$, then $s=\delta-t-1$ satisfies $0 \leqslant s < \delta$ and it follows from \eqref{n-plus-1} that $n-1-\ell = |q-\Serre(p)|+2s$. Conversely, if $n-1-\ell = |q-\Serre(p)|+2s$ for some $0 \leqslant s < \delta$, then $t=\delta-s-1$ satisfies $0 \leqslant t<\delta$ and $\ell = |p-q|+2t$.
\end{proof}

\begin{rmk}
  \label{rmk:negative-ell}
  By \corref{l-geq-n} one has $Q^\ell(p,q)=0$ for every $\ell \geqslant n$. Hence, if we define $Q^\ell(p,q):=0$ for all $\ell<0$, then the equivalence in   \corref{nonzero-symmetry} holds for all $\ell \in \mathbb{Z}$.
\end{rmk}

There is a (kind of) canonical way to choose a basis for the $1$-dimensional free $\Bbbk$-module $Q^\ell(p,q)$ in the case where it is non-zero, cf.~\eqref{1-dim}. Namely, consider the diagram obtained from \eqref{all-paths-2} by replacing all occurences of $a^{}_q$ with $(-1)^q\mspace{1mu}a^{}_q$. For $n=5$ it looks like this:
  \begin{equation}
    \label{eq:all-paths-2-minus}
    \begin{gathered}
    \xymatrix@!=0.1pc{
      \underset{1}{\vtx} \ar[dr]^(0.65){\mspace{-15mu}-a^{}_1} & {} & 
      \underset{1}{\vtx} \ar[dr]^(0.65){\mspace{-15mu}-a^{}_1} & {} & 
      \underset{1}{\vtx} \ar[dr]^(0.65){\mspace{-15mu}-a^{}_1} & {} & 
      \underset{1}{\vtx}
      \\
      {} & 
      \underset{2}{\vtx} 
      \ar[ur]_(0.6){\mspace{-15mu}a^*_1} 
      \ar[dr]^(0.65){\mspace{-15mu}a^{}_2} & 
      {} & 
      \underset{2}{\vtx} 
      \ar[ur]_(0.6){\mspace{-15mu}a^*_1} 
      \ar[dr]^(0.65){\mspace{-15mu}a^{}_2} & 
      {} & 
      \underset{2}{\vtx} 
      \ar[ur]_(0.6){\mspace{-15mu}a^*_1} 
      \ar[dr]^(0.65){\mspace{-15mu}a^{}_2} & 
      \mspace{12mu} \cdots
      \\
      \underset{3}{\vtx} 
      \ar[ur]_(0.6){\mspace{-15mu}a^*_2}
      \ar[dr]^(0.65){\mspace{-15mu}-a^{}_3} & 
      {} & 
      \underset{3}{\vtx} 
      \ar[ur]_(0.6){\mspace{-15mu}a^*_2}
      \ar[dr]^(0.65){\mspace{-15mu}-a^{}_3} & 
      {} & 
      \underset{3}{\vtx} 
      \ar[ur]_(0.6){\mspace{-15mu}a^*_2}
      \ar[dr]^(0.65){\mspace{-15mu}-a^{}_3} & 
      {} &
      \underset{3}{\vtx}
      \\
      {} & 
      \underset{4}{\vtx} 
      \ar[ur]_(0.6){\mspace{-15mu}a^*_3}
      \ar[dr]^(0.65){\mspace{-15mu}a^{}_4} & 
      {} & 
      \underset{4}{\vtx} 
      \ar[ur]_(0.6){\mspace{-15mu}a^*_3}
      \ar[dr]^(0.65){\mspace{-15mu}a^{}_4} & 
      {} & 
      \underset{4}{\vtx} 
      \ar[ur]_(0.6){\mspace{-15mu}a^*_3}
      \ar[dr]^(0.65){\mspace{-15mu}a^{}_4} & 
      \mspace{12mu} \cdots
      \\
      \underset{5}{\vtx} 
      \ar[ur]_(0.6){\mspace{-15mu}a^*_4} & 
      & 
      \underset{5}{\vtx} 
      \ar[ur]_(0.6){\mspace{-15mu}a^*_4} & 
      &       
      \underset{5}{\vtx} 
      \ar[ur]_(0.6){\mspace{-15mu}a^*_4} & 
      &       
      \underset{5}{\vtx}
    }
    \end{gathered}    
  \end{equation}  
Note that by the mesh relations this diagram is now commutative(!) in $Q$. 

\begin{dfn}
  \label{dfn:isoelt}
For every combination of $p$, $q$, and $\ell$ where $Q^\ell(p,q) \neq 0$ (see   \prpref{dim-2} for a precise criterion), we let \smash{$\isoelt{\ell}{p}{q}$} be the be the unique \emph{signed path} (by which we just mean a morphism in $Q$ of the form $\pm$ a path in the quiver \smash{$\dou{(\vec{\mathbb{A}}_n)}$}) from $p$ to $q$ of length $\ell$ determined by the diagram \eqref{all-paths-2-minus}. When $\ell=0$, and hence $p=q$, we set \smash{$\isoelt{0}{q}{q} = \mathrm{id}_q$}.

As noted after \eqref{1-dim}, the singleton set \smash{$\{\isoelt{\ell}{p}{q}\}$} is a basis of $Q^\ell(p,q) \neq 0$.
\end{dfn}

For example, one has \smash{$\isoelt{0}{1}{1} = \mathrm{id}_1$},\, \smash{$\isoelt{1}{1}{2} = -a^{}_1$},\, \smash{$\isoelt{2}{1}{3} = -a^{}_2a^{}_1$}, and \smash{$\isoelt{3}{1}{4} = a^{}_3a^{}_2a^{}_1$}.  

\begin{rmk}
  \label{rmk:a-mult}
  It is clear from \dfnref{isoelt} above that multiplication with $a^{}_q$ and $a^*_q$ on the basis elements \smash{$\isoelt{\ell}{\text{\LARGE\raisebox{0.5pt}{$.$}}}{\text{\LARGE\raisebox{0.5pt}{$.$}}}$} acts as follows:
\begin{prt}
\item $a^{}_q\,\isoelt{\ell}{p}{q}= (-1)^q\,\isoelt{\ell+1}{p}{q+1}$ \ provided that \ $Q^\ell(p,q) \neq 0$ and $Q^{\ell+1}(p,q+1) \neq 0$.

\item $\isoelt{\ell}{q+1}{r}\,a^{}_q = (-1)^q\,\isoelt{\ell+1}{q}{r}$ \ provided that \ $Q^\ell(q+1,r) \neq 0$ and $Q^{\ell+1}(q,r) \neq 0$.

\item $a^*_q\,\isoelt{\ell}{p}{q+1}= \isoelt{\ell+1}{p}{q}$ \ provided that \ $Q^\ell(p,q+1) \neq 0$ and $Q^{\ell+1}(p,q) \neq 0$.

\item $\isoelt{\ell}{q}{r}\,a^*_q = \isoelt{\ell+1}{q+1}{r}$ \ provided that \ $Q^\ell(q,r) \neq 0$ and $Q^{\ell+1}(q+1,r) \neq 0$.
\end{prt}
A more refined version of part \prtlbl{a} can be found in \lemref{matrix}. Of course, it is also possible to make similar refined versions of \prtlbl{b}, \prtlbl{c}, and \prtlbl{d}.
\end{rmk}

\begin{proof}[Proof of \thmref{A-dou}.]
  We know from \corref{dimension} that $Q$ satisfies condition \eqref{Homfin} in \stpref[Setup ]{Bbbk} (Hom-finiteness), indeed, it even satisfies the stronger condition mentioned in \thmref[]{A-dou}\prtlbl{a}.
  
  It is evident that $Q$ satisfies condition \eqref{locbd} in \stpref{Bbbk} (Local Boundedness) as $Q$ only has finitely many objects. 
  
  It is known from \lemref{arrow-ideal} that $Q$ satisfies
  condition \eqmref{strong-retraction} in \dfnref{srp} (the Strong Retraction Property) and that the arrow ideal $\mathfrak{r}$ serves as the pseudo-radical in $Q$. That one has $\mathfrak{r}^n=0$, as asserted last in \thmref{A-dou}, is immediate from \corref{l-geq-n}. 
  
  It remains to argue that $Q$ satisfies condition \eqref{Serre} in \stpref{Bbbk}, i.e.~that $Q$ has a Serre~func\-tor, which is given by the formulae in \thmref[]{A-dou}\prtlbl{b}. The arguments take up the rest of the proof.
  
  By the universal property of the path category $P$, we can well-define a unique $\Bbbk$-linear endo\-functor $\Serre \colon P \to P$ by setting $\Serre(q) = n+1-q$ for $1 \leqslant q \leqslant n$ and \smash{$\Serre(a^{}_q) = (-1)^q \mspace{1mu} a^*_{n-q}$} and \smash{$\Serre(a^*_q) = (-1)^{n-q} \mspace{1mu} a^{}_{n-q}$} for $1 \leqslant q <n$; the same formulae as in \thmref[]{A-dou}\prtlbl{b}. This functor is an involution, i.e.~an automorphism with $\Serre^{-1} = \Serre$, as one has $\Serre(\Serre(q)) = q$ and $\Serre(\Serre(a^{}_q)) = a^{}_q$ and $\Serre(\Serre(a^*_q)) = a^*_q$. For the mesh relations $\mu_1,\ldots,\mu_n$, an easy computation shows that the identity $\Serre(\mu_q) = (-1)^n \mspace{1mu} \mu_{n+1-q}$ holds for all $1 \leqslant q \leqslant n$. Hence $\Serre$ preserves the mesh ideal $\mathfrak{m} = \langle \mu_1,\ldots,\mu_n\rangle$, so by the universal property of the quotient (mesh) category $Q = P/\mathfrak{m}$, it follows that $\Serre$ induces a $\Bbbk$-linear automorphism $\Serre \colon Q \to Q$.
  
It remains to prove that this automorphism $\Serre$ satisfies the defining property of a Serre functor, that is, we must establish an isomorphism of $\Bbbk$-modules,
\begin{equation}
  \label{eq:isomap}
\isomap{}{p}{q} \colon Q(p,q) \stackrel{\cong}{\longrightarrow} \Hom{\Bbbk}{Q(q,\Serre(p))}{\Bbbk}\;,
\end{equation}
which is natural in $p,q \in Q$. It follows from \corref{l-geq-n} that the direct sum in \eqref{Q-decomp} is finite, in fact, one has
\begin{equation*}
   Q(p,q) \,=\, Q^0(p,q) \oplus \cdots \oplus Q^{n-1}(p,q)
\end{equation*}  
(of course, some of these direct summands are zero by \prpref{dim-2}), and hence also
\begin{equation*}
   \Hom{\Bbbk}{Q^{}(q,\Serre(p))}{\Bbbk} \,=\,
   \Hom{\Bbbk}{Q^{n-1}(q,\Serre(p))}{\Bbbk} \oplus \cdots \oplus
   \Hom{\Bbbk}{Q^{0}(q,\Serre(p))}{\Bbbk}\;.
\end{equation*}  
Thus, to construct an isomorphism \smash{$\isomap{}{p}{q}$} as in \eqref{isomap}, it suffices to construct an isomorphism
\begin{equation}
  \label{eq:isomap-ell}
\isomap{\ell}{p}{q} \colon Q^\ell(p,q) \stackrel{\cong}{\longrightarrow} \Hom{\Bbbk}{Q^{n-1-\ell}(q,\Serre(p))}{\Bbbk}
\end{equation}
for every $\ell \in \mathbb{Z}$; here $Q^\ell(p,q)=0$ for $\ell<0$ as in \rmkref{negative-ell}. Indeed, having constructed isomorphisms $\isomap{\ell}{p}{q}$ for $\ell \in \mathbb{Z}$ we simply define $\isomap{}{p}{q} = \bigoplus_{\ell \in \mathbb{Z}}\, \isomap{\ell}{p}{q} = \bigoplus_{\ell=0}^{n-1}\,\isomap{\ell}{p}{q}$\,. 

To define $\isomap{\ell}{p}{q}$ note that the $\Bbbk$-modules
$Q^\ell(p,q)$ and $Q^{n-1-\ell}(q,\Serre(p))$ are simultaneously zero by     \corref{nonzero-symmetry} and \rmkref{negative-ell}, and in all such cases we set \smash{$\isomap{\ell}{p}{q} = 0$}.

For combinations of $p$, $q$, and $\ell$ where $Q^\ell(p,q) \neq 0$, and hence also $Q^{n-1-\ell}(q,\Serre(p)) \neq 0$, both $Q^\ell(p,q)$ and $\Hom{\Bbbk}{Q^{n-1-\ell}(q,\Serre(p))}{\Bbbk}$ are $1$-dimensional free $\Bbbk$-modules by \eqref{1-dim}, and hence they are, at least, non-canonically isomorphic. However, to obtain a \textsl{natural}~isomorphism, we have to be more careful:

In the situation where $Q^\ell(p,q) \neq 0$ we have already chosen a basis \smash{$\{\isoelt{\ell}{p}{q}\}$} for this module, see \dfnref{isoelt}. As a basis of $\Hom{\Bbbk}{Q^\ell(p,q)}{\Bbbk}$ we now take the dual of this basis, that is, \smash{$\{\isoeltd{\ell}{p}{q}\}$} where \smash{$\isoeltd{\ell}{p}{q} \colon Q^\ell(p,q) \to \Bbbk$} is the $\Bbbk$-linear map given by \smash{$\isoelt{\ell}{p}{q} \mapsto 1$}. Now let
\begin{equation*}
  \isomap{\ell}{p}{q} 
  \quad \text{be the $\Bbbk$-isomorphism given by} \quad
  \isoelt{\ell}{p}{q} \longmapsto \isoeltd{n-1-\ell}{q}{\Serre(p)}\;.
\end{equation*}

It remains to prove that the hereby obtained isomorphism \smash{$\isomap{}{p}{q} = \bigoplus_{\ell} \isomap{\ell}{p}{q}$}, see \eqref{isomap}, is natural in $p$ and $q$. Naturality in the variable $q$ means that for every morphism $\beta \colon q' \to q''$ in $Q$, the following diagram should be commutative,
\begin{equation*}
  \begin{gathered}
  \xymatrix@C=3pc{
    Q(p,q') 
    \ar[r]^-{\isomap{}{p}{q'}}_-{} 
    \ar[d]_-{\beta \,\circ\, -} 
    & 
    \Hom{\Bbbk}{Q(q',\Serre(p))}{\Bbbk}
    \ar[d]^-{- \,\circ\, Q(\beta,\,\Serre(p))}
    \\
    Q(p,q'') \ar[r]^-{\isomap{}{p}{q''}} & 
    \Hom{\Bbbk}{Q(q'',\Serre(p))}{\Bbbk}\;.\mspace{-8mu}
  }
  \end{gathered}
\end{equation*}
Evidently, it is enough to check this in the case where $\beta=a^{}_q$ or $\beta=a^*_q$ for some $1 \leqslant q<n$. We only consider the first case, as the second case can be dealt with similarly. By definition, \smash{$\isomap{}{p}{q}$} is the direct sum \smash{$\bigoplus_{\ell} \isomap{\ell}{p}{q}$}, so it suffices to argue that for every $\ell$, the diagram
\begin{equation}
  \label{eq:beta2}
  \begin{gathered}
  \xymatrix@C=3pc{
    Q^\ell(p,q) 
    \ar[r]^-{\isomap{\ell}{p}{q}} 
    \ar[d]_-{a^{}_q \,\circ\, -} 
    & 
    \Hom{\Bbbk}{Q^{n-1-\ell}(q,\Serre(p))}{\Bbbk}
    \ar[d]^-{- \,\circ\, Q(a^{}_q,\,\Serre(p))}
    \\
    Q^{\ell+1}(p,q+1) \ar[r]^-{\isomap{\ell+1}{p}{q+1}} & 
    \Hom{\Bbbk}{Q^{n-2-\ell}(q+1,\Serre(p))}{\Bbbk}
  }
  \end{gathered}
\end{equation}
is commutative. As already mentioned (see \corref{nonzero-symmetry} and \rmkref{negative-ell}), the domain and codomain of \smash{$\isomap{\ell}{p}{q}$} are simultaneously zero, and so are the domain and codomain of \smash{$\isomap{\ell+1}{p}{q+1}$}. Thus we may assume that all four modules that appear in the diagram above are non-zero (otherwise the diagram is trivially commutative), in which case the basis elements 
\begin{equation*}
 \isoelt{\ell}{p}{q}\,, \ \isoelt{n-1-\ell}{q}{\Serre(p)}\,, \ \isoeltd{n-1-\ell}{q}{\Serre(p)} \qquad \text{and} \qquad \isoelt{\ell+1}{p}{q+1}\,, \ \isoelt{n-2-\ell}{q+1}{\Serre(p)}\,, \ \isoeltd{n-2-\ell}{q+1}{\Serre(p)}
\end{equation*} 
exist. To see that the diagram \eqref{beta2} is commutative, it must be shown that the maps
\begin{equation*}
  \isomap{\ell}{p}{q}(\isoelt{\ell}{p}{q}) \circ Q(a^{}_q,\,\Serre(p))
  \qquad \text{and} \qquad
  \isomap{\ell+1}{p}{q+1}(a^{}_q\,\isoelt{\ell}{p}{q})
\end{equation*}
are identical. Using the definition of \smash{$\isomap{\ell}{p}{q}$} and \rmkref{a-mult}\prtlbl{a}, we get
\begin{align*}
  \isomap{\ell}{p}{q}(\isoelt{\ell}{p}{q}) \circ Q(a^{}_q,\,\Serre(p))
  &\,=\,
  \isoeltd{n-1-\ell}{q}{\Serre(p)} \circ Q(a^{}_q,\,\Serre(p)) \qquad \text{and}
  \\
  \isomap{\ell+1}{p}{q+1}(a^{}_q\,\isoelt{\ell}{p}{q})
  &\,=\, (-1)^q\, \isomap{\ell+1}{p}{q+1}(\isoelt{\ell+1}{p}{q+1})
  \,=\, (-1)^q\, \isoeltd{n-2-\ell}{q+1}{\Serre(p)}\;.
\end{align*}
To prove that these two $\Bbbk$-linear maps $Q^{n-2-\ell}(q+1,\Serre(p)) \to \Bbbk$ are identical, it suffices to see that they agree on \smash{$\isoelt{n-2-\ell}{q+1}{\Serre(p)}$}. By \rmkref{a-mult}\prtlbl{b} and the definition of dual bases, we get for the first map above:
\begin{align*}
  \big(\isoeltd{n-1-\ell}{q}{\Serre(p)} \circ Q(a^{}_q,\,\Serre(p))\big)(\isoelt{n-2-\ell}{q+1}{\Serre(p)}) 
  &\,=\, 
  \isoeltd{n-1-\ell}{q}{\Serre(p)}(\isoelt{n-2-\ell}{q+1}{\Serre(p)}\,a^{}_q) 
  \\
  &\,=\, 
  (-1)^q\,\isoeltd{n-1-\ell}{q}{\Serre(p)}(\isoelt{n-1-\ell}{q}{\Serre(p)})
  \,=\, (-1)^q\;.
\end{align*}
And for the second map we obviously also have $(-1)^q\, \isoeltd{n-2-\ell}{q+1}{\Serre(p)}(\isoelt{n-2-\ell}{q+1}{\Serre(p)}) = (-1)^q$. 

These arguments show that the isomorphism \smash{$\isomap{}{p}{q}$} is natural in $q$. A similar argument shows that it is natural in $p$, and consequently the proof is concluded.
\end{proof}

\begin{dfn}
  \label{dfn:T-matrices}
  For every $p,q \in Q$ we consider the (ordered) set
\begin{equation*}
  B_{p,q} \,=\, \big\{\isoelt{|p-q|}{p}{q},\isoelt{|p-q|+2}{p}{q},\isoelt{|p-q|+4}{p}{q},\ldots,\isoelt{|p-q|+2(d(p,q)-1)}{p}{q}\big\}\;,
\end{equation*}  
which is a basis of the $d(p,q)$-dimensional free $\Bbbk$-module $Q(p,q)$; see the direct sum decomposition \eqref{Q-decomp}, \prpref{dim-2}, and \dfnref{isoelt}. 

For every $1 \leqslant p \leqslant n$ and $1 \leqslant q < n$ we let $\mat{p}{q}$ and $\mat[*]{p}{q}$ be the matrices given by
\begin{equation*}
  \begin{gathered}
  \xymatrix@C=3pc{
    Q(p,q) \ar[d]_-{\cong} \ar[r]^-{a^{}_q\,\circ\,-} & Q(p,q+1) \ar[d]^-{\cong}
    \\
    \Bbbk^{d(p,q)} \ar[r]^-{\mat{p}{q}} & \Bbbk^{d(p,q+1)}
  }
  \end{gathered}
  \qquad \text{and} \qquad
  \begin{gathered}
  \xymatrix@C=3pc{
    Q(p,q+1) \ar[d]_-{\cong} \ar[r]^-{a^*_q\,\circ\,-} & Q(p,q) \ar[d]^-{\cong}
    \\
    \Bbbk^{d(p,q+1)} \ar[r]^-{\mat[*]{p}{q}} & \Bbbk^{d(p,q)}\;,\mspace{-8mu}
  }
  \end{gathered}
\end{equation*}
where the vertical isomorphisms are induced by the bases $B_{p,q}$ and $B_{p,q+1}$.
Here we view elements in $\Bbbk^m$ as a column vectors, so the matrices $\mat{p}{q}$ and $\mat[*]{p}{q}$ act from the left and have sizes $d(p,q+1) \times d(p,q)$ and $d(p,q) \times d(p,q+1)$, respectively. 
\end{dfn}

Our next goal is to find explicit descriptions of the matrices $\mat{p}{q}$ and $\mat[*]{p}{q}$.

\begin{lem}
  \label{lem:matrix}
  Let $1 \leqslant p \leqslant n$ and $1 \leqslant q < n$ be given. For every $0 \leqslant t < d(p,q)$ set 
\begin{equation*}
  \kfct{p}{q}(t) \,=\, 
  \left\{\mspace{-5mu}
  \begin{array}{ccl}
    t & \text{if} & p \leqslant q
    \\
    t+1 & \text{if} & p > q\;.
  \end{array}
  \right.
\end{equation*} 
The following formula holds:
\begin{equation*}
  a^{}_{q}\,\isoelt{|p-q|+2t}{p}{q} \,=\, 
  \left\{\mspace{-5mu}
  \begin{array}{ccl}
    (-1)^q\,\isoelt{|p-(q+1)|+2\mspace{1mu}\kfct{p}{q}(t)}{p}{q+1} & \text{if} & \kfct{p}{q}(t)<d(p,q+1)
    \vspace{0.5ex} \\
    0 & \multicolumn{2}{l}{\text{otherwise}\;.} 
  \end{array}
  \right.
\end{equation*}
\end{lem}

\begin{proof}
  As \smash{$a^{}_q\,\isoelt{|p-q|+2t}{p}{q}$} is a (signed) path from $p$ to $q+1$ of length $\ell =|p-q| + 2t+1$ we know from \prpref{dim-2} that it is non-zero if and only if $\ell$ has the form $|p-(q+1)|+2s$ for some $0 \leqslant s < d(p,q+1)$. Note that the equation $|p-q| + 2t+1 = \ell = |p-(q+1)|+2s$ implies that $s$ is equal to $\kfct{p}{q}(t)$, so the desired conclusion follows from \dfnref{isoelt} (see also   \rmkref{a-mult}).
\end{proof}

In the next result, $I_m$ denotes the $m \times m$ identity matrix; for $m=0$ it is the empty matrix.

\begin{prp}
  \label{prp:matrix}
  For every $1 \leqslant p \leqslant n$ and $1 \leqslant q < n$ the following assertions hold.
  \begin{prt}
  \item If $p+q<n+1$ and $p \leqslant q$, then $\mat{p}{q}$ is the following $p \times p$ matrix:
  \begin{equation*}
    \mat{p}{q} \,=\, (-1)^q\cdot I_{p}\;.
  \end{equation*}

  \item If $p+q<n+1$ and $p > q$, then $\mat{p}{q}$ is the following $(q+1) \times q$ matrix:
  \begin{equation*}
    \mat{p}{q} \,=\, (-1)^q\cdot
    \left(\mspace{-4mu}
    \begin{array}{c}
      0 \\ \hline
      I_q
    \end{array}
    \mspace{-4mu}\right).
  \end{equation*}

  \item If $p+q \geqslant n+1$ and $p \leqslant q$, then $\mat{p}{q}$ is the following $(n-q) \times (n+1-q)$ matrix:
  \begin{equation*}
    \mat{p}{q} \,=\, (-1)^q\cdot
    \big(\, I_{n-q} \,\big|\; 0 \,\big)\;.
  \end{equation*}

  \item If $p+q \geqslant n+1$ and $p > q$, then $\mat{p}{q}$ is the following $(n+1-p) \times (n+1-p)$ matrix:
  \begin{equation*}
    \mat{p}{q} \,=\, (-1)^q\cdot
    \left(\mspace{-4mu}
    \begin{array}{c|c}
      0 & 0 
      \\ \hline
      I_{n-p} & 0
    \end{array}
    \mspace{-4mu}\right).
  \end{equation*}  
  \end{prt}
\end{prp}

\begin{proof}
  We only prove parts \prtlbl{b} and \prtlbl{c} as the remaining two parts are proved similarly. Recall from \prpref{dim-2} that the function $d$ is given by
 that 
  \begin{align*}
    d(p,q) &\,=\, \min\{p,q,n+1-p,n+1-q\} \quad \text{and hence}
    \\
    d(p,q+1) &\,=\, \min\{p,q+1,n+1-p,n-q\}\;.
  \end{align*}
     
    \proofoftag{b} Clearly one has $d(p,q)=q$ and $d(p,q+1)=q+1$ under the given assumptions on $p$ and $q$. In this case, $\kfct{p}{q}(t)=t+1$ and \lemref{matrix} implies that 
\begin{equation*}
  a^{}_{q}\,\isoelt{|p-q|+2t}{p}{q} \,=\, (-1)^q\,\isoelt{|p-(q+1)|+2(t+1)}{p}{q+1}
  \quad \text{for every} \quad  0 \leqslant t<q\;.
\end{equation*}     
This shows that the matrix $\mat{p}{q}$ has the asserted form.

    \proofoftag{c} Clearly one has $d(p,q)=n+1-q$ and $d(p,q+1)=n-q$ under the given assumptions on $p$ and $q$. In this case, $\kfct{p}{q}(t)=t$ and \lemref{matrix} implies that 
\begin{equation*}
  a^{}_{q}\,\isoelt{|p-q|+2t}{p}{q} \,=\, 
  \left\{\mspace{-5mu}
  \begin{array}{ccl}
    (-1)^q\,\isoelt{|p-(q+1)|+2t}{p}{q+1} & \text{if} & 0 \leqslant t<n-q
    \vspace{0.5ex} \\
    0 & \text{if} & t=n-q\;.
  \end{array}
  \right.
\end{equation*}
This shows that the matrix $\mat{p}{q}$ has the asserted form.
\end{proof}

\begin{prp}
  \label{prp:matrix-star}
  For every $1 \leqslant p \leqslant n$ and $1 \leqslant q < n$ the following assertions hold.
  \begin{prt}
  \item If $p+q<n+1$ and $p \leqslant q$, then $\mat[*]{p}{q}$ is the following $p \times p$ matrix:
  \begin{equation*}
    \mat[*]{p}{q} \,=\, 
    \left(\mspace{-4mu}
    \begin{array}{c|c}
      0 & 0 
      \\ \hline
      I_{p-1} & 0
    \end{array}
    \mspace{-4mu}\right).
  \end{equation*}  

  \item If $p+q<n+1$ and $p > q$, then $\mat[*]{p}{q}$ is the following $q \times (q+1)$ matrix:
  \begin{equation*}
    \mat[*]{p}{q} \,=\, \big(\, I_{q} \,\big|\; 0 \,\big)\;.
  \end{equation*}

  \item If $p+q \geqslant n+1$ and $p \leqslant q$, then $\mat[*]{p}{q}$ is the following $(n+1-q) \times (n-q)$ matrix:
  \begin{equation*}
    \mat[*]{p}{q} \,=\, 
    \left(\mspace{-4mu}
    \begin{array}{c}
      0 \\ \hline
      I_{n-q}
    \end{array}
    \mspace{-4mu}\right).
  \end{equation*}

  \item If $p+q \geqslant n+1$ and $p > q$, then $\mat[*]{p}{q}$ is the following $(n+1-p) \times (n+1-p)$ matrix:
  \begin{equation*}
    \mat[*]{p}{q} \,=\, I_{n+1-p}\;.
  \end{equation*}
  \end{prt}
\end{prp}

\begin{proof}
  Similar to the proof of \prpref{matrix}.
\end{proof}

\begin{proof}[Proof of \thmref{A-dou-normal}.]
  Let $1 \leqslant p,q \leqslant n$ be given. We must show that $\mH{q}(Q(p,-))=0$. The proof is divided into three different cases: $q=1$, ($1<q<n$), and $q=n$. We start with the case $q=1$; the case $q=n$ is handled similarly and therefore left to the reader.
  
The mesh \eqref{mesh} associated to $q=1$ is: 
  \begin{equation*}
    \xymatrix{
      1 \ar[r]^-{a^{}_1} & 2 \ar[r]^-{a^*_1} & 1\;.
    }
  \end{equation*}
  It must be shown that the sequence 
  \begin{equation*}
    \xymatrix@C=3pc{
      Q(p,1) \ar[r]^-{a^{}_1 \,\circ\,-} & Q(p,2) \ar[r]^-{a^*_1 \,\circ\,-} & Q(p,1)
    }
  \end{equation*}
  is exact. By \dfnref{T-matrices} this sequence is isomorphic to
  \begin{equation*}
    \xymatrix@C=3pc{
      \Bbbk \ar[r]^-{\mat{p}{1}} & \Bbbk^{d(p,2)} \ar[r]^-{\mat[*]{p}{1}} & \Bbbk
    }.
  \end{equation*}
  For $p=1$, ($1<p<n$), and $p=n$, we get from parts \prtlbl{a}, \prtlbl{b}, and \prtlbl{d} in \prpref[Propositions~]{matrix} and \prpref[]{matrix-star} that this sequence is:
  \begin{equation*}
    \xymatrix{
      \Bbbk \ar[r]^-{-1} & \Bbbk \ar[r]^-{0} & \Bbbk
    }, 
    \quad
    \xymatrix{
      \Bbbk \ar[r]^-{
      \left(
        \begin{smallmatrix}
          0 \\ 
          -1
        \end{smallmatrix}
        \right)
      }
    & 
    \Bbbk^2
      \ar[r]^-{(1 \ \ 0)} 
    & 
    \Bbbk
  },   
  \quad \text{and} \quad
    \xymatrix{
      \Bbbk \ar[r]^-{0} & \Bbbk \ar[r]^-{1} & \Bbbk
    },
  \end{equation*}
  so evidently the sequence is exact in all three cases.
  
  It remains to consider the mesh at $1<q<n$, which is:
  \begin{equation*}
    \xymatrix@!=0.5pc{
      {} & q-1 \ar[dr]^(0.6){\mspace{-10mu}a^{}_{q-1}} & {}
      \\
      q
      \ar[ur]^(0.4){a^*_{q-1}\mspace{-10mu}}
      \ar[dr]_(0.4){a^{}_{q}\mspace{-10mu}}
      & {} & q
      \\
      {} & q+1 \ar[ur]_(0.6){\mspace{-5mu}a^*_{q}} & {}
    }
  \end{equation*}
It must be shown that the sequence
\begin{equation*}
  \xymatrix@C=6pc{
    Q(p,q) \ar[r]^-{
      \left(
      \begin{smallmatrix}
        a^*_{q-1} \,\circ\,- \\ 
        a^{}_{q} \,\circ\,-
      \end{smallmatrix}
      \right)
    }
    & 
    \mspace{-5mu}
    {
    \begin{array}{c}
      Q(p,q-1)
      \\
      \oplus
      \\
      Q(p,q+1)
    \end{array}
    }
    \mspace{-8mu}
    \ar[r]^-{
      \big(\!
      \begin{smallmatrix}
        a^{}_{q-1} \,\circ\,- & & a^*_{q} \,\circ\,-
      \end{smallmatrix}
      \!\big)    
    } 
    & 
    Q(p,q)
  }
\end{equation*}
is exact. By \dfnref{T-matrices} this sequence is isomorphic to
\begin{equation}
  \label{eq:3-term-complex}
  \xymatrix@C=7pc{
    \Bbbk^{d(p,q)} \ar[r]^-{M_{p,q} \ := \
      \left(\mspace{-1mu}
      \begin{smallmatrix}
        \mat[*]{p}{q-1} \vspace{1pt} \\ \hline \\
        \mat{p}{q}
      \end{smallmatrix}
      \mspace{-1mu}\right)
    }
    & 
    \Bbbk^{d(p,q-1)\,+\,d(p,q+1)}
    \ar[r]^-{N_{p,q} \ := \
      \big(
        \mat{p}{q-1} \big| \mat[*]{p}{q}
      \big)    
    } 
    & 
    \Bbbk^{d(p,q)}
  }.
\end{equation}
There are now four cases to check: \prtlbl{a}, \prtlbl{b}, \prtlbl{c}, and \prtlbl{d}, corresponding to the four cases in \prpref[Propositions~]{matrix} and \prpref[]{matrix-star}. We only consider the first case as the remaining three cases are handled similarly.

Thus, assume that the pair $(p,q)$ satisfies $p+q<n+1$ and $p \leqslant q$. In this case, part \prtlbl{a} in \prpref[Propositions~]{matrix} and \prpref[]{matrix-star} yields expressions for the matrices $\mat{p}{q}$ and $\mat[*]{p}{q}$. If $p<q$ then the pair $(p,q-1)$ satisfies $p+(q-1)<n+1$ and $p \leqslant q-1$, however, if $p=q$ then $p+(q-1)<n+1$ and $p > q-1$. Thus depending on the situation $p<q$ or $p=q$ we can use either part \prtlbl{a} or \prtlbl{b} in \prpref[Propositions~]{matrix} and \prpref[]{matrix-star} to find expressions for the matrices $\mat{p}{q-1}$ and $\mat[*]{p}{q-1}$. Explicitly, if $p<q$, then the block matrices in \eqref{3-term-complex} are
\begin{equation*}
  M_{p,q} \,=\,
  \left(\mspace{-3mu}
   \setlength{\arraycolsep}{1pt}  
   \begin{array}{c|c}
     0 & 0 \\ \hline
     I_{p-1} & 0 \\ \hline
     (-1)^q I_{p-1} & 0 \\ \hline
     0 & (-1)^q
   \end{array}
  \mspace{-3mu}\right)  
  \quad \text{and} \quad
  N_{p,q} \,=\,
  \left(\mspace{-3mu}
   \setlength{\arraycolsep}{2pt}
   \begin{array}{c|c|c|c}
     (-1)^{q-1} & 0 & 0 & \,0\, \\ \hline
     0 & (-1)^{q-1} I_{p-1} & I_{p-1} & \,0\, 
   \end{array}
  \mspace{-3mu}\right)
\end{equation*}
of sizes $2p \times p$ and $p \times 2p$, and if $p=q$ they are
\begin{equation*}
  M_{q,q} \,=\,
  \left(\mspace{-3mu}
   \setlength{\arraycolsep}{1pt}  
   \begin{array}{c|c}
     I_{q-1} & 0 \\ \hline
     (-1)^q I_{q-1} & 0 \\ \hline
     0 & (-1)^q
   \end{array}
  \mspace{-3mu}\right)  
  \quad \text{and} \quad
  N_{q,q} \,=\,
  \left(\mspace{-3mu}
   \setlength{\arraycolsep}{2pt}
   \begin{array}{c|c|c}
     0 & 0 & \,0\, \\ \hline
     (-1)^{q-1} I_{q-1} & I_{q-1} & \,0\, 
   \end{array}
  \mspace{-3mu}\right)
\end{equation*}
of sizes $(2q-1) \times q$ and $q \times (2q-1)$. In both cases, the sequence \eqref{3-term-complex} is clearly exact.
\end{proof}

\section*{Acknowledgement}

The second author was supported by a DNRF Chair from the Danish National Research Foundation (grant no. DNRF156) and by Aarhus University Research Foundation (grant no. AUFF-F-2020-7-16).

\def\cprime{$'$} \def\soft#1{\leavevmode\setbox0=\hbox{h}\dimen7=\ht0\advance
  \dimen7 by-1ex\relax\if t#1\relax\rlap{\raise.6\dimen7
  \hbox{\kern.3ex\char'47}}#1\relax\else\if T#1\relax
  \rlap{\raise.5\dimen7\hbox{\kern1.3ex\char'47}}#1\relax \else\if
  d#1\relax\rlap{\raise.5\dimen7\hbox{\kern.9ex \char'47}}#1\relax\else\if
  D#1\relax\rlap{\raise.5\dimen7 \hbox{\kern1.4ex\char'47}}#1\relax\else\if
  l#1\relax \rlap{\raise.5\dimen7\hbox{\kern.4ex\char'47}}#1\relax \else\if
  L#1\relax\rlap{\raise.5\dimen7\hbox{\kern.7ex
  \char'47}}#1\relax\else\message{accent \string\soft \space #1 not
  defined!}#1\relax\fi\fi\fi\fi\fi\fi} \def\cprime{$'$}
  \providecommand{\arxiv}[2][AC]{\mbox{\href{http://arxiv.org/abs/#2}{\tt
  arXiv:#2 [math.#1]}}}
  \providecommand{\oldarxiv}[2][AC]{\mbox{\href{http://arxiv.org/abs/math/#2}{\sf
  arXiv:math/#2
  [math.#1]}}}\providecommand{\MR}[1]{\mbox{\href{http://www.ams.org/mathscinet-getitem?mr=#1}{#1}}}
  \renewcommand{\MR}[1]{\mbox{\href{http://www.ams.org/mathscinet-getitem?mr=#1}{#1}}}
\providecommand{\bysame}{\leavevmode\hbox to3em{\hrulefill}\thinspace}
\providecommand{\MR}{\relax\ifhmode\unskip\space\fi MR }
\providecommand{\MRhref}[2]{%
  \href{http://www.ams.org/mathscinet-getitem?mr=#1}{#2}
}
\providecommand{\href}[2]{#2}


\end{document}